\documentclass[a4paper,11pt,reqno]{amsart}
\usepackage{graphicx}
\usepackage{subfigure}
\usepackage{amsmath}
\usepackage[T1]{fontenc}
\usepackage{amssymb}
\usepackage{amsthm}
\usepackage{color}
\usepackage[lmargin=2.5 cm,rmargin=2.5 cm,tmargin=3.5cm,bmargin=2.5cm,
paper=a4paper]{geometry}

\usepackage[english]{babel}

\newcommand{\Ab}{\mathbf A}
\newcommand{\Fb}{\mathbf F}

\newcommand{\R}{\mathbb R}
\newcommand{\Z}{\mathbb Z}
\newcommand{\N}{\mathbb N}
\newcommand{\C}{\mathbb C}

\DeclareMathOperator{\E0}{E_{\rm g}}
\DeclareMathOperator{\RE}{Re}
\DeclareMathOperator{\IM}{Im}
\DeclareMathOperator{\curl}{curl}
\DeclareMathOperator{\Div}{div}

\DeclareMathOperator{\dist}{dist}
\DeclareMathOperator{\supp}{supp} 
\newtheorem{thm}{Theorem}[section]
\newtheorem{prop}[thm]{Proposition}
\newtheorem{lem}[thm]{Lemma}
\newtheorem{corol}[thm]{Corollary}
\newtheorem{theorem}[thm]{Theorem}
\newtheorem{definition}[thm]{Definition}
\newtheorem{lemma}[thm]{Lemma}
\newtheorem{ass}[thm]{Assumption}

\newtheorem{app}[thm]{Application}
\theoremstyle{remark}
\newtheorem{rem}[thm]{Remark}

\numberwithin{equation}{section}
\title[2D Ginzburg-Landau  functional]{Pinning with a variable magnetic field of the two dimensional Ginzburg-Landau model}
\author[K.Attar]{}
\author[]{K. Attar}
\begin{document}
\begin{abstract}
We study the Ginzburg-Landau  energy of a superconductor with a
variable magnetic field and a pinning term in a bounded smooth two
dimensional domain $\Omega$. Supposing that the Ginzburg-Landau
parameter and the intensity of the magnetic field are large and of the
same order, we determine an accurate asymptotic formula for the
minimizing energy. This asymptotic formula displays the influence of
the pinning term. Also, we discuss the existence of non-trivial
solutions and prove some asymptotics of the third critical field.
\end{abstract}
\maketitle
\section{Introduction}
We consider a bounded, open and simply connected set
$\Omega\subset\R^2$ with smooth boundary. We suppose that $\Omega$
models an inhomogeneous superconducting sample submitted to an
applied external magnetic field. The energy of the sample is given
by the so called pinned Ginzburg-Landau functional,
\begin{equation}\label{eq-2D-GLf}
\mathcal E_{\kappa,H,a,B_{0}}(\psi,\Ab)= \int_\Omega\left( |(\nabla-i\kappa
H\Ab)\psi|^2+\frac{\kappa^2}{2}(a(x,\kappa)-|\psi|^2)^2\right)\,dx
+\kappa^2H^2\int_{\Omega}|\curl\Ab-B_0|^2\,dx\,.
\end{equation}
Here $\kappa$ and $H$ are two positive parameters such that $\kappa$ describes the properties of the material, and $H$ measures the variation of the intensity of the applied magnetic field. The modulus $|\psi|^{2}$ of the wave function
(order parameter) $\psi\in H^1(\Omega;\C)$  measures the density of the superconducting electron Cooper pairs. The magnetic potential
$\Ab$ belongs to $H^1_{\Div}(\Omega)$ where
\begin{equation}\label{eq-2D-hs} H^1_{\Div}(\Omega)=\{\Ab=(\Ab_{1},\Ab_{2})\in
H^1(\Omega)^{2}~:~\Div \Ab=0~{\rm in}~\Omega \,,\,\Ab\cdot\nu=0~{\rm
on}\,
\partial\Omega \,\}\,,
\end{equation}
with  $\nu$ being the unit interior normal vector of
$\partial\Omega$.\\
 The function $\kappa H\curl\Ab$ gives the induced magnetic field.

When $\psi\equiv0$ and $(\psi,\Ab)$ is a minimizer or a critical point of the functional, we call this pair normal state. In our case it is easy to see normal minimizers (if any)  are necessarily in the form $(0,\Ab)$ with $\Ab$ in $H^1_{\Div} (\Omega)$  such that $\curl\Ab=B_{0}$. This solution is unique
and denoted by $\Fb$.  A natural question will be to determine under which condition this normal solution is a minimizer.

The function $B_{0}\in C^{\infty}(\overline{\Omega})$ is the intensity of the external magnetic field which is variable in our problem. Let
\begin{equation}\label{gamma}
\Gamma=\{x\in\overline{\Omega}: B_{0}(x)=0\}\,.
\end{equation}
We assume that either $\Gamma$ is empty or that $B_{0}$ satisfies :
\begin{equation}\label{B(x)}
\left\{
\begin{array}{ll}
|B_{0}| + |\nabla B_0 | >0&\mbox{ in } \overline{\Omega}\\
\nabla B_{0}\times\vec{n}\neq 0 &\mbox{ on } \Gamma\cap\partial\Omega\,.
\end{array}
\right.
\end{equation}
The assumption in \eqref{B(x)} implies that for any open set $\omega$ relatively compact in $\Omega$, $\Gamma\cap\omega$ is  either empty, or consists of a union of smooth curves. 

The energy $\mathcal E_{\kappa,H,a,B_{0}}$ considered here is slightly different from the classical Ginzburg-Landau energy in the sense that there is a varying term denoted by $a(x,\kappa)$ penalizing the variations of the order parameter $\psi$ and called the pinning term. This term arises also naturally in the microscopic derivation of the Ginzburg-Landau theory from BCS theory (see \cite{FHSS})
without any a priori assumption on the sign of $a$.\\

In this paper, we will assume that  the pining term $a$ satisfies:
\begin{ass}\label{assumption}
The function $a(x,\kappa)$ is real, defined on $\overline{\Omega}\times[\kappa_{0},+\infty)$, and satisfies for some $\kappa_0 >0$ the following assumptions:
\begin{enumerate}
 \item[$(A_{1})$] \begin{equation} \label{a1} \forall \kappa\geq \kappa_0\,,  a(\cdot,\kappa)\in C^{1}(\overline{\Omega})\,.\end{equation}
  \item[$(A_{2})$]
\begin{equation}\label{a2}
\sup_{x\in\overline{\Omega},\,\kappa\geq\kappa_{0}}| a(x,\kappa)| <+\infty \,.
\end{equation}
\item[$(A_{3})$]
\begin{equation}\label{a3}
\sup_{{x\in\overline{\Omega},\,\kappa \geq \kappa_{0}}} |\nabla_{x}\,a(x,\kappa)|< +\infty \,.
\end{equation}
\item[$(A_{4})$]
There exists a  positive constant $C_{1}$, such that,
\begin{equation}\label{a4}
 \forall \kappa\geq \kappa_0\,,\qquad \mathcal{L} \left(\partial\{a(x,\kappa)>0\}\right)\leq C_{1}\,\kappa^{\frac{1}{2}}\,,
\end{equation}
where $\mathcal{L}$  is the "length" of $\partial\{a(x,\kappa)>0\}$  in $\Omega$ in a sense  that will be explained in \eqref{defA4}.
 \end{enumerate}
\end{ass}
Let us introduce for later use,
\begin{equation}\label{def:L}
L(\kappa) = \sup_x  |\nabla_{x}\,a(x,\kappa)| \,,
\end{equation}
\begin{equation}\label{def:sup-a}
\overline{a}=\sup_{x\in\overline{\Omega},\,\kappa\geq\kappa_{0}}a(x,\kappa)
\end{equation}
and
\begin{equation}\label{def:inf-a}
\underline{a}=\inf_{x\in\overline{\Omega},\,\kappa\geq\kappa_{0}}a(x,\kappa).
\end{equation}
The assumption in  ($A_{3}$) gives a uniform control for any $\kappa$ of the oscillation of $a(.,\kappa)$ which will be made precise later by an assumption on $L(\kappa)$.
Notice that the normal state $(0,\Fb)$ is a critical point of the functional in \eqref{eq-2D-GLf}.
It is standard, starting from a minimizing sequence, to prove the existence of minimizers in $ H^1(\Omega;\C)\times
H^1_{\Div}(\Omega)$ of the functional $\mathcal E_{\kappa,H,a,B_{0}}$.  A minimizer $(\psi,\Ab)$ of \eqref{eq-2D-GLf} is a weak solution of
the Ginzburg-Landau equations,

\begin{equation}\label{eq-2D-GLeq}
\left\{
\begin{array}{llll}
-(\nabla-i\kappa H\Ab)^2\psi=\kappa^2\, (a(x,\kappa)-|\psi|^2)\, \psi&{\rm in}&
\Omega&(a)
\\
-\nabla^{\bot}\curl(\Ab-\Fb)=\displaystyle\frac1{\kappa
H}\IM(\overline{\psi}\,(\nabla-i\kappa
H\Ab)\psi) &{\rm in}&\Omega&(b)\\
\nu\cdot(\nabla-i\kappa H\Ab)\psi=0&{\rm
on}&\partial\Omega&(c)\\
\curl\Ab=\curl\Fb&{\rm on}&\partial\Omega&(d)\,.
\end{array}\right.
\end{equation}
Here, $\curl\Ab=\partial_{x_1}\Ab_{2}-\partial_{x_2}\Ab_{1}$ and
$\nabla^{\bot}\curl\Ab=(\partial_{x_2}(\curl\Ab),
-\partial_{x_1}(\curl\Ab)).$

Let us introduce the  magnetic Schr\"odinger operator in an open set $\widetilde{\Omega}$ in $\R^2$:
\begin{equation}\label{def:P}
P_{A,V}^{\widetilde{\Omega}}=-(\nabla-iA)^{2}+V(x)\,,
\end{equation}
where $A\in H^{1}_{\Div}(\widetilde{\Omega})$ and $V$ is a continuous function bounded from below.\\
The form domain of $P_{A,V}^{\widetilde{\Omega}}$ is
$$
\mathcal{V}(\widetilde{\Omega})=\{u\in L^{2}(\widetilde{\Omega})\,,\quad (\nabla-iA)u\in L^{2}(\widetilde{\Omega})\,,\quad(V+C)^{\frac{1}{2}}u\in L^{2}(\widetilde{\Omega})\}\,,
$$
and its  operator domain is  given by
$$
D(P_{A,V}^{\widetilde{\Omega}}):=\{u\in\mathcal{V}(\widetilde{\Omega})\,,\quad P_{A,V}^{\widetilde{\Omega}}u\in L^{2}(\widetilde{\Omega}),\quad \nu\cdot(\nabla-i A)u=0~{\rm on}~\partial\widetilde{\Omega}\}\,.
$$
Then, $\eqref{eq-2D-GLeq}_{a,c}$ reads  $$P_{A,V}^{\Omega}\,\psi=-\kappa^{2}\,|\psi|^{2}\psi\,,$$ with $A=\kappa H \Ab$, $\psi \in D(P_{A,V}^{\Omega})$ and $V=-\kappa^{2}\,a\,$.\\

There are many papers on the Ginzburg-Landau functional with a
pinning term, most of them study the influence of the pinning term
on the location of {\it vortices}, i.e. the zeros of the minimizing
order parameter. For the functional without a magnetic field (i.e.
$B_0=0$ in \eqref{eq-2D-GLf}), the influence of the pinning term is
studied in \cite{LM} and more recently in \cite{DSM} and the
references therein. The pinning term (i.e. the function $a$) in
\cite{LM} is a step function independent of $\kappa$; more
complicated $\kappa$-dependent periodic step functions are
considered in \cite{DSM}. The magnetic version of the functional in
\cite{LM} is studied  in \cite{A.K, K-cocv}.


In \cite{ASS},  Aftalion, Sandier and Serfaty considered a {\bf
smooth} and $\kappa$-dependent pinning term $a$ satisfying:
\begin{enumerate}
 \item[$(H_{1})$]
 $L(\kappa) \ll\kappa H.$
 \item[$(H_{2})$] There exist a continuous function $a(x)$,  a positive constant $a_{0}$ and, for all $\kappa\geq 0$, there exist two functions $
\sigma(\kappa)=\textit{o}\left(\left(\ln\left|\ln\frac{1}{\kappa}\right|\right)^{-\frac{1}{2}}\right)
$
and $ \beta(x,\kappa)\geq 0$ such that,
$$\displaystyle \min_{B(x,\sigma(\kappa))} \beta(x,\kappa)=0\,, \qquad a(x,\kappa)= a(x)+\beta(x,\kappa)\,,\qquad{\rm and}\qquad 0< a_{0}\leq a(x)\leq 1\,.$$
 \end{enumerate}
The study contains the case when $a(x,\kappa)=a(x)$  ($\beta=0$) but
also cases with a $\kappa$- control of the $x$-oscillation  of
$\beta(\cdot,\kappa)$ which could increase with $\kappa$. In the
scales of this paper, the results in \cite{ASS} are valid
 when the parameter
  $H$ is of  order $\frac{|\ln\kappa|}{\kappa}$ as $\kappa \longrightarrow +\infty$.

 Extending the discussion, the
 functional in \eqref{eq-2D-GLf} is close to models of Bose-Einstein
 condensates (see e.g. \cite{AfAB, AlBr}).

In this paper, we will analyze how the pinning term appears in
the asymptotics of the energy in the presence of a strong external
variable magnetic field (see Theorem~\ref{thm-2D-main} below). Also,
we discuss the influence of the pinning on the asymptotic expression
of the third critical field $H_{C_3}$ (see Theorems~\ref{thm:HC3}
and \ref{thm:HC3-vr}).

We focus on the regime of large values of $\kappa$, $\kappa\rightarrow+\infty$ and we study the ground state energy defined as follows,
\begin{equation}\label{eq-2D-gs}
\E0(\kappa,H,a,B_{0})=\inf\big\{ \mathcal
E_{\kappa,H,a,B_{0}}(\psi,\Ab)~:~(\psi,\Ab)\in H^1(\Omega;\C)\times
H^1_{\Div}(\Omega)\big\}\,.
\end{equation}
More precisely, we give an asymptotic estimate which is valid in the simultaneous limit $\kappa\longrightarrow+\infty$ and $H(\kappa)\longrightarrow+\infty$ with the constraint that $\frac{H(\kappa)}{\kappa}$ remains asymptotically of uniform size, that is satisfying
\begin{equation}\label{cond-H}
\lambda_{\min}\leq \frac{H(\kappa)}{\kappa}\leq\lambda_{\max}\qquad(\kappa\geq\kappa_{0})\,,
\end{equation}
where $\lambda_{\min},\,\lambda_{\rm max}$ are positive constants such that $\lambda_{\min}<\lambda_{\max}$.\\
The behavior of $\E0(\kappa,H,a,B_{0})$ involves a  function $\hat{f}:[0,+\infty)\longrightarrow[0,\frac{1}{2}]$ introduced in \cite[Theorem~2.1]{KA2}. The function $\hat{f}$ is increasing, continuous and $\hat{f}(b)=\frac{1}{2}$, for all $b\geq 1$.

\begin{thm}\label{thm-2D-main}
Suppose that Assumption~\ref{assumption} and \eqref{cond-H} hold,
and
\begin{equation}
L(\kappa)= \mathcal O(\kappa^{\frac{1}{2}})\qquad{\rm as}~\kappa \rightarrow +\infty\,.
\end{equation}
 The ground state energy in \eqref{eq-2D-gs} satisfies
\begin{multline}\label{eq-2D-thm}
 \E0(\kappa,H,a,B_{0})=\kappa^{2}\int_{\{a(x,\kappa)>0\}}a(x,\kappa)^{2}\,\hat{f}\left(\frac{H}{\kappa}\frac{|B_{0}(x)|}{a(x,\kappa)}\right)\,dx\\
 +\frac{\kappa^{2}}{2}\int_{\{a(x,\kappa)\leq 0\}}a(x,\kappa)^{2}\,dx+\textit{o}\left(\kappa^{2}\right)\,,\qquad{\rm as}~\kappa\longrightarrow+\infty\,.
\end{multline}
\end{thm}
When $\Omega\cap\{a(x,\kappa)>0\}=\varnothing$, we obtain directly from \eqref{eq-2D-gs}
$$
\mathcal E_{\kappa,H,a,B_{0}}(\psi,\Ab)\geq\frac{\kappa^{2}}{2}\int_{\Omega}a(x,\kappa)^{2}\,dx = \mathcal E_{\kappa,H,a,B_{0}}(0,\Fb)\,.
$$
Hence the minimizer of $\mathcal E_{\kappa,H,a,B_{0}}$  is the normal state. In physical terms, this case corresponds to the case when we are above the critical temperature.

We will describe later  cases when the remainder term in \eqref{eq-2D-thm} is indeed small compared with the leading order term (see Section~\ref{examples}).

The assumptions in Theorem~\ref{thm-2D-main} contain the case when
the function $a$ is constant and equals~$1$, which was proved in
\cite{KA}  under Assumption~\eqref{cond-H}.

Along the proof of Theorem~\ref{thm-2D-main}, we obtain an estimate
of the `magnetic energy' as follows:

\begin{corol}\label{corol-2D-main}
Under the assumptions of Theorem~\ref{thm-2D-main}, we have
\begin{equation}
(\kappa H)^2\int_{\Omega}|\curl\Ab-B_0|^2\,dx=\textit{o}(\kappa^{2})\,,\qquad{\rm as}~\kappa\longrightarrow+\infty\,.
\end{equation}
\end{corol}
If $\mathcal D$ is a domain in $\Omega$,
we introduce the local energy in $\mathcal D$ of $(\psi,\Ab) \in H^1(\Omega;\C)\times H^1_{\Div}(\Omega)$ by:
\begin{equation}\label{eq-GLe0}
\mathcal E_{0}(\psi,\Ab;a,\mathcal{D})=\int_{\mathcal{D}}|(\nabla -i\kappa
H\Ab)\psi|^2\,dx+\frac{\kappa^{2}}{2}\int_{\mathcal{D}}(a(x,\kappa)-|\psi|^{2})^2\,dx\,.
\end{equation}

The next theorem gives an estimate of the local energy $\mathcal E_{0}(\psi,\Ab;a,\mathcal{D})$.
\begin{theorem}\label{lc-en}
Under the assumptions of Theorem~\ref{thm-2D-main}, if $(\psi,\Ab)$
is a minimizer of \eqref{eq-2D-GLf} and $\mathcal{D}$ is regular set
such that $\mathcal{\overline{D}}\subset\Omega$, then
\begin{multline}\label{eq-lc-en}
\mathcal E_{0}(\psi,\Ab;a,\mathcal{D})=\kappa^{2}\int_{\mathcal{D}\cap\{a(x,\kappa)>0\}}a(x,\kappa)^{2}\,\hat{f}\left(\frac{H}{\kappa}\frac{|B_{0}(x)|}{a(x,\kappa)}\right)\,dx\\
 +\frac{\kappa^{2}}{2}\int_{\mathcal{D}\cap\{a(x,\kappa)\leq 0\}}a(x,\kappa)^{2}\,dx+\textit{o}\left(\kappa^{2}\right)\,,\qquad{\rm as}~\kappa\longrightarrow+\infty\,.
\end{multline}
\end{theorem}
Theorem~\ref{lc-en} will be useful in the proof of the next theorem which gives the asymptotic  behavior of the order parameter $\psi$, when $(\psi,\Ab)$ is a global minimizer.
\begin{theorem}\label{est-psi-main}
Under the assumptions of Theorem~\ref{thm-2D-main}, if $(\psi,\Ab)$ is a
minimizer of \eqref{eq-2D-GLf} and $\mathcal{D}$ is a regular set such that $\overline{\mathcal{D}}\subset \Omega$, then
\begin{equation}\label{est-psi-D}
\int_{\mathcal{D}}|\psi(x)|^{4}\,dx=-\int_{\mathcal{D} \cap\{a(x,\kappa)>0\}} a(x,\kappa)^{2}\left\{2\hat{f}\left(\frac{H}{\kappa}\frac{|B_{0}(x)|}{a(x,\kappa)}\right)-1\right\}\,dx+\textit{o}\left(1\right)\,,\quad{\rm as}~\kappa\longrightarrow+\infty\,.
\end{equation}
\end{theorem}
Formula \eqref{est-psi-D} indicates that $\psi$ is asymptotically localized in the region where $a>0$. When $a(x,\kappa)=1$, Theorem~\ref{est-psi-main} was proved in \cite{KA}.

The techniques that we are going to use here are inspired from those of \cite{KA} and \cite{KA2} (where the case $a=1$ was treated). At a technical level, our proof is slightly different than the proofs in \cite{KA,FK2,SS} since we do not use the uniform elliptic estimates. These important estimates are frequently used in the papers about the Ginzburg-Landau functional (see \cite{FH1}) with a  constant pinning term. They appeared first in \cite{LP} and were then extended to the full regime in \cite{FH2}.

Compared with other papers studying the pinned functional, one
novelty here is that the pinning term has no definite sign, another
one being the consideration of a variable (and a potentially
vanishing) applied magnetic field.

The rest of this paper is devoted to the study of third critical field, i.e. the field above which the normal state $(0,\Fb)$ is the only critical point of the functional in \eqref{eq-2D-GLf},  in the case when the pining term $a$ is independent of $\kappa$ (i.e. $a(x,\kappa)=a(x)$). We define the set:
\begin{equation}\label{def:Ncp}
\mathcal{N}^{\rm cp}(\kappa)=\{H>0: \mathcal E_{\kappa,H,a,B_{0}}~\text{\rm has a non-normal critical point}\}\,.
\end{equation}\label{def:N}
Notice that the above set is bounded (see Theorem~\ref{thm:GP}). We also  introduce the two sets:
\begin{equation}
\mathcal{N}(\kappa)=\{H>0:\mathcal E_{\kappa,H,a,B_{0}}~\text{\rm has a non-normal minimizer}\}\,.
\end{equation}
\begin{equation}\label{def:Nloc}
\mathcal{N}^{\rm loc}(\kappa)=\{H>0:\mu_{1}(\kappa,H)<0\}\,.
\end{equation}
Here, $\mu_{1}(\kappa,H)$ is the ground state energy of the semi-bounded quadratic form
\begin{equation}\label{Quad}
\mathcal{Q}_{\kappa H\Fb,-\kappa^{2}a}^{\Omega}(\phi)=\int_{\Omega}\left(|(\nabla-i\kappa H\Fb)\phi|^{2}-\kappa^{2}\,a(x,\kappa)|\phi|^{2}\right)\,dx\,,
\end{equation}
i.e.
\begin{equation}\label{def:mu1}
\mu_{1}(\kappa,H)=\inf_{\substack{\phi\in H^{1}(\Omega)\\ \phi\neq 0}}\left(\frac{\mathcal{Q}_{\kappa H\Fb,-\kappa^{2}a}^{\Omega}(\phi)}{\|\phi\|^{2}_{L^{2}(\Omega)}}\right)\,.
\end{equation}
Note that $\mu_{1}(\kappa,H)$ is the lowest eigenvalue of $P_{\kappa H\Fb,-\kappa^{2}a}^{\Omega}$. Here, we  refer to \cite{CR,KIH,JPM,XB-KH} for previous contributions.\\
We introduce the following critical fields (cf. e.g.\cite{FH3,LP})\,.
\begin{align}
&\overline{H}_{C_3}^{cp}(\kappa)=\sup\,\mathcal{N}^{cp}(\kappa)\,,\qquad\underline{H}_{C_3}^{cp}(\kappa)=\inf\,(\R_{+}\setminus\mathcal{N}^{cp}(\kappa))\label{def:HC3-o}\,,\\
&\overline{H}_{C_3}(\kappa)=\sup\,\mathcal{N}(\kappa)\,,\qquad\quad\underline{H}_{C_3}(\kappa)=\inf\,(\R_{+}\setminus\mathcal{N}(\kappa))\,,\label{def:HC3}\\
&\overline{H}_{C_3}^{loc}(\kappa)=\sup\,\mathcal{N}^{loc}(\kappa)\,,\qquad\underline{H}_{C_3}^{loc}(\kappa)=\inf\,(\R_{+}\setminus\mathcal{N}^{loc}(\kappa))\label{def:HC3-u}\,.
\end{align}

Below $\underline{H}_{C_3}$, normal states will loose their stability and  above $\overline{H}_{C_3}$, the normal state is (up to a gauge transformation) the only critical point of the functional in \eqref{eq-2D-GLf}.\\
Our aim is to determine the asymptotics of  all the critical fields as $\kappa\longrightarrow+\infty$. This involves spectral quantities related to  three models depending on $\Gamma$ being empty or not. \\
Let us introduce
$$\displaystyle\Theta_{0}=\inf_{\xi\in\R} \mu(\xi)\,,$$
where $\mu$ is the lowest eigen value of the operator
$$
\mathfrak{h}^{N,\xi}:=-\frac{d^2}{dt^2}+(t+\xi)^2\qquad{\rm in}~L^{2}(\R_+)\,,
$$
subject to the Neumann boundary condition $u'(0)=0$.
\begin{theorem}\label{thm:HC3}
Suppose that $\Gamma=\{x\in\Omega: B_{0}(x)=0\}=\varnothing$  and  that $a\in C^{1}(\overline{\Omega})$ satisfies  $\{a>0\}\neq\varnothing$. Then, as $\kappa\longrightarrow+\infty$, all the six critical fields satisfy an asymptotic expansion in the form:
\begin{equation}
H_{C_3}(\kappa)=\max\left(\sup_{x\in\Omega}\frac{a(x)}{|B_{0}(x)|},\sup_{x\in\partial\Omega}\frac{a(x)}{\Theta_{0}|B_{0}(x)|}\right)\,\kappa+\mathcal{O}(\kappa^{\frac{1}{2}})\,.
\end{equation}
\end{theorem}
We introduce
\begin{equation}\label{lambda0}
\lambda_{0}=\inf_{\tau\in \R} \lambda(\tau)\,,
\end{equation}
where $\lambda(\tau)$ is the lowest eigenvalue of the  selfadjoint realization of the differential operator
\begin{equation}\label{defM}
 M(\tau) = -\frac{d^2}{dt^2} +\frac 14  (t^2+2\tau)^2\qquad{\rm in}~L^{2}(\R)\,.
 \end{equation}
We consider, for any $\theta\in(0,\pi)$ the bottom of the spectrum  $\lambda(\R_{+}^{2},\theta)$  of the operator
\begin{equation}\label{def:lambda-theta}
P_{\Ab_{\rm app,\theta},0}^{\R^{2}_{+}}\quad{\rm with}\quad\Ab_{\rm app,\theta}=-\left(\frac{x^{2}_{2}}{2}\cos\,\theta,\frac{x^{2}_{1}}{2}\sin\,\theta \right)\,.
\end{equation}
\begin{theorem}\label{thm:HC3-vr}
Suppose that $\Gamma=\{x: B_{0}(x)=0\}\neq\varnothing$, that
\eqref{B(x)} holds  and  that $a\in C^{1}(\overline{\Omega})$ satisfies  $\{a>0\}\neq\varnothing$. As
$\kappa\longrightarrow+\infty$, the six critical fields in
\eqref{def:HC3-o}-\eqref{def:HC3-u} satisfy the asymptotic
expansion:
$$
H_{C_3}(\kappa)=\max\left(\sup_{x\in\Gamma\cap\overline{\Omega}}\frac{a(x)^{\frac{3}{2}}}{\lambda_{0}^{\frac{3}{2}}|\nabla B_{0}(x)|},\sup_{x\in\Gamma\cap\partial\Omega}\frac{a(x)^{\frac{3}{2}}}{\lambda(\R^{2}_{+},\theta(x))^{\frac{3}{2}}|\nabla B_{0}(x)|}\right)\,\kappa^{2}+\mathcal{O}\left(\kappa^{\frac{7}{4}}\right)\,.
$$
Here $\theta(x)$ denotes the angle between $\nabla B_{0}(x)$ and the inward normal vector $-\nu(x)$.
\end{theorem}
%
\subsection*{Organization of the paper} The rest of the paper is split into twelve sections. Section~\ref{2} analyzes the model problem with a constant magnetic field and a constant pinning term.  Section~\ref{upperbound} establishes an upper bound on the ground state energy.  Section~\ref{section:P.E.} contains useful estimates on minimizers. The estimates in Section~\ref{section:P.E.} are used in Section~\ref{section5} to establish a lower bound of the ground state energy and to finish the proof of Theorem~\ref{thm-2D-main}, Corollary~\ref{corol-2D-main} and Theorem~\ref{lc-en}. In Section~\ref{examples}, we discuss the conclusion in Theorem~\ref{thm-2D-main} by providing various examples of pinning terms obeying Assumption~\ref{assumption}.  Section~{7} is devoted to the proof of Theorem~\ref{est-psi-main}.  Section~\ref{GP} generalizes a theorem of Giorgi-Phillips concerning the breakdown of superconductivity under a large applied magnetic field. Sections~\ref{Section:4} and \ref{10} are devoted to the proof of Theorem~\ref{thm:HC3}. The proof of Theorem~\ref{thm:HC3-vr} is the purpose of Sections~\ref{Section:Asympt-m1-vanish} and \ref{12}.

\subsection* {Notation.}
{Throughout the paper, we use the following notation:}
\begin{itemize}
\item If $b_{1}(\kappa)$ and $b_{2}(\kappa)$ are two positive functions on $[\kappa_{0},+\infty)$, we write $b_{1}(\kappa)\ll b_{2}(\kappa)$ if $b_{1}(\kappa)/b_{2}(\kappa)\to0$ as $\kappa\to\infty$.
\item If $b_{1}(\kappa)$ and $b_{2}(\kappa)$ are two functions with $b_{2}(\kappa)\not=0$, we write $b_{1}(\kappa)\sim b_{2}(\kappa)$\\
 if $b_{1}(\kappa)/b_{2}(\kappa)\to1$ as $\kappa\to\infty$.
\item If $b_{1}(\kappa)$ and $b_{2}(\kappa)$ are two positive functions, we write $b_{1}(\kappa)\approx b_{2}(\kappa)$
if there exist positive constants $c_1$, $c_2$ and $\kappa_0$ such
that $c_1b_{2}(\kappa)\leq b_{1}(\kappa)\leq c_2b_{2}(\kappa)$ for all
$\kappa\geq\kappa_0$.
\item Let $a_{+}(\widetilde{x}_{0},\kappa)=[a(\widetilde{x}_{0},\kappa)]_{+}$ and $a_{-}(\widetilde{x}_{0},\kappa)=[a(\widetilde{x}_{0},\kappa)]_-$  where, for any $x\in\R$, $[x]_+=\max(x,0)$ and $[x]_{-}=\max(-x,0)$.
\item Given $R>0$ and $x=(x_{1},x_{2})\in\R^{2}$, $Q_{R}(x)=(-R/2+x_{1},R/2+x_{1})\times(-R/2+x_{2},R/2+x_{2})$ denotes the
square of side length $R$ centered at $ x=(x_1,x_2)$ and we write $Q_{R}=Q_{R}(0)$.
\end{itemize}

\section{A reference problem}\label{2}
The reference problem  is obtained by freezing the pinning term and the magnetic field. This approximation will appear to be reasonable in squares avoiding the boundary and the zero set $\Gamma$ of the magnetic field $B_0$.
\subsection{A useful function}\label{uf}
Consider $R>0$, $b>0$,  $\zeta\in\{-1,+1\}$  and $\alpha\in\R\,$. We define the following Ginzburg-Landau energy with constant magnetic field on $H^1(Q_R)$ by
\begin{equation}\label{eq-GL-F}
u\mapsto F^{\zeta,\alpha}_{b,Q_{R}}(u)=\int_{Q_{R}}\left(b|(\nabla-i\zeta\Ab_0)u|^2+\frac{1}{2}\left(\alpha-|u|^2\right)^{2}\right)\,dx\,,
\end{equation}
where
\begin{equation}\label{eq-hc2-mpA0}
\Ab_{0}(x)=\frac1{2}(-x_2,x_1)\,,\qquad\forall\,x=(x_1,x_2)\in\R^{2}\,.
\end{equation}
We have two cases according to the sign of $\alpha$\,:~\\
\textbf{Case~1}. $\alpha>0$:~\\
We notice that
\begin{equation}\label{change-F}
F^{\zeta,\alpha}_{b,Q_{R}}(u)=\alpha^{2} F^{\zeta,1}_{\widetilde{b},Q_{R}}(\widetilde{u})\,,
\end{equation}
where
\begin{equation}\label{def-b-u}
\widetilde{b}=\frac{b}{\alpha}\qquad\text{and}\qquad\widetilde{u}=\frac{u}{\sqrt{\alpha}}\,.
\end{equation}
We introduce the two ground state energies
\begin{eqnarray}
e_{N}(b,R,\alpha)=\inf\left\{F^{+1,\alpha}_{b,Q_{R}}(u): u\in H^{1}(Q_{R};\C)\right\}\label{eN}\\
e_{D}(b,R,\alpha)=\inf\left\{F^{+1,\alpha}_{b,Q_{R}}(u): u\in H^{1}_{0}(Q_{R};\C)\right\}\label{eD}\,.
\end{eqnarray}
As $F^{+1,\alpha}_{b,Q_{R}}(u)=F^{-1,\alpha}_{b,Q_{R}}(\overline{u})$, it is immediate that,
\begin{equation}\label{F+=F-}
\inf F^{+1,\alpha}_{b,Q_{R}}(u)=\inf F^{-1,\alpha}_{b,Q_{R}}(u)\,.
\end{equation}
Using \eqref{eN} and \eqref{eD}, we get from \eqref{change-F}
\begin{equation}\label{eNa=eN}
e_{N}(b,R,\alpha)=\alpha^{2}\,e_{N}\left(\frac{b}{\alpha},R,1\right)=\alpha^{2}\,e_{N}\left(\frac{b}{\alpha},R\right)\,,
\end{equation}
and
\begin{equation}\label{eDa=eD}
e_{D}(b,R,\alpha)=\alpha^{2}\,e_{D}\left(\frac{b}{\alpha},R,1\right)=\alpha^{2}\,e_{D}\left(\frac{b}{\alpha},R\right)\,.
\end{equation}
As a consequence of \eqref{change-F} and \eqref{def-b-u}, $\widetilde{u}$ is a minimizer of $F^{\zeta,1}_{\widetilde{b},Q_{R}}$ if and only if $u$ is a minimizer of $F^{\zeta,\alpha}_{b,Q_{R}}$. In particular any minimizer of $F^{\zeta,\alpha}_{b,Q_{R}}$ satisfies
\begin{equation}\label{up-u-a}
|u|\leq \sqrt{\alpha}\,.
\end{equation}
Recall  from  \cite[Theorem~2.1]{FK2}  that,
\begin{equation}\label{f(x)}
\displaystyle \hat{f}\left(\mathfrak{b}\right)=\lim_{R\longrightarrow\infty}\frac{e_{D}(\mathfrak{b},R)}{R^{2}}\,.
\end{equation}

The next proposition was proved in \cite[Lemma~2.2, Proposition~2.4]{KA2} in the case $\alpha=1$. It's present form can be deduced immediately from \eqref{eNa=eN}.
\begin{prop}\label{pro-f(b)}
For all $M>0$, there exist universal constants $C_{M}$ and $R_{M}$ such that $\forall R\geq R_{M},\,\forall\,b>0,\,\forall\,\alpha>0$ such that $\displaystyle 0<\frac{b}{\alpha}\leq M$, we have
\begin{equation}\label{eN>eD}
e_{N}(b,R,\alpha)\geq\,e_{D}\left(b,R,\alpha\right)-C_{M}\alpha^{2}R\left(\frac{b}{\alpha}\right)^{\frac{1}{2}}
\end{equation}
\begin{equation}\label{est-f(b)-eD}
\alpha^{2}\hat{f}\left(\frac{b}{\alpha}\right) \leq \frac{e_{D}(b,R,\alpha)}{R^{2}}\leq\alpha^{2}\hat{f}\left(\frac{b}{\alpha}\right)+C_{M}\frac{\alpha^{\frac{3}{2}}\sqrt{b}}{R}.
\end{equation}
\end{prop}
\textbf{Case~2}. $\alpha\leq 0\,$:~\\
When $\alpha\leq 0$, we write $\alpha=-\alpha_{0}$, $\alpha_{0}\geq 0$ and \eqref{eq-GL-F} becomes
\begin{equation}\label{eq-GL-F2}
F^{\zeta,\alpha}_{b,Q_{R}}(u)=\int_{Q_{R}}\left(b|(\nabla-i\zeta\Ab_0)u|^2+\frac{1}{2}\left(\alpha_{0}+|u|^2\right)^{2}\right)\,dx\,.
\end{equation}
It is clear that,
$$F^{\zeta,\alpha}_{b,Q_{R}}(u)\geq \frac{1}{2}\alpha_{0}^{2} R^{2}\qquad{\rm and}\qquad F^{\zeta,\alpha}_{b,Q_{R}}(0)=\frac{1}{2}\alpha_{0}^{2} R^{2}\,.$$
As a consequence, we have
$$
\frac{1}{2}\alpha_{0}^{2}R^{2}\leq e_{D}(b,R,\alpha)\leq F^{\zeta,\alpha}_{b,Q_{R}}(0)=\frac{1}{2}\alpha_{0}^{2}R^{2}\,.
$$
When $\alpha=0$, it is easy to show that
$$F^{\zeta,\alpha}_{b,Q_{R}}(u)=0\,.$$
Notice that the only minimizer of $F^{\zeta,\alpha}_{b,Q_{R}}$ is $u=0\,$.
Thus, for any $\alpha\leq 0\,$, we obtain
\begin{equation}\label{F=}
\frac{e_{D}(b,R,\alpha)}{R^{2}}=\frac{1}{2}\alpha^{2}\,.
\end{equation}
\section{Upper bound of the energy}\label{upperbound}
The aim of this section is to give an upper bound of the ground state energy $\E0(\kappa,H,a,B_{0})$ introduced in \eqref{eq-2D-gs} under Assumption~\eqref{cond-H}.
For this we cover $\Omega$ by (the closure of)  disjoint open squares $(Q_{\ell}(\gamma))_{\gamma}$ whose centers $\gamma$ belong to a square lattice $\Gamma_\ell= \ell \mathbb Z \times \ell \mathbb Z$.

We will get an upper bound  by matching together approximate minimizers, in each square $Q_{\ell}(\gamma)$ contained in $\Omega$, obtained by freezing the pinning term and the magnetic field at a suitable point $\tilde \gamma$. The size $\ell$ of the square will be chosen as a function of $\kappa$. We start with estimates in a given square $Q_\ell (x_0)$ and will take later $x_0=\gamma\,$.\\

{\bf About Assumption $(A_4)$.}\\
 We first explain what was meant in Assumption $(A_4)$. By $\mathcal L(\partial\{a >0\}) \leq  C_1 \kappa^\frac 12$ we mean the existence of $C_2 >0$ and $\kappa_0$ such that:
 \begin{equation}\label{defA4}
 \forall \kappa \geq \kappa_0\,,\, \forall \ell \leq C_2 \kappa^{-\frac 12}\,,\, {\rm card}\,\{ \gamma \in \Gamma_\ell \cap \Omega \mbox{ with } Q_\ell (\gamma) \cap \partial\{a >0\} \cap \Omega \neq \emptyset\} \leq C_1 \kappa^\frac 12 \ell^{-1}\,.
 \end{equation}
~\\

Using Assumption~\eqref{def:L}, for any $\widetilde{x}_{0}\in \overline{Q_{\ell}(x_{0})}$ and $\kappa \geq \kappa_0$, we observe that,
\begin{equation}\label{app-a}
|a(x,\kappa)-a(\widetilde{x}_{0},\kappa)|\leq \left(\sup_x  |\nabla_{x}\,a(x,\kappa)|\right)\,|x-x_{0}|\leq \frac{\ell}{\sqrt{2}}\,L (\kappa) \,,\qquad\forall x\in Q_{\ell}(x_{0})\,.
\end{equation}
\begin{definition}[$\rho$-admissible]\label{rho-adm}
Let $\rho\in(0,1)$. We say that triple $(\ell,x_0,\widetilde{x}_{0})$ is $\rho$-admissible if  $\overline{Q_{\ell}(x_0)}\subset\{|B_{0}|>\rho\}\cap\Omega$ and $\widetilde{x}_{0}\in\overline{Q_{\ell}(x_{0})}$. In this case,  we also say that the pair $(\ell,x_{0})$ is $\rho$-admissible and the corresponding square $Q_{\ell}(x_{0})$ is $\rho$ admissible.
\end{definition}
We recall from \cite[Section~3]{KA2} the definition of the test function,
\begin{equation}\label{def-w2}
\widetilde{w}_{\ell,x_0,\widetilde{x}_{0}}(x)=\begin{cases}
e^{i\kappa H\varphi_{x_{0},\widetilde{x}_{0}}}\widetilde{u}_{R}\left(\frac{R}{\ell}(x-x_{0})\right)&{\rm if}~x\in Q_{\ell}(x_0)\subset\{B_{0}>\rho\}\cap\Omega\\
e^{i\kappa H\varphi_{x_{0},\widetilde{x}_{0}}}\overline{\widetilde{u}}_{R}\left(\frac{R}{\ell}(x-x_{0})\right)&{\rm if}~x\in Q_{\ell}(x_0)\subset\{B_{0}<-\rho\}\cap\Omega\,,
\end{cases}
\end{equation}
where $\widetilde{u}_{R}\in H^{1}_{0}(\Omega)$ is a minimizer of $F^{+1,1}_{b,Q_{R}}$ satisfying by \eqref{up-u-a} $|\widetilde{u}_{R}|\leq 1$ and $\varphi_{x_{0},\widetilde{x}_{0}}$ is the function introduced in \cite[Lemma~A.3]{KA} that satisfies
\begin{equation}\label{F-A}
|\Fb(x)-B_{0}(\widetilde{x}_{0})\Ab_{0}(x-x_{0})-\nabla\varphi_{x_{0},\widetilde{x}_{0}}(x)|\leq C\, \ell^{2},\,\qquad\,\,\forall x \in Q_{\ell}(x_{0})\,.
\end{equation}
Here $B_{0}=\curl\Fb$ and $\Ab_{0}$ is the magnetic potential introduced in \eqref{eq-hc2-mpA0}.

Let us introduce the function:
\begin{equation}\label{def-w}
w_{\ell,x_0,\widetilde{x}_{0}}(x)=\sqrt{a_{+}(\widetilde{x}_{0},\kappa)}\,\widetilde{w}_{\ell,x_0,\widetilde{x}_{0}}(x)\,, \qquad \forall x\in Q_{\ell}(\widetilde{x}_{0})\,.
\end{equation}
Using the bound $|\widetilde{w}_{\ell,x_0,\widetilde{x}_{0}}|\leq1$, which is immediately deduced from the bound of $|\widetilde{u}_{R}|$, we get from \eqref{def-w},
\begin{equation}\label{up-w}
|w_{\ell,x_0,\widetilde{x}_{0}}|^{2} \leq a_{+}(\widetilde{x}_{0},\kappa)\,.
\end{equation}
\begin{prop}\label{pp-up-Eg}
Under Assumptions~\eqref{B(x)}-\eqref{a3}, there exist positive constants $C$ and $\kappa_{0}$ such that if $\kappa\geq\kappa_{0}$, $\ell\in(0,1)$, $\delta\in(0,1)$, $\rho>0$, $\ell^{2}\kappa H\rho>1$ and $(\ell,x_{0},\widetilde{x}_{0})$ is a $\rho$-admissible triple, then,
\begin{multline}\label{up-Eg-eq}
\frac{1}{|Q_{\ell}(x_0)|}\mathcal{E}_{0}(w_{\ell,x_0,\widetilde{x}_{0}},\Fb;a,Q_{\ell}(x_0))\leq (1+\delta)\kappa^{2}\left[a_{+}(\widetilde{x}_{0},\kappa)^{2}\hat{f}\left(\frac{H\,|B_{0}(\widetilde{x}_{0})|}{\kappa\,a_{+}(\widetilde{x}_{0},\kappa)}\right)+\frac{1}{2}a_{-}(\widetilde{x}_{0},\kappa)^{2}\right]\\
+C\left(\frac{1}{\kappa\ell}+\delta^{-1}\ell^{2}L(\kappa)^{2}+\delta^{-1}\kappa^{2}\ell^{4}\right)\kappa^{2}\,.
\end{multline}
\end{prop}
\begin{proof}~\\
Let
\begin{equation}\label{def-Rb}
R=\ell\sqrt{\kappa H |B_{0}(\widetilde{x}_{0})|}\qquad{\rm and}\qquad b=\frac{H\,|B_{0}(\widetilde{x}_{0})|}{\kappa}\,.
\end{equation}
First we estimate $\frac{\kappa^{2}}{2}\int_{Q_{\ell}(x_{0})}(a(x,\kappa)-|w_{\ell,x_0,\widetilde{x}_{0}}|^{2})^{2}\,dx$ from above. Using \eqref{app-a}, we get the existence of a constant $C>0$ such that for any $\delta\in(0,1)$ and any $\kappa\geq \kappa_{0}$,
\begin{align}\label{up-a}
\frac{\kappa^{2}}{2}\int_{Q_{\ell}(x_{0})}\left(a(x,\kappa)-|w_{\ell,x_0,\widetilde{x}_{0}}|^{2}\right)^{2}\,dx&\leq (1+\delta)\frac{\kappa^{2}}{2}\int_{Q_{\ell}(x_{0})}\left(a(\widetilde{x}_{0},\kappa)-|w_{\ell,x_0,\widetilde{x}_{0}}|^{2}\right)^{2}\,dx\nonumber\\
&\qquad\qquad\qquad+(1+\delta^{-1})\frac{\kappa^{2}}{2}\int_{Q_{\ell}(x_{0})}\left(a(\widetilde{x}_{0},\kappa)-a(x,\kappa)\right)^{2}\,dx\nonumber\\
&\leq(1+\delta)\frac{\kappa^{2}}{2}\int_{Q_{\ell}(x_{0})}\left(a(\widetilde{x}_{0},\kappa)-|w_{\ell,x_0,\widetilde{x}_{0}}|^{2}\right)^{2}\,dx\nonumber\\
&\qquad\qquad\qquad\qquad\qquad\qquad\qquad\qquad+C\delta^{-1}\kappa^{2}\ell^{4}L(\kappa)^{2}\,.
\end{align}
The estimate of $\int_{Q_{\ell}(x_{0})}|(\nabla-i\kappa H \Fb)w_{\ell,x_0,\widetilde{x}_{0}}|^{2}\,dx$ from above is the same as in \cite[Proposition~3.1]{KA2}. We have
\begin{align}\label{up-F}
&\int_{Q_{\ell}(x_{0})}|(\nabla-i\kappa H \Fb)w_{\ell,x_0,\widetilde{x}_{0}}|^{2}\,dx\nonumber\\
&\qquad\qquad\leq (1+\delta)\int_{Q_{\ell}(x_0)}\left|\left(\nabla-i\kappa H (B_{0}(\widetilde{x}_{0})\Ab_{0}(x-x_0)+\nabla\varphi_{x_{0},\widetilde{x}_{0}})\right)w_{\ell,x_0,\widetilde{x}_{0}}\right|^{2}\,dx\nonumber\\
&\qquad\qquad\qquad\qquad\qquad\qquad\qquad\qquad\qquad\qquad\qquad\qquad\qquad+C\delta^{-1}\kappa^{4}\ell^{6}|\, w_{\ell,x_0,\widetilde{x}_{0}}|^{2}\,.
\end{align}
From \eqref{def:sup-a}, by collecting \eqref{up-a}, \eqref{up-F} and \eqref{up-w}, we find that,
\begin{multline}\label{1up-loc-en}
\mathcal E_{0}(w_{\ell,x_0,\widetilde{x}_{0}},\Fb;a,Q_{\ell}(x_0))\leq(1+\delta)\mathcal E_{0}\big(w_{\ell,x_0,\widetilde{x}_{0}},\,B_{0}(\widetilde{x}_{0})\Ab_{0}(x-x_0)+\nabla\varphi_{x_{0},\widetilde{x}_{0}};a(\widetilde{x}_{0},\kappa),Q_{\ell}(x_0)\big)\\
+C\delta^{-1}(\kappa^{2}\ell^{4}L(\kappa)^{2}+\kappa^{4}\ell^{6}\,a_{+}(\widetilde{x}_{0},\kappa))\,.
\end{multline}
As we did in \cite{KA2}, we use the change of variable $y=\frac{R}{\ell}(x-x_0)$ and obtain
\begin{align*}
&\mathcal E_{0}\big(w_{\ell,x_0,\widetilde{x}_{0}},\,B_{0}(\widetilde{x}_{0})\Ab_{0}(x-x_0)+\nabla\varphi_{x_{0},\widetilde{x}_{0}};a(\widetilde{x}_{0},\kappa),Q_{\ell}(x_0)\big)\\
&=\int_{Q_{R}}\left[a_{+}(\widetilde{x}_{0},\kappa)\left|\left(\frac{R}{\ell}\nabla-i\frac{R}{\ell}\zeta_{\ell}\,\Ab_{0}(y)\right)\widetilde{u}_{R}(y)\right|^{2}+\frac{\kappa^{2}}{2}\left(a(\widetilde{x}_{0},\kappa)-a_{+}(\widetilde{x}_{0},\kappa)\left|\widetilde{u}_{R}(y)\right|^2\right)^2\right]\frac{\ell^2}{R^2}dy.
\end{align*}
Here, we denote by $\zeta_{\ell}$ the sign of $B_{0}(x_{0})$.\\
We distinguish between two cases:\\
\textbf{Case~1:} When $a(\widetilde{x}_{0},\kappa)>0$, we get
$$
\mathcal E_{0}\big(w_{\ell,x_0,\widetilde{x}_{0}},B_{0}(\widetilde{x}_{0})\Ab_{0}(x-x_0)+\nabla\varphi_{x_{0},\widetilde{x}_{0}};a(\widetilde{x}_{0},\kappa),Q_{\ell}(x_0)\big)=\frac{a(\widetilde{x}_{0},\kappa)^{2}}{b}F^{\zeta_{\ell},1}_{b/a(\widetilde{x}_{0},\kappa),Q_{R}}(\widetilde{u}_{R})\,.
$$
From \eqref{F+=F-} and \eqref{eNa=eN}, we obtain,
\begin{equation}\label{2up-loc-en}
\mathcal E_{0}\big(w_{\ell,x_0,\widetilde{x}_{0}},B_{0}(\widetilde{x}_{0})\Ab_{0}(x-x_0)+\nabla\varphi_{x_{0},\widetilde{x}_{0}};a(\widetilde{x}_{0},\kappa),Q_{\ell}(x_0)\big)=\frac{1}{b}e_{D}(b,R,a(\widetilde{x}_{0},\kappa))\,.
\end{equation}
As a consequence of the upper bound in \eqref{est-f(b)-eD}, the ground state energy $e_{D}(b,R,a(\widetilde{x}_{0},\kappa))$ in \eqref{2up-loc-en} is bounded for all $b>0$ and $R\geq 1$ by:
\begin{align}\label{F=eD}
e_{D}(b,R,a(\widetilde{x}_{0},\kappa))&\leq a(\widetilde{x}_{0},\kappa)^{2}\,R^{2}\,\hat{f}\left(\frac{b}{a(\widetilde{x}_{0},\kappa)}\right)+C_{M}a(\widetilde{x}_{0},\kappa)^{\frac{3}{2}}\,R\,\sqrt{b}\,.
\end{align}
With the choice of $R$ in \eqref{def-Rb}, we have effectively $R\geq 1$ which follows from the assumption $R\geq \ell\sqrt{\kappa H \rho}>1$.\\
We get from \eqref{2up-loc-en} and \eqref{F=eD} the estimate
\begin{multline}\label{3up-loc-en}
\mathcal E_{0}(w_{\ell,x_0,\widetilde{x}_{0}},\zeta_{\ell}|B_{0}(\widetilde{x}_{0})|\Ab_{0}(x-x_0)+\nabla\varphi_{x_{0},\widetilde{x}_{0}};a(\widetilde{x}_{0},\kappa),Q_{\ell}(x_0))\leq a(\widetilde{x}_{0},\kappa)^{2}\frac{R^{2}}{b}\hat{f}\left(\frac{b}{a(\widetilde{x}_{0},\kappa)}\right)\\
+C_{M}\frac{a(\widetilde{x}_{0},\kappa)^{\frac{3}{2}}\,R}{\sqrt{b}}\,,
\end{multline}
with $(b,R)$ defined in \eqref{def-Rb}.\\
By collecting the estimates in \eqref{1up-loc-en}-\eqref{3up-loc-en} we get,
\begin{multline}\label{up-E0}
\mathcal{E}_{0}(w_{\ell,x_0,\widetilde{x}_{0}},\Fb;a(\widetilde{x}_{0},\kappa),Q_{\ell}(x_0))\leq (1+\delta)\, a(\widetilde{x}_{0},\kappa)^{2}\, \frac{R^{2}}{b}\hat{f}\left(\frac{b}{a(\widetilde{x}_{0},\kappa)}\right)\\
+C_{M}\frac{\overline{a}^{\frac{3}{2}}\,R}{\sqrt{b}}+C\delta^{-1}(\kappa^{2}\ell^{4}L(\kappa)^{2}+\kappa^{4}\ell^{6}\overline{a})\,.
\end{multline}
Here, we have used the fact that $\displaystyle a(\widetilde{x}_{0},\kappa)\leq\sup_{x\in\overline{\Omega},\,\kappa\geq\kappa_{0}}a(x,\kappa)=\overline{a}\,$.\\
\textbf{Case~2:} When $a(\widetilde{x}_{0},\kappa)\leq 0\,$, we have,
$$
\mathcal{E}_{0}(w_{\ell,x_0,\widetilde{x}_{0}},\Fb;a(\widetilde{x}_{0},\kappa),Q_{\ell}(x_0))=\frac{\kappa^2}{2}\int_{Q_{\ell}(x_0)}a(x,\kappa)^{2}\,dx\,.
$$
From \eqref{app-a}, we get the existence of a constant $C>0$ such that for any $\delta\in(0,1)$,
\begin{equation}\label{up-E0-2}
\mathcal{E}_{0}(w_{\ell,x_0,\widetilde{x}_{0}},\Fb;a(\widetilde{x}_{0},\kappa),Q_{\ell}(x_0))\leq (1+\delta)\, \frac{\kappa^{2}}{2} \, a(\widetilde{x}_{0},\kappa)^{2}\ell^{2}
+C\,\delta^{-1}\,\kappa^{2}\ell^{4}L(\kappa)^{2}\,.
\end{equation}
\textbf{The results of cases 1-2}, we obtain,
\begin{multline}
\mathcal{E}_{0}(w_{\ell,x_0,\widetilde{x}_{0}},\Fb;a(\widetilde{x}_{0},\kappa),Q_{\ell}(x_0))\leq (1+\delta)\kappa^{2}\left[a_{+}(\widetilde{x}_{0},\kappa)^{2}\hat{f}\left(\frac{H\,|B_{0}(\widetilde{x}_{0})|}{\kappa\,a_{+}(\widetilde{x}_{0},\kappa)}\right)+\frac{1}{2}a_{-}(\widetilde{x}_{0},\kappa)^{2}\right]\ell^{2}\\
+C\left(\frac{\kappa}{\ell}\,\overline{a}^{\frac{3}{2}}+\delta^{-1}\kappa^{2}\ell^{2}L(\kappa)^{2}+\delta^{-1}\kappa^{4}\ell^{4}\,\overline{a}\right)\ell^{2}\,,
\end{multline}
which finishes the proof of Proposition~\ref{pp-up-Eg}.
\end{proof}
\begin{app}\label{app:1}~\\
We select $\ell,\,\rho,\,\delta$ and the constraint on $L(\kappa)$ as follows:
\begin{equation}\label{choice-ell-rho}
\ell=\kappa^{-\frac{7}{12}}\,,\qquad\rho=\kappa^{-\frac{17}{24}}\,,\qquad L(\kappa)\leq C\,\kappa^{\frac{1}{2}}\,.
\end{equation}
and
\begin{equation}\label{choice-delta}
\delta=\kappa^{-\frac{1}{12}}
\end{equation}
Under Assumption~\eqref{cond-H}, this choice permits to verify the assumptions in Proposition~\ref{pp-up-Eg} and to obtain error terms of order $\textit{o}(\kappa^{2})$. We have indeed as $\kappa\longrightarrow\infty$
$$\frac{\kappa}{\ell}=\kappa^{\frac{19}{12}}\ll\kappa^{2}\,,$$
$$\delta^{-1}\kappa^{2}\ell^{2}\,L(\kappa)^{2}\leq\kappa^{\frac{23}{12}}\ll\kappa^{2}\,,$$
$$\delta^{-1}\kappa^{4}\ell^{4}=\kappa^{\frac{21}{12}}\ll \kappa^{2}\,,$$
$$\ell^{2}\kappa H \rho=\kappa^{\frac{3}{24}}\gg 1\,.$$
\end{app}
\begin{theorem}\label{up-Eg}
Under Assumptions~\eqref{B(x)}-\eqref{a4}, if \eqref{cond-H} holds and $L(\kappa)\leq C\,\kappa^{\frac{1}{2}}$, then, the ground state energy $\E0(\kappa,H,a,B_{0})$ in \eqref{eq-2D-gs} satisfies
\begin{multline}\label{up-Eg-eq1}
\E0(\kappa,H,a,B_{0})\leq \kappa^{2}\int_{\{a(x,\kappa)>0\}} a(x,\kappa)^{2}\,\hat{f}\left(\frac{H\,|B_{0}(x)|}{\kappa\,a(x,\kappa)}\right)\,dx\\
+\frac{\kappa^{2}}{2}\int_{\{a(x,\kappa)\leq 0\}} a(x,\kappa)^{2}\,dx+\textit{o}(\kappa^{2})\,,\quad{\rm as}~\kappa\longrightarrow\infty\,.
\end{multline}
\end{theorem}
\begin{proof}
Let $\ell\in(0,1)$, $\delta$ and $\rho$ be chosen as in \eqref{choice-ell-rho} and \eqref{choice-delta}.
We consider the lattice $\Gamma_{\ell}:=\ell\Z\times\ell\Z$ and write, for $\gamma \in\Gamma_{\ell}$,  $ Q_{\gamma,\ell}=Q_{\ell}(\gamma)$. In the next decomposition we keep the $\rho$-admissible boxes $Q_\ell(\gamma)$ in $\Omega$ which in addition are either contained in $\{a >0\}$ or in $\{ a \leq 0\}$. Hence we introduce \begin{equation}\label{def-card-squares}
\mathcal I_{\ell, \rho}^{+}=\left\{\gamma; \,\overline{Q_{\gamma, \ell}}\subset \Omega\cap\left\{|B_{0}|>\rho\,; a>0\right\}\right\},\quad \mathcal I_{\ell, \rho}^{-}=\left\{\gamma; \,\overline{Q_{\gamma, \ell}}\subset \Omega\cap\left\{ |B_{0}|>\rho\,; a\leq 0\right\}\right\}\,,
\end{equation}
and
\begin{equation}
N^{+}={\rm card}~ \mathcal I_{\ell, \rho}^{+}\,,\qquad N^{-}={\rm card}~ \mathcal I_{\ell, \rho}^{-}\,.
\end{equation}
Under Assumption \eqref{a4}, we have,
\begin{equation}\label{N}
N^{+}+N^{-}=|\Omega|\ell^{-2}+\mathcal{O}(\kappa^{\frac{1}{2}}\ell^{-1}+\ell^{-1} + \rho\ell^{-2})\,,\qquad{\rm as}~\kappa\rightarrow+\infty\,.
\end{equation}
In \eqref{N},  $\kappa^{\frac{1}{2}}\ell^{-1}$ appears when treating the boundary of the set $\{a(x,\kappa)>0\}$ (using Assumption $(A_4)$ as explained in \eqref{defA4}), $\ell^{-1}$ appears in the treatment of the boundary and $\rho\ell^{-2}$ appears when treating the neighborhood of $\Gamma$.\\
In each  $\rho$-admissible $Q_\ell(\gamma)$, we consider some $\widetilde \gamma$ (to be chosen later) such that  $(\ell, \gamma,\widetilde{\gamma})$ be a $\rho$-admissible triple. We consider $w_{\ell,\gamma,\widetilde{\gamma}}$ and extend it by $0$ outside of $Q_{\gamma,\ell}$, keeping the same notation for this extension. Then we define
\begin{equation}
s(x)=
\sum_{\gamma\in\mathcal{I}_{\ell,\rho}^{+}\cup\mathcal{I}_{\ell,\rho}^{-}} \,w_{\ell,\gamma,\widetilde{\gamma}}(x)\,.
\end{equation}
We compute the Ginzburg-Landau energy of the test configuration $(s,\Fb)$ in $\Omega$. Since $\curl \Fb=B_{0}\,$, we get,
\begin{align}\label{sumE1}
\mathcal{E}_{\kappa,H,a,B_{0}}(s,\Fb,\Omega)=\sum_{\gamma\in\mathcal{I}_{\ell,\rho}^{+}\cup\mathcal{I}_{\ell,\rho}^{-}}\mathcal{E}_{0}(w_{\ell,\gamma,\widetilde{\gamma}},\Fb;a(\widetilde{\gamma},\kappa),Q_{\gamma, \ell})\,.
\end{align}
Notice that for any $\widetilde{\gamma}\in Q_{\gamma,\ell}\,$, $a(\widetilde{\gamma},\kappa)$ satisfies \eqref{app-a} with $x=\gamma$ and $\widetilde{x}_{0}=\widetilde{\gamma}\,$, and $B_{0}(\widetilde{\gamma})$ satisfies \eqref{F-A}. We recall that $\hat{f}$ is a continuous, non-decreasing function (see \cite[Theorem~2.1]{KA2}) and that $B_{0}$ and  $ a(\cdot,\kappa)$ are in $C^{1}$. Then, in each box $Q_{\gamma,\ell}$,  we select $\widetilde{\gamma}\in \overline{Q_{\gamma,\ell}}$ such that
$$
|a(\widetilde{\gamma},\kappa)|^{2}\,\hat{f}\left(\frac{H\,B_{0}(\widetilde{\gamma})}{\kappa\,a(\widetilde{\gamma},\kappa)}\right)=\inf_{\widehat{\gamma}\in Q_{\gamma,\ell}} |a(\widehat{\gamma},\kappa)|^{2}\,\hat{f}\left(\frac{H\,B_{0}(\widehat{\gamma})}{\kappa\,a(\widehat{\gamma},\kappa)}\right)\quad(\text{if}~ \gamma\in \mathcal I_{\ell, \rho}^{+})
$$
and
$$
|a(\widetilde{\gamma},\kappa)|^{2}=\inf_{\widehat{\gamma}\in Q_{\gamma,\ell}} |a(\widehat{\gamma},\kappa)|^{2}\quad(\text{if}~ \gamma\in \mathcal I_{\ell, \rho}^{-})\,.
$$
Using Proposition~\ref{pp-up-Eg} and noticing that $|Q_{\gamma,\ell}|=\ell^2$, we get the existence of $C>0$ such that, for any $\delta\in(0,1)$
\begin{multline}\label{sumE2}
\sum_{\gamma\in\mathcal{I}_{\ell,\rho}^{+}\cup\mathcal{I}_{\ell,\rho}^{-}}\mathcal{E}_{0}(w_{\ell,\gamma,\widetilde{\gamma}},\Fb;a(\widetilde{\gamma},\kappa),Q_{\gamma, \ell})\leq \kappa^{2}(1+\delta)\sum_{\gamma\in\mathcal{I}_{\ell,\rho}^{+}} \inf_{\widehat{\gamma}\in Q_{\gamma,\ell}} [a(\widehat{\gamma},\kappa)]_{+}^{2}\,\hat{f}\left(\frac{H\,B_{0}(\widehat{\gamma})}{\kappa\,a(\widehat{\gamma},\kappa)}\right)\ell^{2}\\
+\kappa^{2}(1+\delta)\sum_{\gamma\in\mathcal{I}_{\ell,\rho}^{-}} \inf_{\widehat{\gamma}\in Q_{\gamma,\ell}} \frac{[a(\widehat{\gamma},\kappa)]_{-}^{2}}{2}\ell^{2}+\, C\, \sum_{\gamma\in\mathcal{I}_{\ell,\rho}^{+}\cup\mathcal{I}_{\ell,\rho}^{-}}\left(\frac{\kappa}{\ell}+\delta^{-1}\kappa^{2}\ell^{2}L(\kappa)^{2}+\delta^{-1}\kappa^{4}\ell^{4}\right)\,\ell^{2}\,.
\end{multline}
We recognize the lower Riemann sum of the function $x\longmapsto [a(x,\kappa)]_+^{2}\,\hat{f}\left(\frac{H\,B_{0}(x)}{\kappa\,a(x,\kappa)}\right)$ in $(\cup_{\gamma\in\mathcal{I}^{+}_{\ell,\rho}}Q_{\gamma,\ell})$ and the function $x\longmapsto [a(x,\kappa)]_-^{2}$ in $(\cup_{\gamma\in\mathcal{I}^{-}_{\ell,\rho}} Q_{\gamma,\ell})$ . Notice that $\{\cup_{\gamma\in \mathcal{I}_{\ell, \rho}}Q_{\gamma, \ell}\}\subset\Omega$. Thanks to Application~\ref{app:1}, using \eqref{N}  and the non negativity of $\hat f$, we get by collecting \eqref{sumE1}-\eqref{sumE2} that,
\begin{equation}\label{final-est}
\mathcal{E}_{\kappa,H,a,B_{0}}(s,\Fb,\Omega)\leq\kappa^{2}\int_{\{a(x,\kappa)>0\}} a(x,\kappa)^{2} \hat{f}\left(\frac{H}{\kappa}\frac{|B_{0}(x)|}{a(x,\kappa)}\right)\,dx+\frac{\kappa^{2}}{2}\int_{\{a(x,\kappa)\leq 0\}}a(x,\kappa)^{2}\,dx+C\,\kappa^{\frac{23}{12}}\,.
\end{equation}
Since $(\psi,\Ab)$ is a minimizer of the functional $\mathcal E_{\kappa,H,a,B_{0}}$ in \eqref{eq-2D-GLf}, we get
$$
\E0(\kappa,H,a,B_{0})\leq \mathcal{E}_{\kappa,H,a,B_{0}}(s,\Fb,\Omega)\,.
$$
This finishes the proof of Theorem~\ref{up-Eg}.
\end{proof}
\section{A priori estimates of minimizers}\label{section:P.E.}
The aim of this section is to give a priori estimates for the solutions of the Ginzburg-Landau equations \eqref{eq-2D-GLeq}. In the case when $a(x,\kappa)=1$ the starting point is an $L^{\infty}$ estimate of $\psi$. This estimate can be easly extended in the general case considered in this paper when $\eqref{eq-2D-GLeq}_{a}$ and $\eqref{eq-2D-GLeq}_{c}$ hold. Let us introduce:
\begin{equation}
\overline{a}(\kappa)= \sup_{x\in\overline{\Omega}}a(x,\kappa)\,.
\end{equation}
\begin{prop}\label{prop-psi<a}
Let $\kappa>0$; if $(\psi,\Ab)$  is a critical point (see \eqref{eq-2D-GLeq}), then,
\begin{equation}\label{eq-psi<supa}
 |\psi(x)|^{2}\leq \max\left\{\overline{a}(\kappa),0\right\}\,,\qquad\forall x\in\overline{\Omega}\,.
\end{equation}
\end{prop}
\begin{proof}
We distinguish between two cases:\\
\textbf{Case 1:} \textbf {$\displaystyle\overline{a}(\kappa)\leq 0\,$.}\\
 Multiplying the equation for $\psi$ in \eqref{eq-2D-GLeq}$_a$ by $\overline{\psi}$ and integrating over $\Omega$, we get
\begin{equation}\label{eq:A-psi}
\int_{\Omega}|(\nabla-i\kappa H\Ab)\psi|^{2}\,dx=\kappa^{2}\int_{\Omega}(a(x,\kappa)-|\psi|^{2})|\psi|^{2}\,dx\,.
\end{equation}
Since $(a(x,\kappa)-|\psi|^{2})\leq -|\psi|^{2}$, we obtain that $ |\psi|^{2}=0\,$ almost everywhere.\\
\textbf{Case 2:} \textbf{ $\displaystyle\overline{a}(\kappa)>0\,$}.\\
We will show that $\psi\in C^{0}(\overline{\Omega})$. In fact, $(\psi,\Ab)$ satisfies  \eqref{eq-2D-GLeq}$_{a}$, $\psi\in L^{p}(\Omega)$ for all $2\leq p<+\infty$ and $\Ab\in H_{\rm div}^{1}(\Omega)\hookrightarrow L^{p}(\Omega)$. Thus, $\psi \in W^{2,q}(\Omega)$ for all $q<2$. As a consequence of the continuous Sobolev embedding of $W^{j+m,q}(\Omega)$ into $C^{j}(\overline{\Omega})$ for any $q>\frac{2}{m}$, we obtain that $\psi\in C^{0}(\overline{\Omega})$. Define for any $\kappa>0$ the following open set:
\begin{equation}\label{omega+}
\Omega_{+}=\left\{x\in\Omega:\,|\psi(x)|>\sqrt{\overline{a}(\kappa)}\right\}\,,
\end{equation}
and the following functions on $\Omega_{+}$
$$\phi=\frac{\psi}{|\psi|}\,,\qquad \widehat{\psi}=\left[|\psi|-\sqrt{\overline{a}(\kappa)}\right]_{+}\phi\,.$$
It is clear that
$$\nabla\left[|\psi|-\sqrt{\overline{a}(\kappa)}\right]_{+}=1_{\Omega_{+}}\nabla\left(|\psi|-\sqrt{\overline{a}(\kappa)}\right)=1_{\Omega_{+}}\nabla|\psi|\,.$$
Notice that $\psi\in H^{1}(\Omega)$, so applying \cite[Proposition~3.1.2]{FH1}, we get the property that\break  $\displaystyle\nabla\left[|\psi|-\sqrt{\overline{a}(\kappa)}\right]_{+}\in L^{2}(\Omega)$, which implies that $\displaystyle\left[|\psi|-\sqrt{\overline{a}(\kappa)}\right]_{+}\in H^{1}(\Omega)$.\\
We introduce an increasing cut-off function $\chi\in C^{\infty}(\R)$ such that,
\begin{equation}\label{chi}
\chi(t)=\left\{
\begin{array}{ll}
0 & \text{for }~t\leq\frac{1}{4}\displaystyle\sqrt{\overline{a}(\kappa)}\\
1& \text{for}~t\geq\frac{3}{4}\displaystyle\sqrt{\overline{a}(\kappa)}\,,
\end{array}
\right.
\end{equation}
and define
\begin{equation}\label{phi1}
\widehat{\phi}=\chi(|\psi|)\frac{\psi}{|\psi|}\,.
\end{equation}
Since $\chi(|\psi|)\frac{\psi}{|\psi|}$ is smooth with bounded derivatives and $\psi\in H^{1}(\Omega)$, the chain rule gives that $\widehat{\phi}\in H^{1}(\Omega)\,.$
Furthermore,
\begin{equation}\label{phi2}
(\nabla-i\kappa H \Ab)\widehat{\psi}=1_{\Omega_{+}}\widehat{\phi}\,\nabla|\psi|+\displaystyle\left[|\psi|-\sqrt{\overline{a}(\kappa)}\right]_{+}(\nabla-i\kappa H\Ab)\widehat{\phi}.
\end{equation}
Using \eqref{chi} and \eqref{phi1}, we get
\begin{equation}\label{phi3}
1_{\Omega_{+}}(\nabla-i\kappa H \Ab)\psi=1_{\Omega_{+}}(\nabla-i\kappa H\Ab)(|\psi|\widehat{\phi})=1_{\Omega_{+}}\{\widehat{\phi}\,\nabla|\psi|+|\psi|(\nabla-i\kappa H\Ab)\widehat{\phi}\}\,.
\end{equation}
We have on $\Omega_{+}$ that $|\phi|=|\widehat{\phi}|=1\,$. Therefore
\begin{align*}
\phi\nabla\overline{\phi}+\overline{ \phi\nabla\overline{\phi}}&=\phi\nabla\overline{\phi}+\overline{\phi}\nabla\phi\\
&=\nabla|\phi|^{2}\\
&=0\,.
\end{align*}
So, $\RE(1_{\Omega_{+}}\phi\nabla\overline{\phi})=0\,$. This implies by using \eqref{phi2} and \eqref{phi3} that
$$
\RE\left\{\overline{(\nabla-i\kappa H\Ab)\widehat{\psi}}\cdot(\nabla-i\kappa H\Ab)\psi \right\}=
1_{\Omega_{+}}\left(|\nabla|\psi||^{2}+\left(|\psi|-\sqrt{\overline{a}(\kappa)}\right)|\psi||(\nabla-i\kappa H\Ab)\widehat{\phi}|^{2}\right).
$$
Multiplying $\eqref{eq-2D-GLeq}_{a}$ by $\overline{\widehat{\psi}}$ and using $\eqref{eq-2D-GLeq}_{c}$, it results from an integration by parts over $\Omega$ that
\begin{align*}
0&=\RE\left\{\int_{\Omega} \overline{(\nabla-i\kappa H\Ab)\widehat{\psi}}(\nabla-i\kappa H\Ab)\psi+\overline{\widehat{\psi}}(|\psi|^{2}-a)\psi\,dx \right\}\\
&\geq \RE\left\{\int_{\Omega} \overline{(\nabla-i\kappa H\Ab)\widehat{\psi}}(\nabla-i\kappa H\Ab)\psi+\overline{\widehat{\psi}}\left(|\psi|^{2}-\overline{a}(\kappa)\right)\psi\,dx \right\}\\
&\geq \int_{\Omega_{+}} |\nabla|\psi||^{2}+\left(|\psi|-\overline{a}(\kappa)\right)|\psi||(\nabla-i\kappa H\Ab)\widehat{\phi}|^{2}\\
&\qquad\qquad\qquad\qquad+\left(|\psi|+\sqrt{\overline{a}(\kappa)}\right)\left(|\psi|-\sqrt{\overline{a}(\kappa)}\right)^{2}|\psi|\,dx\,.
\end{align*}
Since the integrand is non-negative in $\Omega_{+}$, we easily conclude that $\Omega_{+}$ has measure zero, and consequently, we get that $|\psi|\in L^{\infty}(\Omega)$\,.\\
Since $\Omega_{+}$ has measure zero and $\psi\in C^{0}(\overline{\Omega})$, we get
$$
|\psi(x)|^{2}\leq  \overline{a}(\kappa) \,,\qquad\forall x\in\overline{\Omega}\,.
$$
\end{proof}

\begin{corol}
Let $\kappa>0$; If $(\psi,\Ab)\in H^1(\Omega;\C)\times H^1_{\Div}(\Omega)$ is a  critical point, we have,
\begin{equation}\label{eq-psi<a}
 |\psi(x)|^{2}\leq \max\left\{\overline{a},0\right\}\,,\qquad\forall x\in\overline{\Omega}\,,
\end{equation}
where $ \overline{a}= \sup_{\kappa} \overline{a}(\kappa)$ was introduced in \eqref{def:sup-a}.
\end{corol}

The following estimates play an essential role in controlling the errors resulting from various approximations (see  Section~\ref{section5}). These estimates are simpler than the delicate elliptic estimates in \cite{FH2} and \cite{LP}.
\begin{prop}\label{pr-est}
Suppose that \eqref{cond-H} holds. Let $\beta\in(0,1)$. There exist positive constants $\kappa_{0}$ and $C$ such that, if $\kappa\geq\kappa_{0}$ and $(\psi, \Ab)$ is a minimizer of \eqref{eq-2D-GLf}, then
\begin{equation}\label{eq-curlAF}
\|\curl(\Ab-\Fb)\|_{L^{2}(\Omega)}\leq \frac{C}{H}\,.
\end{equation}
\begin{equation}\label{est-A-F}
\|\Ab-\Fb\|_{H^{2}(\Omega)}\leq \frac{C}{H}\,,
\end{equation}
\begin{equation}\label{est-A-F2}
\|\Ab-\Fb\|_{C^{0,\beta}(\overline{\Omega})}\leq  \frac{C}{H}\,.
\end{equation}
Here  we recall that $\Fb$ is the magnetic potential  defined by
\begin{equation}\label{div-curlF}
\curl \Fb = B_0\,,\, \Fb \in H^1_{\Div}(\Omega)\,.
\end{equation}
\end{prop}
\begin{proof}
Under Assumption~\eqref{cond-H}, Theorem~\ref{up-Eg} yields
\begin{align}
&\|\curl(\Ab-\Fb)\|_{L^{2}(\Omega)}\leq\frac{1}{\kappa H}\E0(\kappa,H,a,B_{0})^{\frac{1}{2}}\nonumber\\
&\qquad\qquad\leq\frac{1}{\kappa H}\left(\kappa^{2}\,\int_{\{a(x,\kappa)>0\}} \,a(x,\kappa)^{2} \hat{f}\left(\frac{H}{\kappa}\frac{|B_{0}(x)|}{a(x,\kappa)}\right)\,dx+\frac{\kappa^{2}}{2}\int_{\{a(x,\kappa)\leq 0\}}\,a(x,\kappa)^{2}\,dx\right)^{\frac{1}{2}}\,.
\end{align}
Using \eqref{a2} and the bound $\hat{f}(b)\leq\frac{1}{2}$, we get,
\begin{equation}\label{est-curlAF}
\|\curl(\Ab-\Fb)\|_{L^{2}(\Omega)}\leq\frac{C}{H}\,.
\end{equation}
As in \cite[Proposition~4.1]{KA2}, we prove that
\begin{equation}\label{est-F2}
\|\Ab-\Fb\|_{H^{2}(\Omega)}\leq\frac{C}{H}\,.
\end{equation}
Now, the estimate in $C^{0,\beta}$-norm is a consequence of the continuous Sobolev embedding of $H^{2}(\Omega)$ in $C^{0,\beta}(\overline{\Omega})$.
\end{proof}

\section{Lower bounds for the  global  and local energies} \label{section5}
In this section, we suppose that $\mathcal{D}$ is an open set with smooth boundary such that $\overline{\mathcal{D}}\subset\Omega$ (or $\mathcal{D}=\Omega$). We will give a lower bound of the ground state energy $\E0(\kappa,H,a,B_{0})$ introduced in \eqref{eq-2D-gs}.
\begin{prop}\label{prop-lb}
 Under Assumptions~\eqref{B(x)}-\eqref{a3}, there exist for all $\beta\in(0,1)$ positive
constants $C$ and $\kappa_{0}$ such that if $\kappa\geq\kappa_{0}$, $\ell \in (0,\frac 12)$, $\delta\in(0,1)$, $\rho>0$, $\ell^{2}\kappa H \rho>1$, $(\psi,\Ab)$ is a minimizer of \eqref{eq-2D-GLf}, $h\in C^{1}(\overline{\Omega})$, $\|h\|_{\infty}\leq 1$ and $(\ell,x_{0},\widetilde{x}_{0})$ is a $\rho$-admissible triple, then,
\begin{multline}\label{lb}
\frac{1}{|Q_{\ell}(x_{0})|}\mathcal E_0(h\psi,\Ab;a,Q_{\ell}(x_0))\geq(1-\delta)\kappa^{2}\left\{ a_{+}(\widetilde{x}_{0},\kappa)^{2}\hat{f}\left(\frac{H}{\kappa}\frac{|B_{0}(\widetilde{x}_{0})|}{a_{+}(\widetilde{x}_{0},\kappa)}\right)+\frac{1}{2}a_{-}(\widetilde{x}_{0},\kappa)^{2}\right\}\\
-C\kappa^{2}\left(\delta^{-1}\ell^{2}L(\kappa)^{2}+\delta^{-1}\kappa^{2}\ell^{4}+\delta^{-1}\ell^{2\beta}+(\kappa\ell)^{-1}+\ell\,L(\kappa)\right)\,,
\end{multline}
where $L(\kappa)$ is introduced in \eqref{def:L}.
\end{prop}

\begin{proof}
We distinguish between two cases according to the sign of $a(\widetilde{x}_{0},\kappa)$.\\
\textbf{We begin with the case when $a(\widetilde{x}_{0},\kappa)\leq 0\,$.} We have,
\begin{align*}
\mathcal E_0(h\psi,\Ab;a,Q_{\ell}(x_0))&=\int_{Q_{\ell}(x_0)}|(\nabla -i\kappa
H\Ab)h\psi|^2\,dx+\frac{\kappa^{2}}{2}\int_{Q_{\ell}(x_0)}(a(x,\kappa)-|h\psi|^{2})^2\,dx\\
&\geq\frac{\kappa^{2}}{2}\int_{Q_{\ell}(x_0)} a(x,\kappa)^{2}\,dx-\kappa^{2}\int_{Q_{\ell}(x_0)}a(x,\kappa)|h\psi|^{2}\,dx\,.
\end{align*}
Using \eqref{app-a}, \eqref{eq-psi<a} and the assumptions on $h$, the simple decomposition $a(x,\kappa)=a(\widetilde{x}_{0},\kappa)+(a(x,\kappa)-a(\widetilde{x}_{0},\kappa))$ yields for any $\delta\in(0,1)$
\begin{align}\label{up-a2}
\frac{\kappa^{2}}{2}\int_{Q_{\ell}(x_0)} a(x,\kappa)^{2}\,dx&\geq (1-\delta)\frac{\kappa^{2}}{2}\int_{Q_{\ell}(x_0)} a(\widetilde{x}_{0},\kappa)^{2}\,dx\nonumber\\
&\qquad\qquad\qquad+(1-\delta^{-1})\frac{\kappa^{2}}{2}\int_{Q_{\ell}(x_0)} (a(x,\kappa)-a(\widetilde{x}_{0},\kappa))^{2}\,dx\nonumber\\
&\geq (1-\delta)\,\frac{\kappa^{2}}{2}\,a(\widetilde{x}_{0},\kappa)^{2}\,|Q_{\ell}(x_0)|-C\delta^{-1}\kappa^{2}\ell^{2}L(\kappa)^{2}\,|Q_{\ell}(x_0)|\,,
\end{align}
and
\begin{align}\label{up-apsi}
-\kappa^{2}\int_{Q_{\ell}(x_0)} a(x,\kappa)|h\psi|^{2}\,dx&\geq -\kappa^{2}\int_{Q_{\ell}(x_0)} a(\widetilde{x}_{0},\kappa)|h\psi|^{2}\,dx-C\,\ell\,L(\kappa)\,\kappa^{2}\,|Q_{\ell}(x_0)|\nonumber\\
&\geq-C\,\ell\,L(\kappa)\,\kappa^{2}\,|Q_{\ell}(x_0)|\,.
\end{align}
Collecting \eqref{up-a2} and \eqref{up-apsi}, we get,
\begin{equation}\label{es of A2}
\frac{1}{|Q_{\ell}(x_{0})|}\mathcal E_0(h\psi,\Ab;a,Q_{\ell}(x_0))\geq (1-\delta)\,\frac{\kappa^{2}}{2}\,a(\widetilde{x}_{0},\kappa)^{2}-C\delta^{-1}\kappa^{2}\ell^{2}L(\kappa)^{2}-C'\,\ell\,L(\kappa)\,\kappa^{2}\,.
\end{equation}
\textbf{Now, we treat the case when $a(\widetilde{x}_{0},\kappa)>0\,$.} Let $\phi_{x_{0}}(x)=(\Ab(x_0)-\Fb(x_0))\cdot x$, where $\Fb$ is the magnetic potential introduced in \eqref{div-curlF}. Using the estimate of $\|\Ab-\Fb\|_{C^{0,\beta}(\Omega)}$ given in Proposition~\ref{pr-est}, we get for any $\beta\in (0,1)$ the existence of a constant $C$ such that for all $x\in Q_{\ell}(x_0)$,
\begin{equation}\label{alpha}
|\Ab(x)-\nabla\phi_{x_0}-\Fb(x)|\leq C\,\frac{\ell^{\beta}}{H}\,.
\end{equation}
Let $\widetilde{x}_{0}\in\overline{ Q_{\ell}(x_{0})}$ and $\varphi=\varphi_{x_{0},\widetilde{x}_{0}}+\phi_{x_{0}}$ with $\varphi_{x_{0},\widetilde{x}_{0}}$ satisfying \eqref{F-A}. We define the function in $Q_{\ell}(x_{0})$,
\begin{equation}\label{defu}
u(x)=e^{-i\kappa H\varphi}h\psi(x)\,.
\end{equation}
Similarly to \eqref{up-a}, we have, for any $\delta\in(0,1)$,
\begin{align}\label{lw-a}
\frac{\kappa^{2}}{2}\int_{Q_{\ell}(x_{0})}\left(a(x,\kappa)-|h\psi|^{2}\right)^{2}\,dx\geq(1-\delta)\frac{\kappa^{2}}{2}\int_{Q_{\ell}(x_{0})}\left(a(\widetilde{x}_{0},\kappa)-|h\psi|^{2}\right)^{2}\,dx-C\delta^{-1}\kappa^{2}\ell^{4}L(\kappa)^{2}\,.
\end{align}
Using the same techniques as in \cite[Lemma~4.1]{KA}, we get, for any $\beta\in(0,1)$,
\begin{multline}\label{lw-A}
\int_{Q_{\ell}(x_{0})}|(\nabla-i\kappa H \Ab)h\psi|^{2}\,dx\geq (1-\delta)\int_{Q_{\ell}(x_{0})}|(\nabla-i\kappa H(\zeta_{\ell}|B_{0}(\widetilde{x}_{0})|\Ab_{0}(x-x_{0})+\nabla\varphi(x)))h\psi|^{2}\,dx\\
-C\delta^{-1}(\kappa H)^{2}\left(\ell^{4}+\frac{\ell^{2\beta}}{H^{2}}\right)\int_{Q_{\ell}(x_0)}|h\psi|^2\,dx\,.
\end{multline}
Thus, by collecting \eqref{lw-a} and \eqref{lw-A}, using \eqref{a3}, \eqref{eq-psi<a} and $\|h\|_{L^{\infty}(\Omega)}\leq 1$, we get
\begin{multline}\label{es of A}
\mathcal E_0(h\psi,\Ab;a(\widetilde{x}_{0},\kappa),Q_{\ell}(x_0))\geq (1-\delta)\mathcal
E_0(e^{-i\kappa H\varphi}h\psi(x),\zeta_{\ell}|B_{0}(\widetilde{x}_{0})|\Ab_{0}(x-x_{0});a(\widetilde{x}_{0},\kappa),Q_{\ell}(x_0))\\
-C\delta^{-1}\kappa^{2}\ell^{4}L(\kappa)^{2}
-C_{1}\delta^{-1}\kappa^{2}H^{2}\left(\ell^{4}+\frac{\ell^{2\beta}}{H^{2}}\right)\ell^{2}\,.
\end{multline}
Let $R$ and $b$ be as in \eqref{def-Rb}. Let us introduce the function $v_{\ell,x_0,\widetilde{x}_{0}}$ in $Q_{R}$ as follows:
\begin{equation}\label{def-v}
v_{\ell,x_0,\widetilde{x}_{0}}(x)=\begin{cases}
u\left(\frac{\ell}{R} x+x_{0}\right)&{\rm if}~x\in Q_{R}\subset\{B_{0}>\rho\}\cap\Omega\\
\overline{u}\left(\frac{\ell}{R} x+x_{0}\right)&{\rm if}~x\in Q_{R}\subset\{B_{0}<-\rho\}\cap\Omega\,,
\end{cases}
\end{equation}
where $u$ is defined in \eqref{defu}.\\
Similarly to \eqref{2up-loc-en}, we use the change of variable $y=\frac{R}{\ell}(x-x_{0})$ and get
\begin{equation}\label{E0<}
\mathcal E_0(e^{-i\kappa H\varphi}h\psi(x),\zeta_{\ell}\,\kappa H|B_{0}(\widetilde{x}_{0})|\Ab_{0}(x-x_0);a(\widetilde{x}_{0},\kappa),Q_{\ell}(x_0))=\frac{1}{b} F^{+1,a(\widetilde{x}_{0},\kappa)}_{b,Q_{R}}(v_{\ell,x_{0},\widetilde{x}_{0}})\,,
\end{equation}
where $F^{+1,a(\widetilde{x}_{0},\kappa),}_{b,Q_{R}}$ is introduced in \eqref{eq-GL-F}.\\
Since $v_{\ell,x_{0},\widetilde{x}_{0}}\in H^{1}(Q_{R})$ then, using \eqref{eN>eD} and \eqref{est-f(b)-eD}, we get
\begin{align}\label{F>f}
\frac{1}{b}F^{+1,a(\widetilde{x}_{0},\kappa)}_{b,Q_{R}}(v_{\ell,x_{0},\widetilde{x}_{0}})&\geq \frac{1}{b} e_{N}\left(b,R,a(\widetilde{x}_{0},\kappa)\right)\nonumber\\
&\geq \frac{1}{b}e_{D}\left(b,R,a(\widetilde{x}_{0},\kappa)\right)-C_{M}\,a(\widetilde{x}_{0},\kappa)^{\frac{3}{2}}\frac{R}{\sqrt{b}}\nonumber\\
&\geq a(\widetilde{x}_{0},\kappa)^{2}\frac{R^{2}}{b}\hat{f}\left(\frac{b}{a(\widetilde{x}_{0},\kappa)}\right)-\widehat{C}_{M}\,\frac{R}{\sqrt{b}}\,.
\end{align}
Inserting \eqref{F>f} into \eqref{E0<}, we get
\begin{multline}\label{E0<2}
\mathcal E_0(e^{-i\kappa H\varphi}h\psi(x),\zeta_{\ell}\,\kappa H|B_{0}(\widetilde{x}_{0})|\Ab_{0}(x-x_0);a(\widetilde{x}_{0},\kappa),Q_{\ell}(x_0))\geq a(\widetilde{x}_{0},\kappa)^{2}\frac{R^{2}}{b}\hat{f}\left(\frac{b}{a(\widetilde{x}_{0},\kappa)}\right)\\
-\widehat{C}_{M}\frac{R}{\sqrt{b}}\,.
\end{multline}
Having in mind \eqref{def-Rb} and \eqref{E0<2}, we get from \eqref{es of A},
\begin{multline}\label{E0<4}
\frac{1}{|Q_{\ell}(x_{0})|}\mathcal E_0(h\psi,\Ab;a(\widetilde{x}_{0},\kappa),Q_{\ell}(x_0))\geq(1-\delta)\kappa^{2}a(\widetilde{x}_{0},\kappa)^{2}\hat{f}\left(\frac{H}{\kappa}\frac{|B_{0}(\widetilde{x}_{0})|}{a(\widetilde{x}_{0},\kappa)}\right)\\
-C\delta^{-1}\kappa^{2}\ell^{2}L(\kappa)^{2}
-C_{1}\delta^{-1}\kappa^{2}H^{2}\left(\ell^{4}+\frac{\ell^{2\beta}}{H^{2}}\right)
-C_{2}\frac{\kappa}{\ell}\,.
\end{multline}
The estimates in \eqref{es of A2} and \eqref{E0<4}  achieve the proof of Proposition~\ref{prop-lb}.
\end{proof}
\begin{app}
We keep the same choice of $\ell$, $\rho$, $L(\kappa)$ and $\delta$ as in \eqref{choice-ell-rho}, \eqref{choice-delta} and choose:
\begin{equation}\label{choice-delta-alpha}
\beta=\frac{3}{4}\,.
\end{equation}
This choice and Assumption~\eqref{cond-H}  permit to have the assumptions  in Proposition~\ref{prop-lb} satisfied and make the error terms in its statement
 of order $\textit{o}(\kappa^{2})$.
We have as $\kappa\longrightarrow\infty\,$,
$$\delta^{-1}\kappa^{4}\ell^{4}=\kappa^{\frac{21}{12}}\ll \kappa^{2}\,,$$
$$\delta^{-1}\kappa^{2}\ell^{2\beta}=\kappa^{\frac{29}{24}}\ll \kappa^{2}\,,$$
$$\delta^{-1}\kappa^{2}\ell^{2}L(\kappa)^{2}=\kappa^{\frac{23}{12}}\ll\kappa^{2}\,,$$
$$\frac{\kappa}{\ell}=\kappa^{\frac{19}{12}}\ll \kappa^{2}\,,$$
$$
\ell\,L(\kappa)\,\kappa^{2}=\kappa^{\frac{23}{12}}\ll\kappa^{2}\,,
$$
$$\ell^{2}\kappa H\rho=\kappa^{\frac{3}{24}}\gg 1\,.$$
\end{app}
The next theorem presents a lower bound of the local energy in a relatively compact smooth domain $\mathcal{D}$ in $\Omega$. We deduce the lower bound of the global energy by replacing $\mathcal{D}$ by $\Omega$.

\begin{theorem}\label{lw-Eg}~\\
Under Assumptions~\eqref{B(x)}-\eqref{a4}, if \eqref{cond-H}  holds, $L(\kappa)\leq C\,\kappa^{\frac{1}{2}}$ with $C>0$, $h\in C^{1}(\overline{\Omega})$, $\|h\|_{\infty}\leq 1$, $(\psi,\Ab)$ is a minimizer of \eqref{eq-2D-GLf} and $\mathcal D$ an open set  in $\Omega$, then as $\kappa\longrightarrow+\infty$,
\begin{multline}\label{fianl-Eg1}
\mathcal{E} (h\psi,\Ab;a,B_{0},\mathcal{D})\geq\mathcal{E}_{0} (h\psi,\Ab;a,\mathcal{D}) \geq \kappa^{2}\int_{\mathcal{D}\cap\{a(x,\kappa)>0\}} a(x,\kappa)^{2}\hat{f}\left(\frac{H}{\kappa}\frac{|B_{0}(x)|}{a(x,\kappa)}\right)\,dx\\
+\frac{\kappa^{2}}{2}\int_{\mathcal{D}\cap\{a(x,\kappa)\leq 0\}} a(x,\kappa)^{2}\,dx+\textit{o}\left(\kappa^{2}\right)\,.
\end{multline}
\end{theorem}
\begin{proof}
The proof is similar to the one in Theorem~\ref{up-Eg} and we keep the same notation. Let
$$
\mathcal{D}^{+}_{\ell,\rho}={\rm int}\left(\cup_{\gamma\in \mathcal I_{\ell, \rho}^{+}} \overline{Q_{\gamma,\ell}}\right)\qquad{\rm and}\qquad\mathcal{D}^{-}_{\ell,\rho}={\rm int}\left(\cup_{\gamma\in \mathcal I_{\ell, \rho}^{-}} \overline{Q_{\gamma,\ell}}\right)\,,
$$
where $\gamma\in \mathcal I_{\ell, \rho}^{+}$ and $\gamma\in \mathcal I_{\ell, \rho}^{-}$ are introduced in \eqref{def-card-squares}.\\
Thanks to Proposition~\ref{prop-lb}, we can easily prove the existence of positive constant $C$ such that for any $\delta\in(0,1)$ and $\beta\in(0,1)$,
\begin{multline*}
\mathcal{E}_{0} (h\psi,\Ab;a,\mathcal{D})\geq \kappa^{2}(1-\delta)\left\{\int_{\mathcal{D}^{+}_{\ell,\rho}\cap\{a(x,\kappa)>0\}} a(x,\kappa)^{2}\hat{f}\left(\frac{H}{\kappa}\frac{|B_{0}(x)|}{a(x,\kappa)}\right)\,dx\right.\\
\left.+\frac{1}{2}\int_{\mathcal{D}^{-}_{\ell,\rho}\cap\{a(x,\kappa)\leq 0\}} a(x,\kappa)^{2}\,dx\right\}-C\,r(\kappa,\ell,\delta,\rho,L(\kappa),\beta)\,,
\end{multline*}
where
\begin{equation}\label{asympr}
r(\kappa,\ell,\delta,\rho,L(\kappa),\beta)=\kappa^{2}\ell+\kappa^{2}\rho+\frac{\kappa}{\ell}+\delta^{-1}\kappa^{2}\ell^{2}L(\kappa)^{2}+\delta^{-1}\kappa^{4}\ell^{4}+\delta^{-1}\kappa^{2}\ell^{2\beta}+\ell\,L(\kappa)\,\kappa^{2}\,.
\end{equation}
Notice that using the regularity of $\partial\mathcal{D}$, \eqref{B(x)} and \eqref{a4} (see \eqref{defA4}), we get the existence of constants $C_{1}>0$ and $C_{2}>0$ such that,
\begin{equation}\label{eq:D/D+-}
\forall \ell\leq C_{2}\,\kappa^{-\frac{1}{2}}\,,\quad \forall \rho\in(0,1)\,,\qquad |\mathcal{D}\setminus\mathcal{\mathcal{D}^{+}_{\ell,\rho}}|+ |\mathcal{D}\setminus\mathcal{\mathcal{D}^{-}_{\ell,\rho}}|\leq C_{1}(\kappa^{\frac{1}{2}}\,\ell+\rho)\,.
\end{equation}
This implies by using \eqref{a3} and the upper bound $\hat{f}\leq \frac{1}{2}$,
\begin{multline}\label{eq:1st-main}
\int_{\mathcal{D}^{+}\cap\{a(x,\kappa)>0\}} a(x,\kappa)^{2}\hat{f}\left(\frac{H}{\kappa}\frac{|B_{0}(x)|}{a(x,\kappa)}\right)\,dx\geq\int_{\mathcal{D}^{+}_{\ell,\rho}\cap\{a(x,\kappa)>0\}} a(x,\kappa)^{2}\hat{f}\left(\frac{H}{\kappa}\frac{|B_{0}(x)|}{a(x,\kappa)}\right)\,dx\\
-\frac{1}{2}\,\overline{a}\,|\mathcal{D}\setminus\mathcal{\mathcal{D}_{\ell,\rho}}|
\end{multline}
and
\begin{equation}\label{eq:2nd-main}
\frac{1}{2}\int_{\mathcal{D}^{-}\cap\{a(x,\kappa)\leq 0\}} a(x,\kappa)^{2}\,dx\geq\frac{1}{2}\int_{\mathcal{D}_{\ell,\rho}\cap\{a(x,\kappa)\leq 0\}} a(x,\kappa)^{2}\,dx-\frac{1}{2}\,\overline{a}\,|\mathcal{D}\setminus\mathcal{\mathcal{D}^{-}_{\ell,\rho}}|\,,
\end{equation}
where $\overline{a}$ is introduced in \eqref{def:sup-a}.\\
Collecting \eqref{eq:1st-main} and \eqref{eq:2nd-main}, using Assumptions  \eqref{a2} and \eqref{eq:D/D+-}, we find that,
\begin{multline}\label{fianl-E0}
\mathcal{E}_{0} (h\psi,\Ab;a,\mathcal{D})\geq \kappa^{2}(1-\delta)\left\{\int_{\mathcal{D}\cap\{a(x,\kappa)>0\}} a(x,\kappa)^{2}\hat{f}\left(\frac{H}{\kappa}\frac{|B_{0}(x)|}{a(x,\kappa)}\right)\,dx\right.\\
\left.+\frac{1}{2}\int_{\mathcal{D}\cap\{a(x,\kappa)\leq 0\}} a(x,\kappa)^{2}\,dx\right\}-C\,\hat r(\kappa,\ell,\delta,\rho,L(\kappa),\beta)\,,
\end{multline}
where $\hat r(\kappa,\ell,\delta,\rho,L(\kappa),\beta)$ satisfies \eqref{asympr}.\\
Under Assumption~\eqref{cond-H}, the choice of the parameters $\rho$, $\ell$, $L(\kappa)$  in \eqref{choice-ell-rho}, $\delta$ in \eqref{choice-delta} and $\beta$ in \eqref{choice-delta-alpha}, implies that all error terms are of lower order compared to $\kappa^{2}$.\\
As a consequence of  \eqref{cond-H}, the inequality \eqref{fianl-E0} becomes as $\kappa\longrightarrow+\infty$
\begin{multline}\label{fianl-E01}
\mathcal{E}_{0} (h\psi,\Ab;a,\mathcal{D})\geq \kappa^{2}\left\{\int_{\mathcal{D}\cap\{a(x,\kappa)>0\}} a(x,\kappa)^{2}\hat{f}\left(\frac{H}{\kappa}\frac{|B_{0}(x)|}{a(x,\kappa)}\right)\,dx+\frac{1}{2}\int_{\mathcal{D}\cap\{a(x,\kappa)\leq 0\}} a(x,\kappa)^{2}\,dx\right\}\\
+\textit{o}(\kappa^{2})\,.
\end{multline}
Moreover, we know that
$$\mathcal{E}(h\psi,\Ab;a,B_{0},\mathcal{D})\geq \mathcal{E}_{0} (h\psi,\Ab;a,\mathcal{D})\,.$$
This achieves the proof of Theorem~\ref{lw-Eg}.
\end{proof}
As we now show, Theorem~\ref{lw-Eg} permits to achieve the proof of two statements presented in the introduction:

\begin{proof}[\textbf{Proof of Corollary~\ref{corol-2D-main}}]~\\
If $(\psi,\Ab)$ is a minimizer of \eqref{eq-2D-GLf}, we have
\begin{equation}\label{eq-glob-en}
\E0(\kappa,H)=\mathcal E_0(\psi,\Ab;a,\Omega) + (\kappa H)^2
\int_{\Omega} |\curl\big(\Ab - \Fb\big)|^2\,dx \,,
\end{equation}
where $\mathcal E_{0}(\psi,\Ab;a,\Omega)$ is defined in \eqref{eq-GLe0}.\\
Using \eqref{eq-2D-thm} and \eqref{fianl-E01} (with $\mathcal{D}=\Omega$), then under Assumption~\eqref{cond-H} as $\kappa\longrightarrow+\infty$
\begin{equation}\label{eq-2D}
\mathcal E_{0}(\psi,\Ab;a,\Omega)=\kappa^{2}\int_{\{a(x,\kappa)>0\}}a(x,\kappa)^{2}\,\hat{f}\left(\frac{H}{\kappa}\frac{|B_{0}(x)|}{a(x,\kappa)}\right)\,dx+\frac{\kappa^{2}}{2}\int_{\{a(x,\kappa)\leq 0\}}a(x,\kappa)^{2}\,dx+\textit{o}\left(\kappa^{2}\right)\,.
\end{equation}
Putting \eqref{eq-2D} and \eqref{eq-2D-thm} into \eqref{eq-glob-en}, we finish the proof of Corollary~\ref{corol-2D-main}.
\end{proof}
~\\
\begin{proof}[\textbf{Proof of Theorem~\ref{lc-en}}.]~\\
Noticing that \eqref{fianl-E01} is valid when $h=1$ and $\mathcal{D}$ replaced by $\mathcal{\overline{D}}^{c}:=\Omega \setminus \overline{\mathcal D}$ for any open domain $\mathcal{D}\subset\Omega$ with smooth boundary, then we get:
\begin{multline}\label{fianl-E01c}
\mathcal{E}_{0} (\psi,\Ab;a,\mathcal{\overline{D}}^{c})\geq \kappa^{2}\left\{\int_{\mathcal{\overline{D}}^{c}\cap\{a(x,\kappa)>0\}} a(x,\kappa)^{2}\hat{f}\left(\frac{H}{\kappa}\frac{|B_{0}(x)|}{a(x,\kappa)}\right)\,dx\right.\\
\left.+\frac{1}{2}\int_{\mathcal{\overline{D}}^{c}\cap\{a(x,\kappa)\leq 0\}} a(x,\kappa)^{2}\,dx\right\}+\textit{o}(\kappa^{2})\,.
\end{multline}
We can decompose $\mathcal{E}_{0} (\psi,\Ab;a,\mathcal{D})$ as follow:
$$
\mathcal{E}_{0} (\psi,\Ab;a,\mathcal{D})=\mathcal{E}_{0} (\psi,\Ab;a,\Omega)-\mathcal{E}_{0} (\psi,\Ab;a,\mathcal{\overline{D}}^{c})\,.
$$
Using \eqref{eq-2D} and \eqref{fianl-E01c}, we get
\begin{multline}\label{fianl-E01-l}
\mathcal{E}_{0} (\psi,\Ab;a,\mathcal{D})\leq \kappa^{2}\left\{\int_{\mathcal{D}\cap\{a(x,\kappa)>0\}} a(x,\kappa)^{2}\hat{f}\left(\frac{H}{\kappa}\frac{|B_{0}(x)|}{a(x,\kappa)}\right)\,dx+\frac{1}{2}\int_{\mathcal{D}\cap\{a(x,\kappa)\leq 0\}} a(x,\kappa)^{2}\,dx\right\}\\
+\textit{o}(\kappa^{2})\,.
\end{multline}
\end{proof}

\section{study of examples}\label{examples}
In this section, we will describe situations where the remainder term in \eqref{eq-2D-thm} is indeed small as $\kappa \rightarrow +\infty$ compared with the leading order term
\begin{equation}\label{def:leading-term}
 E_{\rm g}^{\textbf{L}}(\kappa,H,a,B_{0}):=\kappa^{2} \left(\int_{\{a(x,\kappa)>0\}}a(x,\kappa)^{2}\,\hat{f}\left(\sigma\frac{|B_{0}(x)|}{a(x,\kappa)}\right)\,dx\\
 +\frac{1}{2}\int_{\{a(x,\kappa)\leq 0\}}a(x,\kappa)^{2}\,dx\right)\,,
\end{equation}
where,
 \begin{equation}\label{def:sigma}
\sigma=\frac{H}{\kappa}\,.
\end{equation}
Note that $0<\lambda_{\min}\leq\sigma\leq\lambda_{\max}$, so that $\sigma$ will be considered as an independent parameter in  $[\lambda_{\min}\,,\,\lambda_{\max}]$.\\
We will also explore, case by case how one can verify Assumption $(A_4)$ as formulated precisely in \eqref{defA4}.

\subsection{The case of a $\kappa$-independent pinning}~\\

\begin{prop}
Suppose \eqref{B(x)} and \eqref{cond-H} hold. Let  $a(x,\kappa)=a(x)$ where $a(x)\in C^{1}(\overline{\Omega})$ is a function independent of $\kappa$ and satisfies,
\begin{equation}\label{cond-a1}
\left\{
\begin{array}{lll}
\{x\in\Omega:\,a(x)>0\}\neq\varnothing\,,\\
{\rm or}\\
 \{x\in\Omega:\,a(x)<0\}\neq\varnothing\,.
\end{array}
\right.
\end{equation}
 There exist positive constants $C$ and $\kappa_{0}$ such that,
$$
\forall\kappa\geq\kappa_{0}\,,\qquad E_{\rm g}^{\textbf{L}}(\kappa,H,a,B_{0})\geq C\,\kappa^{2}\,.
$$
\end{prop}

\begin{proof}
Since $a(x,\kappa)=a(x)$,  the energy $E_{\rm g}^{\textbf{L}}$ becomes:
$$
 E_{\rm g}^{\textbf{L}}(\kappa,H,a,B_{0}):=\kappa^{2} \left(\int_{\{a(x)>0\}}a(x)^{2}\,\hat{f}\left(\sigma\frac{|B_{0}(x)|}{a(x)}\right)\,dx\\
 +\frac{1}{2}\int_{\{a(x)\leq 0\}}a(x)^{2}\,dx\right)\,.
$$
Each term being positive, it is  clear that the leading term is positive if $\{x\in\Omega:\,a(x)<0\} \neq\varnothing$.\\
If $ \{x\in\Omega:\,a(x)<0\}=\varnothing$ and $\{x\in\Omega:\,a(x)>0\} \neq\varnothing$, there exist $\rho_{0}>0$, $a_{0}>0$ and a disk $D(x_{0},r_{0})$ such that
$$
D(x_{0},r_{0})\subset\{a(x)>a_{0}\}\cap\{|B_{0}|>\rho_{0}\}\,.
$$
Using the monotonicity of $\hat{f}$ and the bound of $a(x)$ in \eqref{a2}, we may write
\begin{align}\label{est:a-k-ind}
\int_{\{a(x)>0\}}a(x)^{2}\,\hat{f}\left(\frac{H}{\kappa}\frac{|B_{0}(x)|}{a(x)}\right)\,dx&\geq \int_{D(x_{0},r_{0})} a(x)^{2}\,\hat{f}\left(\sigma\frac{|B_{0}(x)|}{a(x)}\right)\,dx\nonumber\\
&\geq \pi\,r_{0}^{2}\,a_{0}^{2}\,\hat{f}\left(\frac{\rho_{0}}{\overline{a}}\sigma\right)\,,
\end{align}
where $\overline{a}$ is introduced in \eqref{def:sup-a}.\\
In particular, when \eqref{cond-H} is satisfied, there exists $\kappa_{0}>0$ such that
\begin{equation}\label{ex1}
\forall\kappa\geq \kappa_{0}\,,\qquad\int_{\{a(x)>0\}}a(x)^{2}\,\hat{f}\left(\frac{H}{\kappa}\frac{|B_{0}(x)|}{a(x)}\right)\,dx\geq \pi\,r_{0}^{2}\,a_{0}^{2}\,\hat{f}\left(\frac{\rho_{0}}{\overline{a}}\lambda_{\min}\right)\,.
\end{equation}
\end{proof}

\begin{figure}[ht!]
\begin{center}
\includegraphics[scale=1]{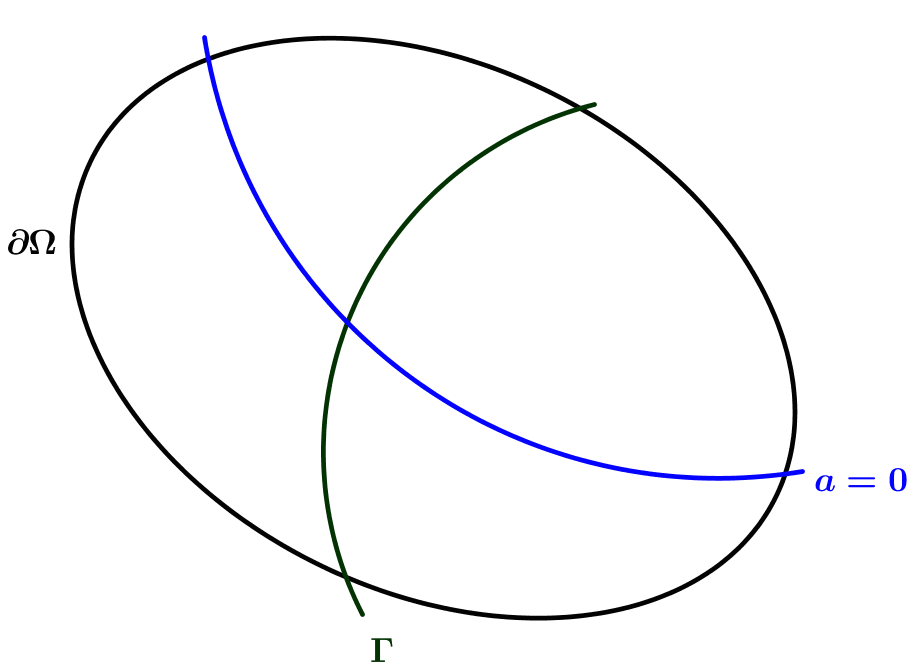}
\caption{Schematic representation of  $\Omega$  with pinning term independent of $\kappa$ and with variable magnetic field.}\label{example:1}
\end{center}
\end{figure}

\begin{prop}[\textbf{Verification of $(A_{4})$}]
Suppose that the function $a$ satisfies (see Fig.\ref{example:1}),
\begin{equation}\label{cond-a}
\left\{
\begin{array}{ll}
|a| + |\nabla a | >0&\mbox{ in } \overline{\Omega}\,,\\
\nabla a\times\vec{n}\neq 0 &\mbox{on}~ \widetilde{\Gamma}\cap\partial\Omega\,,
\end{array}
\right.
\end{equation}
where $\widetilde{\Gamma}$ defined as follows:
\begin{equation}\label{gamma-tilde}
\widetilde{\Gamma}=\{x\in\overline{\Omega}: a(x)=0\}\,.
\end{equation}
Then Assumption $(A_{4})$ is satisfied.
\end{prop}

\begin{proof}
From \eqref{cond-a}, we observe that,
$$
{\rm card}\,\{ \gamma \in \Gamma_\ell \cap \Omega \mbox{ with } Q_\ell (\gamma) \cap \partial\{a >0\} \neq \emptyset\}={\rm card}\,\{ \gamma \in \Gamma_\ell \cap \Omega \mbox{ with } Q_\ell (\gamma) \cap \widetilde{\Gamma} \neq \emptyset\}\,.
$$
Let $\epsilon\in(0,1)$, we introduce the domain
$$
D_{\epsilon}=\{x\in\Omega: \dist(x,\widetilde{\Gamma})\leq \epsilon\}\,.
$$
\textbf{Now we give a rough upper bound for the area of $D_\epsilon$.}\\
 By assumption  $\widetilde\Gamma$ consists of a finite number of connected curves, which are either closed in $\Omega$ or join two points of $\partial\Omega$. Let us consider the first case, we denote by $\widetilde\Gamma^{(1)}$ such a curve. We can parametrize this curve using the standard tubular coordinates $(s,t)$, where $s$ measures the arc-length in $\widetilde\Gamma^{(1)}$ and $t$ measures the distance to $\widetilde\Gamma^{(1)}$ (see \cite[Appendix~F]{FH1} for the detailed construction of these coordinates).

In the neighborhood of $\widetilde{\Gamma}^{(1)}$, we choose one point $\gamma_{0}$ on $\widetilde{\Gamma}^{(1)}$ corresponding to $(0,0)$. Let $N\in\N$ and $\mathcal{L}$ the length of $\widetilde\Gamma^{(1)}$. We consider for $i=0,...,N$, $s_{i}=\frac{i}{N}\,\mathcal{L}\,\,({\rm modulo}\,\,\mathcal{L}\Z)$ and $\gamma_{i}=(s_{i},0)$.\\
Notice that, there exists a positive constant $C$ such that,
$$
|\dist(\gamma_{i},\gamma_{i+1})|=(1+\epsilon_{i})|s_{i}-s_{i+1}|\,,\qquad\left(-\frac{C}{N}\leq \epsilon_{i}<0\right)\,.
$$
Thus,
\begin{align*}
\left|\left\{x\in\Omega: \dist\Big(x,\widetilde{\Gamma}^{(1)}\Big)\leq\frac{\mathcal{L}}{N}\right\}\right|&\leq \sum_{i} \left|Q_{\frac{\mathcal{L}}{N}}((s_i,0))\right|\,.
\end{align*}
Coming back to our problem, we select $N=\left[\frac{\mathcal{L}}{\epsilon}\right]$ and we note that $$\frac{\mathcal{L}}{N +1} \leq \epsilon\leq \frac{\mathcal{L}}{N}\,,$$
which implies that,
\begin{align*}
|D_{\epsilon}|&\leq \frac{\mathcal{L}^2}{N}\left(1+\mathcal{O}\left(\frac{1}{N}\right)\right)\\
&\leq \mathcal{L}\,\epsilon\,\left(1+\mathcal{O}\left(\frac{1}{N}\right)\right) = \epsilon \mathcal L (1+ \mathcal O(\epsilon))\,.
\end{align*}
Hence we have shown that,
$$\limsup_{\epsilon\to 0}\frac{|D_{\epsilon}|}{\epsilon}\leq\mathcal{L}\,.$$
In a similar fashion, we prove that
$$\liminf_{\epsilon\to 0}\frac{|D_{\epsilon}|}{\epsilon}\geq\mathcal{L}\,.$$
and, as a consequence, we end up with  the following conclusion:
\begin{equation}\label{Area:D}
\lim_{\epsilon\to 0}\frac{|D_{\epsilon}|}{\epsilon}=\mathcal{L}\,.
\end{equation}

Coming back to Assumption $(A_4)$,
we now observe that all the $Q_{\ell}(\gamma)$ touching $\widetilde \Gamma$ are inside $D_{\sqrt{2}\ell}$, hence
we get, by comparison of the area
$$
\ell^2 {\rm card}\,\{ \gamma \in \Gamma_\ell \cap \Omega \mbox{ with } Q_\ell (\gamma) \cap \widetilde{\Gamma} \neq \emptyset\}\leq C\,\ell \,,
$$
and consequently,
there exist positive constants $C_1$, $C_{2}$ and $\kappa_{0}$ such that
$$
\forall \kappa \geq \kappa_0\,,\, \forall \ell \leq C_2 \kappa^{-\frac 12}\,,\, {\rm card}\,\{ \gamma \in \Gamma_\ell \cap \Omega \mbox{ with } Q_\ell (\gamma) \cap \partial\{a >0\} \neq \emptyset\}\leq C_{1}\,\ell^{-1}\,,
$$
which is a stronger form of $(A_4)$.
\end{proof}

\subsection{The case with a $\kappa$-dependent oscillation. }
\subsubsection{Preliminaries}
We start with two lemmas which are standard in homogenization theory (see \cite[Section~16-17]{BLP})
\begin{lem}\label{lem:a-}
Let $D\subset\R^{2}$ be a bounded open set and $\varphi$ be a $\Gamma_{T_{1},T_{2}}$-periodic continuous function in $\R^{2}$ with $\Gamma_{T_{1},T_{2}}=T_{1}\Z\times T_{2}\Z$. There exists a positive constant $M_{0}$ such that if $M\geq M_{0}$, then,
$$
\int_{D}\varphi(M x)\,dx=\frac{|D|}{T_{1}T_{2}}\int_{0}^{T_{1}}\int_{0}^{T_{2}}\varphi(t_{1},t_{2})dt_{1}dt_{2}+\mathcal{O}(M^{-1})\,.
$$
\end{lem}
\begin{lem}\label{lem:a+}
Let $D\subset\R^{2}$ be a bounded open set
and  $\phi:\R^{2}\times \overline{D} \longrightarrow\R^{2}$ be a continuous function satisfying:
\begin{equation}\label{1st:cond}
\phi(t+T,x)=\phi(t,x)\,,\qquad\forall T\in T_{1}\Z\times T_{2}\Z\,,
\end{equation}
and uniformly Lipschitz, i.e. with the property that there exist constants $C>0$ and $\epsilon_0$, such that,
\begin{equation}\label{2nd:cond}
|\phi(t,x)-\phi(t,\widetilde{x})|\leq C\, |x-\widetilde{x}|\,,\quad \, \forall t \in \mathbb R^2\,, \,\forall x, \widetilde{x}\in\overline{D}, \;  {\rm s.t.}\, |x-\widetilde x|<\epsilon_0\,.
\end{equation}
There exists a positive constant $M_{0}$ such that if $M\geq M_{0}$, then,
$$
\int_{D}\phi(M x,x)\,dx=\int_{D}\overline{\phi}(x)\,dx+\mathcal{O}(M^{-1})\,,
$$
where,
\begin{equation}\label{def:phi}
\overline{\phi}(x)=\frac{1}{T_{1}T_{2}}\int_{0}^{T_{1}}\int_{0}^{T_{2}}\phi((t_{1},t_{2}),x)\,dt_{1}dt_{2}\,.
\end{equation}
\end{lem}

\subsubsection{First example:}
\begin{prop}\label{prop:1st-ex}
Suppose that \eqref{B(x)} and \eqref{cond-H} hold. Let $
a(x,\kappa)=\alpha(\kappa^{\frac{1}{2}}\,x)
$
where $\alpha(\cdot)\in C^{1}(\overline{\Omega})$ is a $\Gamma_{T_{1},T_{2}}$-periodic function\footnote{see Fig.~\ref{example:3}}. Then the leading order term $E_{\rm g}^{\textbf{L}}$ defined in \eqref{def:leading-term} satisfies,
$$
E_{\rm g}^{\textbf{L}}(\kappa,H,a,B_{0})=\kappa^{2} \, \int_{\Omega}\overline{\phi}_+(x)\,dx+\kappa^{2} |\Omega| \, \overline{\phi}_{-}  +\textit{o}(\kappa^{2})\,,\quad{\rm as}~\kappa\to+\infty\,.
$$
Here,
$$
\overline{\phi}_{+}(x)=\frac{1}{T_{1}T_{2}}\int_{0}^{T_{1}}\int_{0}^{T_{2}}\alpha_{+}(t_{1},t_{2})^{2}\,\hat{f}\left(\sigma\frac{|B_{0}(x)|}{\alpha_{+}(t_{1},t_{2})}\right)\,dt_{1}dt_{2}\,,
$$
and
$$
\overline{\phi}_{-} = \frac{1}{T_{1}T_{2}}\int_{0}^{T_{1}}\int_{0}^{T_{2}}\alpha_{-}(t_{1},t_{2})^2\,dt_{1}dt_{2}\,.
$$
\end{prop}

\begin{proof}~\\
{\bf We first estimate the second term in \eqref{def:leading-term}.}  We apply Lemma~\ref{lem:a-} with $D=\Omega$, $ M=\kappa^\frac{1}{2}$ and $\varphi=\alpha_{-}^2$, we obtain,
$$
\int_{\Omega}a_{-}(x,\kappa)^{2}\,dx=\frac{|\Omega|}{T_{1}T_{2}}\int_{0}^{T_{1}}\int_{0}^{T_{2}}\alpha_{-}(t_{1},t_{2})^{2}\,dt_{1}dt_{2}+\mathcal{O}(\kappa^{-\frac{1}{2}})\,,
$$
and consequently,
$$
\kappa^{2}\int_{\{a(x)\leq 0\}}a(x,\kappa)^{2}\,dx=\kappa^{2}\,\frac{|\Omega|}{T_{1}T_{2}}\int_{0}^{T_{1}}\int_{0}^{T_{2}}\alpha_{-}(t_{1},t_{2})^2 \,dt_{1}dt_{2}+\mathcal{O}(\kappa^{\frac{3}{2}})\,.
$$
{\bf Now, we estimate the first term in \eqref{def:leading-term}.}
We first prove that $\hat{f}$ is a Lipschitz function in $[\mathfrak{b}_{0},1]$ with $\mathfrak{b}_{0}\in(0,1)$. We consider this restriction because when $\mathfrak{b}\to 0_{+}$ (see \cite[Theorem~2.1]{KA2}), $\hat{f}$ satisfies,
\begin{equation}\label{ashatf}
\hat{f}(\mathfrak{b})=\frac{\mathfrak{b}}{2}\ln\frac{1}{\mathfrak{b}}(1+\textit{o}(1))\,,
\end{equation}
and $\hat{f}$ is not a Lipschitz function at $0$. We recall the definition of $\hat{f}$
$$
\displaystyle \hat{f}\left(\mathfrak{b}\right)=\lim_{R\longrightarrow\infty}\frac{e_{D}(\mathfrak{b},R)}{R^{2}}\qquad(\forall \mathfrak{b}\in[0,1])\,,
$$
where
$$
e_{D}(\mathfrak{b},R)=\inf_{u}F^{+1,+1}_{\mathfrak{b},Q_{R}}(u):=\inf_{u}\int_{Q_{R}}\left(\mathfrak{b}|(\nabla-i\Ab_0)u|^2+\frac{1}{2}\left(1-|u|^2\right)^{2}\right)\,dx\,.
$$
From the definition,  we can conclude that $\hat f$ is concave and hence locally Lipschitz in $(0,+\infty)$ (see \cite[Theorem~2.35]{MG}). For completion we write below a proof making  explicit the Lipschitz constant. For $\mathfrak{b}' >0$, let  $u_{\mathfrak{b}',R}\in H^{1}_{0}(Q_{R})$ be a minimizer of $F^{+1,+1}_{\mathfrak{b}',Q_{R}}$. Then for all $\mathfrak{b}\in (0,1)$, we have,
$$
e_{D}(\mathfrak{b},R)\leq F^{+1,+1}_{\mathfrak{b},Q_{R}} (u_{\mathfrak{b}',R})\leq e_{D}(\mathfrak{b}',R)+\|(\nabla-i\Ab_0)u_{\mathfrak{b}',R}\|^{2}_{L^{2}(Q_{R})}|\mathfrak{b}-\mathfrak{b}'|\,.
$$
Now, we estimate $\|(\nabla-i\Ab_0)u_{\mathfrak{b}',R}\|^{2}_{L^{2}(Q_{R})}$ from above. Coming back to the definition, we get the existence of a positive constant $C$, such that for any $\mathfrak{b}\in[\mathfrak{b}_0,1]$ and for any $\mathfrak{b}'\in[\mathfrak{b}_0,1]$,
\begin{align}\label{eq:est-from-above2}
\|(\nabla-i\Ab_0)u_{\mathfrak{b}',R}\|^{2}_{L^{2}(Q_{R})}&\leq \frac{e_{D}(\mathfrak{b}',R)}{\mathfrak{b}'}\nonumber \,.
\end{align}
This implies that,
$$
e_{D}(\mathfrak{b},R)\leq e_{D}(\mathfrak{b}',R)+ \frac{e_{D}(\mathfrak{b}',R)}{\mathfrak{b}'}   |\mathfrak{b}-\mathfrak{b}'|\,.
$$
Dividing by $R^2$ and taking the limit as $R\rightarrow +\infty$, we obtain
$$
\hat f (\mathfrak{b}) \leq \hat f (\mathfrak{b}') + \frac{ | \hat f (\mathfrak{b}') |}{\mathfrak{b}'} |\mathfrak{b}-\mathfrak{b}'|\,.
$$
Using the asymptotic behavior of $\hat f$ in \eqref{ashatf} as $\mathfrak{b}'\rightarrow 0_{+}$, we finally obtain the existence of $C$ such that
$$
\hat f (\mathfrak{b}) \leq \hat f (\mathfrak{b}') + C \left(\log \frac {1}{\mathfrak{b}_0}\right)  \,  |\mathfrak{b}-\mathfrak{b}'|\,,\, \forall \mathfrak{b}, \mathfrak{b}' \mbox{ with } 1 > \mathfrak{b}>\mathfrak{b}_{0} \mbox{ and } 1 >  \mathfrak{b}' >\mathfrak{b}_0\,.
$$
Exchanging $\mathfrak{b}$ and $\mathfrak{b}'$, we have proved the
\begin{lemma} $\hat f$ is locally Lipschitz in $(0,+\infty)$. More precisely, there exists $C$ such that  for any $\mathfrak{b}_0 >0$,
\begin{equation}
| \hat f (\mathfrak{b}) - \hat f (\mathfrak{b}')| \leq  C \,\left(\log \frac {1}{\mathfrak{b}_0}\right) \,  |\mathfrak{b}-\mathfrak{b}'|\,,\, \forall \mathfrak{b}, \mathfrak{b}' \mbox{ with } 1>\mathfrak{b}>\mathfrak{b}_0 \mbox{ and } 1 > \mathfrak{b}' >\mathfrak{b}_0\,.
\end{equation}
In addition, we have
\begin{equation}
| \hat f (\mathfrak{b}) - \hat f (\mathfrak{b}')| \leq  2   \,  |\mathfrak{b}-\mathfrak{b}'|\,,\, \forall \mathfrak{b}, \mathfrak{b}' \mbox{ with } \mathfrak{b}>\frac 12  \mbox{ and } \mathfrak{b}' > \frac 12\,.
\end{equation}
\end{lemma}

To continue, we consider
$$ \mathbb R^2 \times \Omega_\rho \ni (t,x) \mapsto \phi(t,x)=\alpha_{+}(t)^{2}\,\hat{f}\left(\sigma\frac{|B_{0}(x)|}{\alpha_{+}(t)}\right)\,,$$
where, $\Omega_{\rho}:=\Omega\cap\{|B_0|>\rho\}$.\\
The periodicity condition in \eqref{1st:cond}  is clear. Let us verify the Lipschitz property.
Let
$$
\mathfrak{b}_{0}=\frac{\lambda_{\min}}{\alpha_{0}}\,\rho \,,
$$
where, $\lambda_{\min}$ is introduced in \eqref{cond-H} and $\alpha_{0}=\sup \alpha_{+}(t)$.\\

Let $\epsilon>0$, $\mathcal{I}_{+}=\{t\in\R: \alpha_{+}(t)\geq \epsilon\}$ and $\mathcal{I}_{-}=\{t\in\R: \alpha_{+}(t)\leq \epsilon\}$, we distinguish between two cases:\\
\textbf{Case 1:} $( \alpha_{+}(t)\geq \epsilon)$. We observe that for $(x,t) \in \Omega_\rho\times \mathcal{I}_{+}$, we have
$$
\mathfrak{b}_0\leq\sigma\frac{|B_{0}(x)|}{\alpha_{+}(t)}\leq  \frac{\sigma\,|B_{0}(x)|}{\epsilon} \,.
$$
Thus, for any $t\in \mathcal{I}_{+}$ and for any $x,x'\in\overline{\Omega}_{\rho}$, we get
\begin{align}\label{eq:lipschitz1}
\left|\alpha_{+}(t)^{2}\,\hat{f}\left(\sigma\frac{|B_{0}(x)|}{\alpha_{+}(t)}\right)-\alpha_{+}(t)^{2}\,\hat{f}\left(\sigma\frac{|B_{0}(x')|}{\alpha_{+}(t)}\right)\right|&=\alpha_{+}(t)^{2}|\hat{f}\left(\mathfrak{b}\right)-\hat{f}\left(\mathfrak{b}'\right)|\nonumber\\
&\le
C\,\left(\log \frac {1}{\rho}\right)\,\Big||B_{0}(x)|-|B_{0}(x')|\Big|\,.
\end{align}
Therefore,  using also the Lipschitz property for $x\mapsto |B_0(x)|$, we get that $\Omega_\rho\ni x \mapsto \phi (t,x)$ is uniformly Lipschitz for $t\in  \mathcal{I}_{+}$.\\
\textbf{Case 2:} $( \alpha_{+}(t)\leq \epsilon)$. We observe that for $(x,t) \in \Omega_\rho\times\mathcal{I}_{-}$,
$$
\frac{\sigma\,|B_{0}(x)|}{\alpha_{+}(t)}\geq  \frac{\sigma\,|B_{0}(x)|}{\epsilon}\,.
$$
We note that $
\hat{f}(\mathfrak{b})=\frac{1}{2},\,\forall \mathfrak{b}\geq 1$ (see \cite[Theorem~2.1]{FK2}).  For this reason we choose
$$
\epsilon=\frac{\lambda_{\min}}{2}\rho \,,
$$
which implies that for $(x,t) \in \Omega_\rho\times  \mathcal{I}_{-}$,
$$\frac{\sigma\,|B_{0}(x)|}{\alpha_{+}(t)}\geq 2\qquad{\rm and}\qquad\hat{f}\left(\sigma\frac{|B_{0}(x)|}{\alpha_{+}(t)}\right)=\frac{1}{2}\,.
$$
Thus, for any  $t\in\mathcal{I}_{-}$ and for any $x,x'\in\overline{\Omega}_{\rho}$, we get
\begin{align}\label{eq:lipschitz2}
\left|\alpha_{+}(t)^{2}\,\hat{f}\left(\sigma\frac{|B_{0}(x)|}{\alpha_{+}(t)}\right)-\alpha_{+}(t)^{2}\,\hat{f}\left(\sigma\frac{|B_{0}(x')|}{\alpha_{+}(t)}\right)\right|&=\left|\frac{\alpha_{+}(t)^{2}}{2}-\frac{\alpha_{+}(t)^{2}}{2}\right|\nonumber\\
&=0\,.
\end{align}
Hence we get that $\Omega_\rho\ni x \mapsto \phi (t,x)$ is uniformly Lipschitz for $t\in  \mathcal{I}_{-}\,$.\\

Now, we apply Lemma~\ref{lem:a+} with $D=\Omega_\rho$ and $M=\kappa^\frac 12$ and we obtain,
\begin{align}\label{eq:1st-ex-2}
\int_{\Omega_{\rho}}a_{+}(x,\kappa)^{2}\,\hat{f}\left(\sigma\frac{|B_{0}(x)|}{a_{+}(x,\kappa)}\right)\,dx&=\int_{\Omega_{\rho}}\overline{\phi}(x)\,dx+\mathcal{O}_{\rho}(\kappa^{-\frac{1}{2}})\,,
\end{align}
where  $\overline{\phi}$ is introduced in \eqref{def:phi}.\\
Coming back to the integral over $\Omega$, we get, for any $\rho\in(0,\rho_{0})$ and for any $\kappa\geq\kappa_{0}$ with $\rho_{0}$ small enough and $\kappa_{0}$ large enough,
\begin{equation}\label{eq:final}
\int_{\Omega}a_{+}(x,\kappa)^{2}\,\hat{f}\left(\sigma\frac{|B_{0}(x)|}{a_{+}(x,\kappa)}\right)\,dx=\int_{\Omega}\overline{\phi}(x)\,dx+\mathcal{O}(\rho)+\mathcal{O}_{\rho}(\kappa^{-\frac{1}{2}})\,.
\end{equation}
Here, we have used the fact that $\overline{\phi}$ is a bounded function in $\Omega$.
Let us show that the remainder term $s(\kappa)$ in the right hand side in \eqref{eq:final} is $\textit{o}(1)$. The remainder term has the form $s_{1}(\kappa)+s_{2}(\kappa)$ with $s_{1}(\kappa)=\mathcal{O}(\rho)$ and $s_{2}(\kappa)= \mathcal{O}_{\rho}(\kappa^{-\frac{1}{2}})$. Let us show that it is $o(1)$.
Given $\varepsilon > 0$, there exists $\rho_{\varepsilon}> 0$ such that $|s_{1}(\kappa)|\leq \frac{\varepsilon}{2}$, for all $\kappa \geq \kappa_0$. Then, $\rho=\rho_{\varepsilon}$ being chosen, we can find $\kappa_{\varepsilon}\geq \kappa_0$ such that, for any $\kappa\geq\kappa_{\epsilon}$, $|s_{2}(\kappa)|\leq \frac{\varepsilon}{2}$.

\begin{figure}[ht!]
\begin{center}
\includegraphics[scale=1]{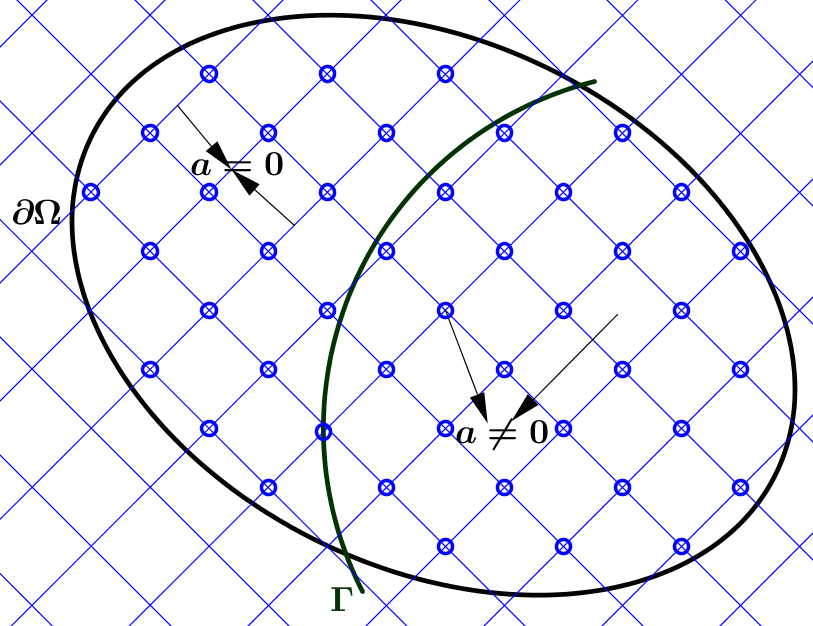}
\caption{Schematic representation of a domain with a $\kappa$-dependent oscillation pinning and with vanishing magnetic field along $\Gamma$.}\label{example:3}
\end{center}
\end{figure}
\end{proof}

\begin{prop}[Verification of $(A_{4})$]
Suppose that the function $\alpha$ defined in Proposition~\ref{prop:1st-ex} satisfies
\begin{equation}\label{cond-alpha}
|\alpha| + |\nabla \alpha | >0\quad\mbox{ in } \R^2\,.
\end{equation}
Then Assumption $(A_{4})$ is satisfied.
\end{prop}
\begin{proof}
Using \eqref{cond-alpha}, a change of variable $y=\kappa^{\frac{1}{2}}\,x$ and $\gamma'=\kappa^{\frac{1}{2}}\,\gamma$ yields,
 \begin{align*}
& {\rm card}\,\{ \gamma \in \Gamma_\ell \cap \Omega \mbox{ with } Q_\ell (\gamma) \cap \partial\{x\in\Omega:\,a(x,\kappa)>0\} \neq \emptyset\}\\
&\hspace*{7cm}={\rm card}\,\{ \gamma' \in \Gamma_{\kappa^{\frac{1}{2}}\ell} \cap \kappa^{\frac{1}{2}}\Omega \mbox{ with } Q_{\kappa^{\frac{1}{2}}\ell} (\gamma') \cap  \widehat{\Gamma} \neq \emptyset\}\,,
\end{align*}
where,
$$
\widehat{\Gamma}=\{y\in \mathbb R^2 \,|\,  \alpha(y)=0\}\,.
$$
Let $\epsilon\in(0,1)$, we introduce the domain
$$
\widehat D_{\epsilon, M}=\{y\in M \,\cdot \,\Omega: \dist(y,\widehat{\Gamma})\leq \epsilon\}\,.
$$
Thanks to \eqref{Area:D} and the periodicity assumption, we get the existence of positive constants $C$, $M_0$  and $\epsilon_0$ such that, for any $\epsilon \in (0,\epsilon_0)$, $M\geq M_0$
$$
|\widehat D_{\epsilon, M}|\leq C\,M\,\epsilon\,.
$$
 In the sequel, we choose $M=\kappa^{\frac{1}{2}}$ and $\epsilon=M\,\sqrt{2}\,\ell$. We note that, there exist constants $c > 0$ and $ \kappa_{0}> 0$ such that,
$$
\forall \kappa\geq \kappa_{0}\,,\quad\forall \ell\leq c\,\kappa^{-\frac{1}{2}}\,,\qquad  0< \epsilon\leq \epsilon_0 \,.
$$
We now observe that all the $Q_{\kappa^{\frac{1}{2}}\ell}(\gamma)$ touching $\widehat{\Gamma}$ are inside $\widehat D_{\kappa^{\frac{1}{2}}\,\sqrt{2}\,\ell,\kappa^{\frac{1}{2}}}$, hence
we get, by comparison of the areas
$$
\kappa\,\ell^2 {\rm card}\,\{ \gamma' \in \Gamma_{\kappa^{\frac{1}{2}}\ell} \cap \kappa^{\frac{1}{2}}\Omega \mbox{ with } Q_{\kappa^{\frac{1}{2}}\ell} (\gamma') \cap  \widehat{\Gamma}_{\kappa} \neq \emptyset\}\leq C\sqrt{2} \,\kappa\,\ell \,.
$$
There exist positive constants $C_1$ and $C_{2}$, such that,
$$
\forall \kappa \geq \kappa_0\,,\, \forall \ell \leq C_2 \kappa^{-\frac 12}\,,\, {\rm card}\,\{ \gamma \in \Gamma_\ell \cap \Omega \mbox{ with } Q_\ell (\gamma) \cap \partial\{x\in\Omega:\,a(x,\kappa)>0\} \neq \emptyset\}\leq C_{1}\,\ell^{-1}\,.
$$
\end{proof}

\subsubsection{Second example.}
This example was considered by Aftalion, Sandier and Serfaty (see $(H_2)$).
\begin{prop}
Suppose that \eqref{B(x)} and \eqref{cond-H} hold. Let $a(x,\kappa)=a(x)+\beta(x,\kappa)$,  where $\beta(x,\kappa)$ is a nonnegative function and $\{ a>0\}\cap\Omega\neq\varnothing$, (see Fig.~\ref{example:2}). There exist positive constants $\tau_1$ and $\kappa_{0}$ such that,
$$
\forall\kappa\geq\kappa_{0}\,,\qquad E_{\rm g}^{\textbf{L}}(\kappa,H,a,B_{0})\geq \tau_1 \,\kappa^{2}\,.
$$
\end{prop}

\begin{proof}
We can write,
\begin{align}\label{cp-err}
\kappa^{2}\int_{\{a(x,\kappa)>0\}}a(x,\kappa)^{2}\,\hat{f}\left(\frac{H}{\kappa}\frac{|B_{0}(x)|}{a(x,\kappa)}\right)\,dx&\geq \kappa^{2}\int_{\{a(x)>0\}}a(x,\kappa)^{2}\,\hat{f}\left(\frac{H}{\kappa}\frac{|B_{0}(x)|}{a(x,\kappa)}\right)\,dx\nonumber\\
&\geq \kappa^{2}\int_{\{a(x)>0\}}a(x)^{2}\,\hat{f}\left(\frac{H}{\kappa}\frac{|B_{0}(x)|}{\overline{a}}\right)\,dx\,.
\end{align}
Here we have used that $\hat{f}$ is increasing, the nonnegativity of $\beta$ to get $a(x,\kappa)\geq a(x)$, Assumption $(A_{2})$ to estimate $\hat{f}$ from below, and $\{a(x)>0\}\subset\{a(x,\kappa)>0\}$.\\
Proceding like in  \eqref{est:a-k-ind}, there exist $\tau_{1}>0$ and $\kappa_{0}>0$ such that,
\begin{equation}\label{lb:a-k2}
\forall\kappa\geq \kappa_{0}\,,\qquad\kappa^{2}\int_{\{a(x,\kappa)>0\}}a(x,\kappa)^{2}\,\hat{f}\left(\frac{H}{\kappa}\frac{|B_{0}(x)|}{a(x,\kappa)}\right)\,dx\geq\tau_{1}\,\kappa^{2}\,.
\end{equation}

\end{proof}

\begin{figure}[ht!]
\begin{center}
\includegraphics[scale=1]{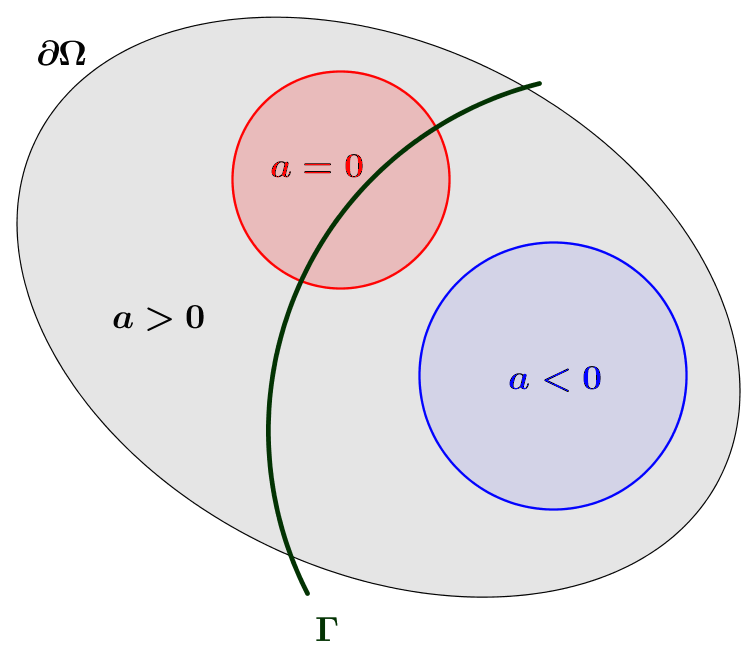}
\caption{Schematic representation of some domain with pinning term dependent of $\kappa$ and with vanishing magnetic field along $\Gamma$.}\label{example:2}
\end{center}
\end{figure}

\subsubsection{Third example:} This example is similar to the previous example, but here we suppose that
$$\beta(x,\kappa)=\alpha(\kappa^{\frac{1}{2}}x)\,,$$
where $\alpha(\cdot)$ is a $\Gamma_{T_{1},T_{2}}$-periodic positive function in $\R^2$.
\begin{prop}\label{prop:3rd-ex}
 Suppose that \eqref{B(x)} and \eqref{cond-H} hold. Let $a(x,\kappa)=a(x)+\alpha(\kappa^{\frac{1}{2}}x)$, where $\alpha(\cdot)$ is a $\Gamma_{T_{1},T_{2}}$-periodic positive bounded function in $\R^2$, $a(\cdot)\in C^{1}(\overline{\Omega})$ and $\{a<0\}\cap\Omega=\varnothing$. Then the leading order term $E_{\rm g}^{\textbf{L}}$ defined in \eqref{def:leading-term} satisfies,
$$
E_{\rm g}^{\textbf{L}}(\kappa,H,a,B_{0})=\kappa^{2} \, \int_{\Omega}\overline{\phi}(x)\,dx+\textit{o}(\kappa^{2})\,,\quad{\rm as}~\kappa\to+\infty\,.
$$
Here,
$$
\overline{\phi}(x)=\frac{1}{T_{1}T_{2}}\int_{0}^{T_{1}}\int_{0}^{T_{2}}(a(x)+\alpha(t_1,t_2))^{2}\,\hat{f}\left(\sigma\frac{|B_{0}(x)|}{a(x)+\alpha(t_1,t_2)}\right)\,dt_{1}dt_{2}\,.
$$
\end{prop}
The proof of Proposition~\ref{prop:3rd-ex} is similar to that of Proposition~\ref{prop:1st-ex}.

\subsection{Upper bound of the main term.}~\\

It is easy to show that $ E_{\rm g}^{\textbf{L}}$ is less than $C\kappa^{2}$ for some $C>0$. Indeed, using the bound of $a$ in \eqref{a2} and the bound $\hat{f}(b)\leq\frac{1}{2}$, we have,
$$
\kappa^{2}\int_{\{a(x,\kappa)>0\}}a(x,\kappa)^{2}\,\hat{f}\left(\frac{H}{\kappa}\frac{|B_{0}(x)|}{a(x,\kappa)}\right)\,dx\leq C\kappa^{2}\,,
$$
and
$$
\frac{\kappa^{2}}{2}\int_{\{a(x,\kappa)\leq 0\}}a(x,\kappa)^{2}\,dx\leq C\kappa^{2}\,.
$$
\section{Proof of Theorem~\ref{est-psi-main}}\label{7}
The technique that will be used  in this proof has been introduced by Helffer-Kachmar  in \cite{HK} for the case $a(x,\kappa)=\,1$. The proof is decomposed into three steps:\\
\textbf{Step 1: Case $\mathcal{D}=\Omega\,$.}\\
 Let $(\psi,\Ab)$ be a solution of \eqref{eq-2D-GLeq}. Thanks to \eqref{eq:A-psi}, we have,

 \begin{align*}
 \int_{\Omega}|(\nabla-i\kappa H\Ab)\psi|^{2}\,dx&=\kappa^{2}\int_{\Omega}(a(x,\kappa)-|\psi|^{2})|\psi|^{2}\,dx\\
 &=\frac{\kappa^{2}}{2}\int_{\Omega}(a(x,\kappa)^{2}-|\psi (x)|^{4})\,dx-\frac{\kappa^{2}}{2}\int_{\Omega}(a(x,\kappa)-|\psi|^{2})^{2}\,dx\,.
 \end{align*}
 Having in mind the definition of $\mathcal E_{0}(\psi,\Ab;a,\Omega)$, we get,
\begin{equation}
\frac{\kappa^{2}}{2}\int_{\Omega}(a(x,\kappa)^{2}-|\psi (x)|^{4})\,dx=\mathcal E_{0}(\psi,\Ab;a,\Omega)\,.
\end{equation}

Using \eqref{eq-2D}, we get that as $\kappa\longrightarrow+\infty$
\begin{multline} \label{previousone}
\frac{\kappa^{2}}{2}\int_{\Omega}(a(x,\kappa)^{2}-|\psi(x)|^{4})\,dx=\kappa^{2}\int_{\{a(x,\kappa)>0\}} a(x,\kappa)^{2}\hat{f}\left(\frac{H}{\kappa}\frac{|B_{0}(x)|}{a(x,\kappa)}\right)\,dx\\
+\frac{\kappa^{2}}{2}\int_{\{a(x,\kappa)\leq 0\}} a(x,\kappa)^{2}\,dx+\textit{o}\left(\kappa^{2}\right)\,.
\end{multline}
Notice that
$$\int_{\Omega}a(x,\kappa)^{2}\,dx=\int_{\{a(x,\kappa)\leq 0\}}a(x,\kappa)^{2}\,dx+\int_{\{a(x,\kappa)> 0\}}a(x,\kappa)^{2}\,dx\,.$$
Therefore, dividing  \eqref{previousone} by $\kappa^2$, we get
\begin{equation}\label{asy-psi-Omega}
\int_{\Omega}|\psi(x)|^{4}\,dx=-\int_{\{a(x,\kappa)>0\}} a(x,\kappa)^{2}\left\{2\hat{f}\left(\frac{H}{\kappa}\frac{|B_{0}(x)|}{a(x,\kappa)}\right)\,dx-1\right\}\,dx+\textit{o}\left(1\right)\,.
\end{equation}
\textbf{Step 2: Upper bound.} \\
Let $\mathcal{D}\subset\Omega$ be a regular domain and, for $\ell \in (0,1)$,
\begin{equation}
\mathcal{D}_{\ell}=\{x\in\mathcal{D}: {\rm dist}(x,\partial\mathcal{D})\geq\ell\}\,.
\end{equation}
We introduce a cut-off function $\chi_{\ell}\in C^{\infty}_{c}(\R^{2})$ such that
\begin{equation}\label{chi-l}
0\leq \chi_{\ell} \leq 1~{\rm in}~\R^{2}\,,\quad {\rm supp}\chi_{\ell}\subset \mathcal{D}\,,\quad \chi_{\ell}=1~{\rm in}~\mathcal{D}_{\ell} \quad{\rm and}\quad|\nabla\chi_{\ell}|\leq\frac{C}{\ell}~{\rm in}~\R^{2}\,,
\end{equation}
where $C$ is a positive constant. We multiply both sides of $\eqref{eq-2D-GLeq}_{a}$ by $\chi_{\ell}^{2}\psi$. It results from an integration by parts that
\begin{align}\label{eq:1}
\int_{\mathcal{D}} \left(|(\nabla-i\kappa H \Ab)\chi_{\ell}\psi|^{2}-\kappa^{2}a\,\chi_{\ell}^{2}|\psi|^{2}+\kappa^{2}\chi_{\ell}^{2}|\psi|^{4}\right)\,dx&=\int_{\mathcal{D}}|\nabla\chi_{\ell}|^{2}\, |\psi|^{2}\,dx\nonumber\\
&=\mathcal{O}(\ell^{-1})\,.
\end{align}
Here, we have used the fact that $|\nabla\,\chi_{\ell}|^{2}=\mathcal{O}(\ell^{-2})$, $|\mathcal{D}_{\ell}|=\mathcal{O}(\ell)$ and the bound of $\psi$ in \eqref{eq-psi<a}.\\
We notice that $\chi_{\ell}^{4}\leq\chi_{\ell}^{2}\leq 1$. We add to both sides the term $\frac{\kappa^{2}}{2}\int_{\mathcal{D}}a^{2}\,dx$ to obtain,
$$
\int_{\mathcal{D}} \left(|(\nabla-i\kappa H \Ab)\chi_{\ell}\psi|^{2}+\frac{\kappa^{2}}{2}a^{2}-\kappa^{2}a\,|\chi_{\ell}\,\psi|^{2}+\kappa^{2}|\chi_{\ell}\,\psi|^{4}\right)\,dx\leq\,C\,\ell^{-1}+\frac{\kappa^{2}}{2}\int_{\mathcal{D}}a^{2}\,dx\,.
$$
This implies that
$$
\mathcal E_0(\chi_{\ell}\psi,\Ab;a,\mathcal{D})\leq
\frac{\kappa^{2}}{2}\int_{\mathcal{D}}(a^{2}-\chi_{\ell}^{4}|\psi|^{4})\,dx+C\,\ell^{-1}\,.
$$
Using \eqref{chi-l}, we get
\begin{align}\label{up-psi42}
\int_{\mathcal{D}}|\psi|^{4}\,dx&= \int_{\mathcal{D}}\chi_{\ell}^{4}|\psi|^{4}\,dx+ \int_{\mathcal{D}}(1-\chi_{\ell}^{4})|\psi|^{4}\,dx\nonumber\\
&\leq \int_{\mathcal{D}}\chi_{\ell}^{4}|\psi|^{4}\,dx+C'\,\ell\,,
\end{align}
and consequently,
\begin{equation}\label{up-psi4}
\mathcal E_0(\chi_{\ell}\psi,\Ab;a,\mathcal{D})\leq
\frac{\kappa^{2}}{2}\int_{\mathcal{D}}(a^{2}-|\psi|^{4})\,dx+C(\ell^{-1}+\ell)\,.
\end{equation}
Using \eqref{fianl-E01} with $h=\chi_{\ell}$ and taking the choice of $\ell$  defined in \eqref{choice-ell-rho}, we get, as $\kappa\to+\infty$,
\begin{multline}
\frac{\kappa^{2}}{2}\int_{\mathcal{D}}(a^{2}-|\psi|^{4})\,dx\geq \kappa^{2}\int_{\mathcal{D}\cap\{a(x,\kappa)>0\}} a(x,\kappa)^{2}\hat{f}\left(\frac{H}{\kappa}\frac{|B_{0}(x)|}{a(x,\kappa)}\right)\,dx+\frac{\kappa^{2}}{2}\int_{\mathcal{D}\cap\{a(x,\kappa)\leq 0\}} a(x,\kappa)^{2}\,dx\\
+\textit{o}\left(\kappa^{2}\right)\,.
\end{multline}
Notice that,
$$\int_{\mathcal{D}}a(x,\kappa)^{2}\,dx=\int_{\mathcal{D}\cap\{a(x,\kappa)\leq 0\}}a(x,\kappa)^{2}\,dx+\int_{\mathcal{D}\cap\{a(x,\kappa)> 0\}}a(x,\kappa)^{2}\,dx\,.$$
Therefore,
\begin{multline}
-\frac{\kappa^{2}}{2}\int_{\mathcal{D}}|\psi|^{4}\,dx\geq \kappa^{2}\int_{\mathcal{D}\cap\{a(x,\kappa)>0\}} a(x,\kappa)^{2}\hat{f}\left(\frac{H}{\kappa}\frac{|B_{0}(x)|}{a(x,\kappa)}\right)\,dx-\frac{\kappa^{2}}{2}\int_{\mathcal{D}\cap\{a(x,\kappa)> 0\}} a(x,\kappa)^{2}\,dx\\
+\textit{o}\left(\kappa^{2}\right)\,.
\end{multline}
Dividing both sides by $-\frac{\kappa^{2}}{2}$, we obtain, as $\kappa\longrightarrow +\infty\,$,
\begin{equation}\label{up-psi45}
\int_{\mathcal{D}}|\psi|^{4}\,dx\leq-\int_{\mathcal{D}\cap\{a(x,\kappa)>0\}} a(x,\kappa)^{2}\left\{2\hat{f}\left(\frac{H}{\kappa}\frac{|B_{0}(x)|}{a(x,\kappa)}\right)-1\right\}\,dx+\textit{o}\left(1\right)\,.
\end{equation}

\begin{rem}\label{rm-up-psi-ng}
We can replace $\mathcal{D}$ by $\mathcal{\overline{D}}^{c}$ such that the estimate in \eqref{up-psi45} is still true. That is:
\begin{equation}\label{up-psi-ng}
\int_{\mathcal{\overline{D}}^{c}}|\psi|^{4}\,dx\leq-\int_{\mathcal{\overline{D}}^{c}\cap\{a(x,\kappa)>0\}} a(x,\kappa)^{2}\left\{2\hat{f}\left(\frac{H}{\kappa}\frac{|B_{0}(x)|}{a(x,\kappa)}\right)-1\right\}\,dx+\textit{o}\left(1\right)\,.
\end{equation}
\end{rem}

\textbf{Step 3: Lower bound.} \\
We can decompose $\int_{\mathcal{D}}|\psi|^{4}\,dx$ as follows:
$$
\int_{\mathcal{D}}|\psi|^{4}\,dx=\int_{\Omega}|\psi|^{4}\,dx-\int_{\mathcal{\overline{D}}^{c}}|\psi|^{4}\,dx
$$
Thanks to Remark~\ref{rm-up-psi-ng}, using the asymptotics in \eqref{asy-psi-Omega} obtained in Step~1 when $\mathcal{D}=\Omega$ and the upper bound in Step~2 , we get
\begin{equation}\label{up-psi-D}
\int_{\mathcal{D}}|\psi|^{4}\,dx\leq-\int_{\mathcal{D}\cap\{a(x,\kappa)>0\}} a(x,\kappa)^{2}\left\{2\hat{f}\left(\frac{H}{\kappa}\frac{|B_{0}(x)|}{a(x,\kappa)}\right)-1\right\}\,dx+\textit{o}\left(1\right)\,.
\end{equation}
\section{Extension of the Giorgi-Phillips Theorem}\label{GP}
In this section we extend a result of Giorgi-Phillips \cite{GP}, in the two cases when the external magnetic field $B_{0}$ is variable (i.e. $\Gamma\neq\varnothing$) and when the external magnetic field $B_{0}$ is constant (i.e. $\Gamma=\varnothing$), with a pinning term.
We recall that the normal solution $(0,\Fb)$ is a trivial solution of the Ginzburg-Landau system \eqref{eq-2D-GLeq}. We will show that this solution is a global minimizer, when $\kappa$ and $H$ are sufficiently large. We first establish a priori estimates for a critical point  $(\psi, \Ab)$ of the G-L-functional.
\subsection{Estimates of $\Ab$ and of $\|(\nabla-i\kappa H\Fb)\psi\|$.}~\\
We need the following estimates on $\Ab$ and on $\|(\nabla-i\kappa H\Fb)\psi\|$ which are independent of the assumption of $\Gamma$.
\begin{thm}\label{thm-2D-apriori}
There exist positive constants $C_{1}$, $C_{2}$ and $C_{3}$ such that, if \eqref{a2} hold, $\kappa>0$, $H>0$ and $(\psi,\Ab)$ is a solution of \eqref{eq-2D-GLeq}, then,
\begin{eqnarray}
\|(\nabla-i\kappa H\Ab)\psi\|_{L^{2}(\Omega)}\leq C_{1}\,\kappa\, \|\psi\|_{L^{2}(\Omega)}\label{2nd-<}\,,\\
\|\curl(\Ab-\Fb)\|_{L^{2}(\Omega)}\leq \frac{C_{2}}{H}\, \|\psi\|_{L^{2}(\Omega)}\|\psi\|_{L^{4}(\Omega)}\label{3d-<}\,,\\
\|(\nabla-i\kappa H\Fb)\psi\|_{L^{2}(\Omega)}\leq C_{3}\,\kappa\,\|\psi\|_{L^{2}(\Omega)}\label{8d-<}\,.
\end{eqnarray}
\end{thm}
\begin{proof}
 \textbf{We first prove \eqref{2nd-<}.} In the case when $\overline{a}\leq 0$ with $\overline{a}$  introduced in \eqref{def:sup-a}, we get  using \eqref{eq-psi<a}  that $\psi=0$ and hence \eqref{2nd-<} is proved.\\
In the case when $\overline{a}>0$, thanks to \eqref{eq-psi<a}, we have,
\begin{equation}\label{eq:a-psi}
0\leq(\overline{a}-|\psi|^{2})\leq \overline{a}\,.
\end{equation}
We recall that if $(\psi,\Ab)$ is a solution of \eqref{eq-2D-GLeq} then, (see \eqref{eq:A-psi})
$$
\int_{\Omega}|(\nabla-i\kappa H\Ab)\psi|^{2}\,dx=\kappa^{2}\int_{\Omega}(a(x,\kappa)-|\psi|^{2})|\psi|^{2}\,dx\,.
$$
Using \eqref{a2} and \eqref{eq:a-psi}, we obtain \eqref{2nd-<}.\\
\textbf{Now, we prove \eqref{3d-<}.} We obtain from the equation in \eqref{eq-2D-GLeq}$_{b}$ the following estimate (see \cite[Equation~(11.9b)]{FH1}):
$$
\kappa H\int_{\Omega}|\curl(\Ab-\Fb)|^{2}\,dx\leq \|(\nabla-i\kappa H\Ab)\psi\|_{L^{2}(\Omega)}\, \|(\Ab-\Fb)\psi\|_{L^{2}(\Omega)}\,.
$$
Using \eqref{2nd-<} and applying H$\rm\ddot{o}$lder's inequality, we get
$$
\kappa H\int_{\Omega}|\curl(\Ab-\Fb)|^{2}\,dx\leq C\,\kappa\,\|\psi\|_{L^{2}(\Omega)}\|\psi\|_{L^{4}(\Omega)}\, \|\Ab-\Fb\|_{L^{4}(\Omega)}\,.
$$
We get by regularity of the $\curl$-$\Div$ system (see \cite[A.7]{FH1}),
\begin{equation}\label{eq:curl-div}
\|\Ab-\Fb\|_{H^{1}(\Omega)}\leq C\|\curl(\Ab-\Fb)\|_{L^{2}(\Omega)}\,,
\end{equation}
where $C$ is a positive constant.\\
By the Sobolev embedding theorem, we get,
\begin{align}\label{sb-emb-L4}
\|\Ab-\Fb\|_{L^{4}(\Omega)}&\leq C_{\rm Sob}\, \|\Ab-\Fb\|_{H^{1}(\Omega)}\nonumber\\
&\leq \widehat C \, \|\curl(\Ab-\Fb)\|_{L^{2}(\Omega)}\,.
\end{align}
Consequently,
$$
\kappa H\, \int_{\Omega}|\curl(\Ab-\Fb)|^{2}\,dx\leq \kappa\, \|\psi\|_{L^{2}(\Omega)}\|\psi\|_{L^{4}(\Omega)}\|\curl(\Ab-\Fb)\|_{L^{2}(\Omega)}\,,
$$
which leads to \eqref{3d-<}.\\
\textbf{Finally, we prove \eqref{8d-<}.}
Using \eqref{3d-<} and \eqref{sb-emb-L4}, H\"older's inequality gives,
\begin{align}\label{5d-<}
\|(\Ab-\Fb)\psi\|_{L^{2}(\Omega)}^{2}&\leq \|\Ab-\Fb\|_{L^{4}(\Omega)}^{2}\|\psi\|_{L^{4}(\Omega)}^{2}\nonumber\\
& \leq \frac{C'}{H^{2}}\|\psi\|_{L^{4}(\Omega)}^{4}\|\psi\|_{L^{2}(\Omega)}^{2}\,,
\end{align}
Using \eqref{2nd-<}, \eqref{5d-<} and the bound of $\psi$ above, Young's inequality gives,
\begin{align}
\|(\nabla-i\kappa H\Fb)\psi\|_{L^{2}(\Omega)}^{2}&\leq 2\|(\nabla-i\kappa H\Ab)\psi\|_{L^{2}(\Omega)}^{2}+2\,(\kappa H)^{2}\|(\Ab-\Fb)\psi\|_{L^{2}(\Omega)}^{2}\nonumber\\
&\leq 2\,C''\,\kappa^{2}\|\psi\|^{2}_{L^{2}(\Omega)}\,.
\end{align}
\end{proof}
\subsection{The case $\Gamma=\varnothing$.}~\\
For $\xi\in\R$, we consider the Neumann realization $\mathfrak{h}^{N,\xi}$ in $L^{2}(\R_+)$ associated with the operator $-\frac{d^2}{dt^2}+(t+\xi)^2$, i.e.
\begin{equation}
\mathfrak{h}^{N,\xi}:=-\frac{d^2}{dt^2}+(t+\xi)^2\,,\qquad \mathcal{D}(\mathfrak{h}^{N,\xi})=\{u\in B^{2}(\R_{+}): u'(0)=0\}\,,
\end{equation}
where,
$$
B^2(\R_+)=\{u\in L^{2}(\R_+): t^p u^{(q)}\in L^{2}(\R_+), \forall p,q\in\N~s.t.~p+q\leq 2\}\,.
$$
M. Dauge and B. Helffer \cite{DH} (see also Fournais-Helffer \cite[Proposition~4.2.2]{FH1}) have  proved that the lowest eigenvalue $\mu$ of $\mathfrak{h}^{N,\xi}$ admits a minimum $\Theta_{0}$, which is attained at a unique point $\xi_{0}<0$, and satisfies:
\begin{equation}\label{Theta0}
\Theta_{0}=\inf_{\xi} \mu(\xi)=\mu(\xi_0)<1\,.
\end{equation}
Moreover
\begin{equation}
\Theta_{0}=\xi_{0}^2\,.
\end{equation}
We introduce the notation:
\begin{equation}\label{inf:B0}
\inf_{x\in\overline{\Omega}} |B_{0}(x)|=b_{0}\qquad{\rm and}\qquad \inf_{x\in\partial\Omega} |B_{0}(x)|=b'_{0}\,.
\end{equation}
We denote by $\mu^{N}(\mathcal{B}\Fb;\Omega)$ the lowest eigenvalue of the $\rm Schr\ddot{o}dinger$ operator $P_{\mathcal{B}\Fb,0}^{\Omega}$ (see \eqref{def:P}) with Neumann condition in $L^{2}(\Omega)$:
\begin{equation}\label{muN(kHF)}
\mu^{N}(\mathcal{B}\Fb;\Omega)=\inf_{\substack{\psi\in H^{1}(\Omega)\\ \psi\neq 0}}\frac{\langle P_{\mathcal{B}\Fb,0}^{\Omega}\,\psi,\psi\rangle}{\|\psi\|^{2}_{L^{2}(\Omega)}}\,.
\end{equation}

In \cite{FH1}, it is proved that
\begin{theorem}\label{thm:FH}
Suppose that $\Omega\subset\R^2$ is an open bounded set with smooth boundary and $\Gamma=\varnothing$. Then,
\begin{equation}\label{eq:FH}
\lim_{\mathcal{B}\longrightarrow+\infty}\frac{\mu^{N}(\mathcal{B}\Fb,\Omega)}{\mathcal{B}}=\min(b_0,\Theta_{0}\,b'_{0})\,.
\end{equation}
\end{theorem}
 In the next theorem, we give a simple proof of the result which says that $(0,\Fb)$ is the unique minimizer of the functional when $H$ is sufficiently large and when the magnetic field  $B_0$ is constant with pinning term.
 \begin{theorem}\label{thm:GP2}
Let $\Omega\subset\R^2$ be a smooth, bounded, simply-connected open set and $\Gamma=\varnothing$. Then, there exist positive constants $C$ and $\kappa_{0}$, such that, if
$$
H\geq C \kappa\,,\qquad\kappa\geq\kappa_{0}\,,
$$
then $(0,\Fb)$ is the unique solution to \eqref{eq-2D-GLeq}.
\end{theorem}
\begin{proof}
We assume that we have a \textbf{non normal} critical point $(\psi, \Ab)$ for $\mathcal E_{\kappa,H,a,B_{0}}$. This means that $(\psi, \Ab)\in H^{1}(\Omega)\times H^{1}_{{\rm div}}(\Omega)$ is a solution of \eqref{eq-2D-GLeq} and
\begin{equation}\label{psi>0}
\int_{\Omega}|\psi|^{2}\,dx>0\,.
\end{equation}
Therefore, we get from \eqref{eq-psi<a} that,
$$
|\psi(x)|^{2}\leq \overline{a}\,\qquad\forall x\in\overline{\Omega}\,,
$$
where $\overline{a}$ is introduced in \eqref{def:sup-a}.\\
Let
\begin{equation}\label{eq:B=kH}
\mathcal{B}=\kappa H\,.
\end{equation}
Theorem~\ref{thm-2D-apriori} tells us that,
$$
\|(\nabla-i\mathcal{B}\Fb)\psi\|_{L^{2}(\Omega)}^{2}\leq C\,\kappa^{2}\,\|\psi\|^{2}_{L^{2}(\Omega)}\,.
$$
Since $\psi$ satisfies \eqref{psi>0}, this implies by assumption that the lowest Neumann eigenvalue\\
$\mu^{N}(\mathcal{B}\Fb;\Omega)$ of $P_{\mathcal{B}\Fb,0}^{\Omega}$ in $L^{2}(\Omega)$ satisfies,
\begin{equation}\label{muN<}
\mu^{N}(\mathcal{B}\Fb;\Omega)\leq C\,\kappa^{2}\,.
\end{equation}
Thanks to Theorem~\ref{thm:FH}, we get the existence of a constant $C>0$, such that, if $H\geq C\,\kappa$, then $(0,\Fb)$ is the unique solution to \eqref{eq-2D-GLeq}.
\end{proof}
\subsection{The case $\Gamma\neq\varnothing$.}~\\
We recall the definition of $\lambda_{0}$ in \eqref{lambda0}, the definition of $\Gamma$ in \eqref{gamma} and for any $\theta\in(0,\pi)$ we recall that $\lambda(\R_{+}^{2},\theta)$ is the bottom of the spectrum of the operator
$ P_{\Ab_{\rm app,\theta},0}^{\R^{2}_{+}}$, with
$$\Ab_{\rm app,\theta}=-\left(\frac{x^{2}_{2}}{2}\cos\,\theta,\frac{x^{2}_{1}}{2}\sin\,\theta \right)\,.$$
Define
\begin{equation}\label{alpha1}
\alpha_{1}(B_{0})=\min\left\{\lambda_{0}^{\frac{3}{2}}\min_{x\in \Gamma\cap\Omega}|\nabla B_{0}(x)|,\min_{x\in \Gamma\cap\partial\Omega}\lambda(\R^{2}_{+},\theta(x))^{\frac{3}{2}}|\nabla B_{0}(x)|\right\}\,.
\end{equation}

In \cite{XB-KH}, it is proved that
\begin{thm}\label{thm:BK}
 Suppose that \eqref{B(x)} holds and $\Gamma\neq\varnothing$. Then
\begin{equation}\label{eq:pan}
\lim_{\mathcal{B}\longrightarrow+\infty}\frac{\mu^{N}(\mathcal{B}\Fb,\Omega)}{\mathcal{B}^{\frac{2}{3}}}=\alpha_{1}(B_{0})^{\frac{2}{3}}\,.
\end{equation}
\end{thm}
In the next theorem, we give a simple proof of the result which says that $(0,\Fb)$ is the unique minimizer of the functional when $H$ is sufficiently large and when $B_{0}$ is variable. This result was obtained  in \cite{GP} for the case with constant magnetic field and with a constant pinning term.
\begin{theorem}\label{thm:GP}
Let $\Omega\subset\R^2$ be a smooth, bounded, simply-connected open set, the pinning term $a$ satisfying \eqref{a2}, and the magnetic field  $B_{0}$ satisfying \eqref{B(x)}. Then, there exist positive constants $C$ and $\kappa_{0}$, such that, if
$$
H\geq C \kappa^{2}\,,\qquad\kappa\geq\kappa_{0}\,.
$$
Then $(0,\Fb)$ is the unique solution to \eqref{eq-2D-GLeq}.
\end{theorem}
\begin{proof}
Similarly to the proof of Theorem~\ref{thm:GP2}, we assume that we have a \textbf{non normal} critical point $(\psi, \Ab)$ for $\mathcal E_{\kappa,H,a,B_{0}}$.\\
Therefore, we get from \eqref{8d-<} that,
$$
\mu^{N}(\mathcal{B}\Fb;\Omega)\leq C\,\kappa^{2}\quad(\mathcal B=\kappa H)\,.
$$
Thanks to Theorem~\ref{thm:BK}, we get a contradiction, if $ \displaystyle H\geq C \kappa^{2}$ and $C$ is sufficiently large.
\end{proof}

\section{Asymptotics of $\mu_{1}(\kappa,H)$: the case with non vanishing magnetic field}\label{Section:4}
The aim of this section is to give  an estimate for the lowest
eigenvalue $\mu_{1}(\kappa,H)$ of the operator $P_{\kappa
H\Fb,-\kappa^{2}a}^{\Omega}$ (see \eqref{def:mu1})  in the case when
$\Gamma=\varnothing$ with a $\kappa$-independent pinning (i.e.
$a(x,\kappa)=a(x)$). Recall that the set $\Gamma$ is introduced in
\eqref{gamma}.

\subsection{Lower bound}~\\
Without loss of generality we suppose that $B_{0}>0 \mbox{ in } \overline{ \Omega}$. Our results will be formulated by introducing:
\begin{equation}\label{Lambda11}
\Lambda_{1}(B_{0},a,\sigma)=\min\left\{\inf_{ x\in\Omega}\left\{\sigma\,B_{0}(x)-a(x)\right\},\inf_{ x\in\partial\Omega}\left\{\Theta_{0}\,\sigma\,B_{0}(x)-a(x)\right\}\right\}\,,
\end{equation}
where $\sigma$ is a positive constant.\\
In the case when the pinning term is constant (i.e. $a(x)=a_0$), \eqref{Lambda11} becomes as follows:
$$
\Lambda_{1}(B_{0},a,\sigma)=\sigma\min\left\{\inf_{ x\in\Omega}\left\{B_{0}(x)\right\},\Theta_{0}\,\inf_{ x\in\partial\Omega}\left\{B_{0}(x)\right\}\right\}-a_0\,.
$$
This case was treated by Pan and Kwek \cite{LP1}.\\
Let $\mathcal{Q}_{\mathcal{B}\,\Fb,-\frac{\mathcal B}{\sigma}\,a}^{\Omega}$ be the quadratic form of $P_{\mathcal{B}\Fb,- \frac{\mathcal{B}}{\sigma}\,a}^{\Omega}$, i.e.
\begin{equation}\label{def:quad}
\mathcal{Q}_{\mathcal{B}\,\Fb,-\frac{\mathcal B}{\sigma}\,a}^{\Omega}(\psi)=\int_{\Omega}\left(|(\nabla-i\mathcal{B}\Fb)\psi|^{2}-\frac{\mathcal{B}}{\sigma} \,a(x)|\psi|^{2}\right)\,dx\,.
\end{equation}
\begin{prop}\label{prop:mu1-cst}
Let $\Omega\subset\R^2$ be an open bounded set with smooth boundary, $I$ a closed interval  in $(0,+\infty)$ and $\Gamma=\varnothing$. There exist positive constant $C$ and $\mathcal{B}_{0}$ such that if $\sigma\in I$, $\mathcal{B}\geq\mathcal{B}_{0}$, $\psi\in H^{1}(\Omega)\setminus \{0\} $ and $a\in C^{1}(\overline{\Omega})$, then,
\begin{equation}\label{eq:lower-bound-constant}
\frac{\mathcal{Q}_{\mathcal{B}\Fb,-\frac{\mathcal B}{\sigma}\,a}^{\Omega}(\psi)}{\|\psi\|_{L^{2}(\Omega)}^{2}}\geq \frac{\mathcal B}{\sigma}\,\,\Lambda_{1}(B_{0},a,\sigma)-C\,\mathcal{B}^{\frac{3}{4}}\,.
\end{equation}
\end{prop}
\begin{proof}
The proof is a consequence of the following inequality that we take
from \cite[Prop.~9.2.1]{FH1},
$$
\forall~\psi\in H^1(\Omega)\,,\quad\int_\Omega|(\nabla-i\mathcal{B}\Fb)\psi|^2\,dx\geq \int_\Omega \big(U(x)-\bar CB^{3/4}\big)|\psi|^2\,dx\,,
$$
where \begin{equation}\label{eq:U}
U(x)= \left\{
\begin{array}{ll}
\mathcal B\,B_0(x)&{\rm if~}{\rm dist}(x,\partial\Omega)\geq B^{-3/8}\,,\\
\Theta_0\mathcal B\,B_0(x)&{\rm if~}{\rm dist}(x,\partial\Omega)< B^{-3/8}\,,
\end{array}
\right.
\end{equation}
$\mathcal B\geq \bar{\mathcal B_0}$, $\bar{\mathcal B_0}$ and $\bar
C$ are two constants independent of $\mathcal B$.

Clearly, there exist two constants $C'>0$ and $\mathcal B_0>0$ such
that, for all $\sigma\in I$, we have,
$$
U(x)-\frac{\mathcal B}{\sigma}a(x)\geq\frac{\mathcal B}{\sigma}\Lambda_1(B_0,a,\sigma)-C'B^{3/4}\,.
$$
\end{proof}
Coming back to our initial parameters $\kappa$ and $H$, we obtain:
\begin{theorem}\label{thm:mu1-cst}
Let $\Omega\subset\R^2$ be an open bounded set with smooth boundary and $\Gamma=\varnothing$. Suppose that \eqref{cond-H} holds and $a\in C^1(\overline{\Omega})$, then,
$$
\mu_{1}(\kappa,H)\geq \kappa^{2}\,\Lambda_{1}\left(B_{0},a,\frac H \kappa\right)+\mathcal{O}(\kappa^{\frac{3}{2}})\,,\qquad{\rm as}\,\kappa\to+\infty\,.
$$
Here, $\Lambda_{1}$ is introduced in \eqref{Lambda11}.
\end{theorem}
\begin{proof}
We apply Proposition~\ref{prop:mu1-cst} with
$$
\mathcal{B}=\kappa H\,,\quad \sigma=\frac{H}{\kappa}\quad{\rm and}\quad I=[\lambda_{\min},\lambda_{\max}]\,.
$$
Let us verify that the conditions of the proposition are satisfied for this choice.\\
It is trivial that $\sigma\in I$. Now, as $\kappa\to+\infty$, we have,
$$
\mathcal{B}=\sigma\,\kappa^{2}\to+\infty\,.
$$
This implies that, as $\kappa\to+\infty$,
$$
\mu_{1}(\kappa,H)\geq \kappa^{2}\,\Lambda_{1}\left(B_{0},a,\frac{H}{\kappa} \right)+\mathcal{O}(\kappa^{\frac{3}{2}})\,.
$$
This finishes the proof of the theorem.
\end{proof}

\subsection{Upper bound}
\begin{prop}[Upper bound in the bulk]\label{prop:up-blk-cst}
Suppose that $\Omega\subset\R^2$ is an open bounded set with smooth boundary $\partial\Omega$, $\lambda_{\max} >0$ and $\Gamma=\varnothing$. For any $ x_0\in\Omega$, there exist positive constants $C$ and $\mathcal{B}_{0}$ such that, if $\sigma\in (0,\lambda_{\max}]$, $\mathcal{B}\geq\mathcal{B}_{0}$ and $a\in C^{1}(\overline{\Omega})$, then,
\begin{equation}\label{eq:up-blk-cst}
\mu_{\mathcal{B},\sigma}\leq \frac{\mathcal B}{\sigma}\,\left\{\sigma\,B_{0}(x_{0})-a(x_{0})\right\}+C\,\mathcal{B}^{\frac{1}{2}}\,.
\end{equation}
Here,
\begin{equation}\label{def:mu1-cst}
\mu_{\mathcal{B},\sigma}=\inf_{\psi\in H^{1}(\Omega)\setminus \{0\}}\frac{\mathcal{Q}_{\mathcal{B}\Fb,-\frac{\mathcal B}{\sigma}\,a}^{\Omega}(\psi)}{\|\psi\|_{L^{2}(\Omega)}^{2}}\,,
\end{equation}
 where $\mathcal{Q}_{\mathcal{B}\Fb,-\frac{\mathcal B}{\sigma}\,a}^{\Omega}$ is introduced in \eqref{def:quad}.
\end{prop}
\begin{proof}
Thanks to \eqref{def:quad}, we have,
$$
\mathcal{Q}_{\mathcal{B}\Fb,-\frac{\mathcal B}{\sigma}\,a}^{\Omega}(u) = \int_{\Omega}|(\nabla-i\mathcal{B}\Fb)u(x)|^{2}\,dx - \frac{ \mathcal B}{\sigma}\,\int_{\Omega}a(x)|u(x)|^{2}\,dx \,.
$$
The upper bound of the first term in the right hand side above is based on the construction of Gaussian quasi-mode (see \cite[Subsection~2.4.2]{FH1} for the case with constant pinning) centered at $x_0\in\Omega$,
$$
\varphi_{1}(x)=\,e^{i\,\mathcal{B}\,\phi_{0}}\,\chi\left(\mathcal{B}^{+\frac{1}{2}}(x-x_{0})\right)\,u\left(\sqrt{\mathcal{B}B_{0}(x_0)}\,(x-x_{0})\right)\,.
$$
Here, $\chi$ is a cut-off function equal to $1$ in a neighborhood of $0$ such that ${\rm supp}\,\chi\subset D(0,1)$, the function  $\phi_{0}$ satisfies \eqref{F-A} and the function $u$ defined as follows:
$$
u(x)=\frac{\pi^{-\frac{1}{4}}}{\sqrt{2}}e^{-\frac{|x|^2}{2}}\,.
$$
We note that ${\rm supp}\,\varphi_{1}\subset \Omega$ for $\mathcal{B}$ large enough. As in \cite[(2.35)]{FH1}, we get the existence of a positive constant $\mathcal{B}_{0}$ such that, for any $\mathcal{B}\geq\mathcal{B}_{0}$,
\begin{equation}\label{up:FA-cst}
\frac{\int_{\Omega}|(\nabla-i\mathcal{B}\Fb)\varphi_{1}(x)|^{2}\,dx}{\int_{\Omega}|\varphi_{1}(x)|^{2}\,dx}\leq \mathcal{B}\,B_{0}(x_0)+\mathcal{O}(\mathcal{B}^{\frac{1}{2}})\,.
\end{equation}
To derive the upper bound for the second term, we use Taylor's formula for the function $a$ near $x_0$,
\begin{equation}\label{eq:app-a-cst}
|a(x)-a(x_0)|\leq C\,\,\mathcal{B}^{-\frac{1}{2}}\,,\qquad\left(x\in D\left(x_0,\mathcal{B}^{-\frac{1}{2}}\right)\right)\,.
\end{equation}
Using \eqref{eq:app-a-cst}, since ${\rm supp}\,\varphi_{1}\subset D\left(x_0,\mathcal{B}^{-\frac{1}{2}}\right)$, we get,
\begin{equation}\label{upp:a1-cst}
-\int_{\Omega}a(x)|\varphi_{1}(x)|^{2}\,dx\leq -a(x_{0})\int_{\Omega}|\varphi_{1}(x)|^{2}\,dx+C\,\mathcal{B}^{-\frac{1}{2}}\,\int_{\Omega}|\varphi_{1}(x)|^{2}\,dx\,,
\end{equation}
and consequently
\begin{equation}\label{upp:a-cst}
-\frac{\mathcal B}{\sigma}\frac{\int_{\Omega}a(x)|\varphi_{1}(x)|^{2}\,dx}{\int_{\Omega}|\varphi_{1}(x)|^{2}\,dx}\leq -\frac{\mathcal B}{\sigma}\,a(x_{0})+C\,\mathcal{B}^{\frac{1}{2}}\,.
\end{equation}
Collecting \eqref{up:FA-cst} and \eqref{upp:a-cst}, we finish the proof of Proposition~\ref{def:mu1-cst}.
\end{proof}
\begin{rem}\label{rem:first term}
When
$$
\inf_{x\in\Omega}\left\{\sigma\,B_{0}(x)-a(x)\right\} <  \inf_{ x\in\partial\Omega}\left\{\Theta_{0}\,\sigma\,B_{0}(x)-a(x)\right\}\,,
$$
we notice that, if the infimum of $\sigma\,B_{0}(x)-a(x)$ was attained on $\partial\Omega$, (i.e. there exists $x_0\in\partial\Omega$ such that $\inf_{x\in\Omega}\left\{\sigma\,B_{0}(x)-a(x)\right\} =\sigma\,B_{0}(x_0)-a(x_0)$), we would have,
$$
\sigma\,B_{0}(x_0)-a(x_0) < \Theta_{0}\,\sigma\,B_{0}(x_0)-a(x_0)\,,
$$
which is impossible, since $\Theta_{0}<1$. Hence, we can  choose $x_0\in\Omega$, such that,
$$
\sigma\,B_{0}(x_0)-a(x_0)=\inf_{x\in\Omega}\left\{\sigma\,B_{0}(x)-a(x)\right\}\,,
$$
and we apply Proposition~\ref{prop:up-blk-cst} with
$$
\mathcal{B}=\kappa H\quad{\rm and}\quad\sigma=\frac{H}{\kappa}\,.
$$
Thus, we get the existence of a positive constant $\kappa_0$ such
that, if,
\begin{equation}\label{cond:H-cst}
\kappa\geq\kappa_{0}\quad{\rm and}\quad\kappa_{0}\,\kappa^{-1}<H< \lambda_{\max}\,\kappa\,,
\end{equation}
then,
\begin{equation}\label{eq:up-mu1-cst}
\mu_{1}(\kappa,H)\leq \kappa^{2}\,\inf_{x\in\Omega}\left\{\frac{H}{\kappa}\,B_{0}(x)-a(x)\right\}+ \mathcal{O}(\kappa)\,,\qquad{\rm as}\,\kappa\to+\infty\,.
\end{equation}
\end{rem}

\begin{prop}[Upper bound near the boundary]\label{prop:up-bnd-cst}
Suppose that $\Omega\subset\R^2$ is an open bounded set with a smooth boundary, $\lambda_{\max} >0$ and $\Gamma=\varnothing$. For any $ x_0\in\partial\Omega$ and for any $\sigma\in (0,\lambda_{\max}]$, we have,
\begin{equation}\label{eq:up-bnd-cst}
\mu_{\mathcal{B},\sigma}\leq \frac{\mathcal B}{\sigma}\big(\sigma\,\Theta_{0}\,B_{0}(x_{0})-a(x_0)\big)+ \mathcal{O}(\mathcal{B}^{\frac{1}{2}})\,,\qquad{\rm as}~\mathcal{B}\to+\infty\,.
\end{equation}
Here, $\Theta_{0}$ is introduced  in \eqref{Theta0}.
\end{prop}

\begin{proof}
We recall the definition of $\mu_{\mathcal{B},\sigma}$ as follows:
$$
\mu_{\mathcal{B},\sigma}=\inf_{u\in H^{1}(\Omega)\setminus \{0\}} \left(\frac{\int_{\Omega}|(\nabla-i\mathcal{B}\Fb)u(x)|^{2}\,dx}{\int_{\Omega}|u(x)|^{2}\,dx}-\frac{\mathcal B}{\sigma}\,\frac{\int_{\Omega}a(x)|u(x)|^{2}\,dx}{\int_{\Omega}|u(x)|^{2}\,dx}\right)\,.
$$
The first term in the right hand side is studied by  Helffer-Morame (see \cite[Theorem~9.1]{HM} with $h=\mathcal{B}^{-1}$ and $\mu_{\mathcal{B},\sigma}=\frac{\mu^{(1)}(h)}{h^{2}}$) or Fournais-Helffer (see \cite[Section~9.2.1]{FH1}). They proved for any $x_0\in \partial \Omega$ the existence of $\mathcal B_0$ such that for $\mathcal{B}\geq \mathcal B_0$  one can construct a trial function $\widehat{u}$ such that,
$$
\frac{\int_{\Omega}|(\nabla-i\mathcal{B}\Fb)\widehat u(x)|^{2}\,dx}{\int_{\Omega}|\widehat u(x)|^{2}\,dx}\leq \mathcal{B}\,\Theta_{0}\,B_{0}(x_{0})+  \mathcal{O}(\mathcal{B}^{\frac{1}{2}})\,,\qquad{\rm as}~\mathcal{B}\to+\infty \,.
$$
The estimates of the second term in the right hand side are just as in \eqref{upp:a-cst} and this achieves the proof of the proposition.
\end{proof}

\begin{rem}\label{rem:second term}
$\partial \Omega$ being compact, we can choose $ x_0\in\partial\Omega$, such that,
$$
\sigma\,\Theta_{0}\,B_{0}(x_0)-a(x_0)=\inf_{x\in\partial\Omega}\left\{\sigma\,\Theta_{0}\,B_{0}(x)-a(x)\right\}\,,
$$
and we apply Proposition~\ref{prop:up-bnd-cst} with
$$
\mathcal{B}=\kappa H\quad{\rm and}\quad \sigma=\frac{H}{\kappa}\,,
$$
which implies under Assumption~\ref{cond:H-cst} that,
\begin{equation}\label{eq:up-mu1-cst2}
\mu_{1}(\kappa,H)\leq \kappa^{2}\,\inf_{ x\in\partial\Omega}\left\{\frac{H}{\kappa}\,\Theta_{0}\,B_{0}(x)-a(x)\right\}+\mathcal{O}(\kappa)\,,\qquad{\rm as}\,\kappa\to+\infty\,.
\end{equation}
\end{rem}

Remarks~\ref{rem:first term} and ~\ref{rem:second term} lead to the conclusion in:
\begin{theorem}\label{thm:mu1-cst-upp}
Let $\Omega\subset\R^2$ is an open bounded set with a smooth boundary and $\Gamma=\varnothing$. Suppose that \eqref{cond:H-cst} hold and $a\in C^1(\overline{\Omega})$, we have
$$
\mu_{1}(\kappa,H)\leq \kappa^{2}\,\Lambda_{1}\left(B_{0},a,\frac H \kappa\right)+  \mathcal{O}(\kappa)\,,\qquad{\rm as}\,\kappa\to+\infty\,.
$$
Here,  $\Lambda_{1}$ is introduced in \eqref{Lambda11}.
\end{theorem}

Notice that the conclusion in Theorem~\ref{thm:mu1-cst-upp} is valid under the assumption $ \kappa H\geq \mathcal B_0$ with  $\mathcal B_0 >0$
sufficiently large. Lemma~\ref{lem-H=kappa} below  takes care of the
regime where $\kappa H=\mathcal O(1)$.

\begin{lem}\label{lem-H=kappa}
Let $C_{\max}>0$. Suppose that $\{a>0\}\not=\emptyset$. There exists a constant $\kappa_0>0$ such that, if
$$\kappa\geq \kappa_0\quad{\rm and}\quad 0\leq H\leq C_{\max}\kappa^{-1}\,,$$
then
$$\mu_1(\kappa,H)<0\,.$$
\end{lem}
\begin{rem}\label{rem:h=0}
The conclusion in Lemma~\ref{lem-H=kappa} is valid in both cases
where $\Gamma=\emptyset$ and $\Gamma\not=\emptyset$.
\end{rem}
\begin{proof}[Proof of Lemma~\ref{lem-H=kappa}]~\\
Let $\ell>0$. Choose $x_0\in\Omega$ such that $a(x_0)>0$. We
introduce a cut-off function $\chi_{\ell}\in C_{c}^{\infty}(\R^{2})$
satisfying:
\begin{equation}\label{def:chil}
0\leq \chi_{\ell} \leq 1~{\rm in}~\R^{2}\,,\quad {\rm supp}\chi_{\ell}\subset B(x_{0},\ell)\,,\quad \chi_{\ell}=1~{\rm in}~B\left(x_{0},\ell/2\right)\quad{\rm and}\quad |\nabla\chi_{\ell}|\leq \frac{C}{\ell}\,.
\end{equation}
The min-max principle yields,
$$
\mu^{(1)}(\kappa,H)\|\chi_{\ell}\|_{L^{2}(\Omega)}^{2}\leq \int_{\Omega}|(\nabla-i\kappa H\Fb)\chi_{\ell}|^{2}\,dx-\kappa^{2}\int_{\Omega}a(x)|\chi_{\ell}(x)|^{2}\,dx\,.
$$
Using the assumptions on $\chi_{\ell}$ and the fact that $\Fb\in C^{\infty}(\overline{\Omega})$, a trivial estimate is,
\begin{align}\label{up:F-varphi3}
\int_{\Omega}|(\nabla-i\kappa H\Fb)\chi_{\ell}|^{2}\,dx&=\int_{B(x_{0},\ell)}|\nabla\chi_{\ell}(x)|^{2}\,dx+\kappa^{2}H^{2}\int_{B(x_{0},\ell)}|\Fb\,\chi_{\ell}(x)|^{2}\,dx\nonumber\\
&\leq C\,(1+(\kappa\,H\,\ell)^{2})\,.
\end{align}
We write by Taylor's formula applied to the function $a$ near $x_0$,
\begin{equation}\label{up:a-varphi3}
-\kappa^{2}\int_{\Omega}a(x)|\chi_{\ell}(x)|^{2}\,dx\leq -a(x_0)\,\kappa^{2}\,\ell^{2}+C\,\kappa^{2}\,\ell^{3}\,.
\end{equation}
Collecting \eqref{up:F-varphi3} and \eqref{up:a-varphi3}, we obtain,
$$
\mu^{(1)}(\kappa,H)\|\chi_{\ell}\|_{L^{2}(\Omega)}^{2}\leq-a(x_0)\,\kappa^{2}\,\ell^{2}+C(\kappa^{2}\,\ell^{3}+1+(\kappa\,H\,\ell)^{2})\,.
$$
We select $\ell=\kappa^{-\frac{1}{2}}$ and note that $\kappa H<C_{\max}$.  We find that,
$$
\mu^{(1)}(\kappa,H)\|\chi_{\ell}\|_{L^{2}(\Omega)}^{2}\leq-a(x_0)\,\kappa+C\Big(\kappa^{\frac{1}{2}}+1+C_{\max}^{2}\kappa^{-1}\Big)\,.
$$
Since $\chi_\ell\not=0$ and $a(x_0)>0$, we deduce that, for $\kappa$
sufficiently large,
$$
\mu^{(1)}(\kappa,H)<0\,.
$$

\end{proof}

\section{Proof of Theorem~\ref{thm:HC3}}\label{10}
\subsection{Analysis of $\underline{H}_{C_3}^{loc}$ and $\overline{H}_{C_3}^{loc}$.}~\\
 In this subsection we give a lower bound of the critical field $\underline{H}_{C_3}^{loc}$ (see \eqref{def:HC3-u}) and we give an upper bound of the critical field $\overline{H}_{C_3}^{loc}$ in the case when the magnetic field $B_0$ is constant with a pining term.
\begin{prop}\label{prop:mu1<0}
Suppose that $\{a>0\}\neq\varnothing$ and $\Gamma=\varnothing$.  There exist constants $C>0$ and $\kappa_{0}\geq 0$ such that if
\begin{equation}\label{cond:HC3}
\kappa\geq \kappa_{0}\,,\qquad H\leq \kappa\,\max\left(\sup_{x\in\Omega}\frac{a(x)}{B_{0}(x)},\sup_{x\in\partial\Omega}\frac{a(x)}{\Theta_{0}\,B_{0}(x)}\right)-C\,\kappa^{\frac{1}{2}}\,,
\end{equation}
then,
$$
\mu_{1}(\kappa,H)<0\,.
$$
Moreover,
$$
 \kappa\,\max\left(\sup_{x\in\Omega}\frac{a(x)}{B_{0}(x)},\sup_{x\in\partial\Omega}\frac{a(x)}{\Theta_{0}\,B_{0}(x)}\right)-C\,\kappa^{\frac{1}{2}}\leq \underline{H}_{C_{3}}^{loc}\,.
$$
\end{prop}
\begin{proof}
To apply the previous results, we  take
$$\lambda_{max} = \max\left(\sup_{x\in\Omega}\frac{a(x)}{B_{0}(x)},\sup_{x\in\partial\Omega}\frac{a(x)}{\Theta_{0}\,B_{0}(x)}\right) +1\,.
$$
We have two cases:\\
\textbf{Case 1.} If
$$\sup_{x\in\Omega}\frac{a(x)}{B_{0}(x)}>\sup_{x\in\partial\Omega}\frac{a(x)}{\Theta_{0}\,B_{0}(x)}\,.$$
then, there exists $x_0\in\Omega$ (the supremum of $\frac{a(x)}{B_{0}(x)}$ can not be attained on the boundary, since $\frac{a(x)}{\Theta_{0}\,B_{0}(x)}>\frac{a(x)}{B_{0}(x)}$), such that,
$$
\sup_{x\in\Omega}\frac{a(x)}{B_{0}(x)} = \frac{a(x_0)}{B_{0}(x_0)} \,.
$$
If \eqref{cond:HC3} is satisfied for some $C>0$, then,
$$
\frac{H}{\kappa}\leq \frac{a(x_0)}{B_{0}(x_0)}-C\,\kappa^{-\frac{1}{2}}\,,
$$
that we can write in the form,
$$
\kappa^{2}\left(\frac{H}{\kappa}B_{0}(x_0)-a(x_0)\right)\leq -C\,M \,\kappa^{\frac{3}{2}}\,,
$$
where $M>0$ is a constant independent of $C$.\\
 Suppose that $\kappa H\geq \mathcal B_0$ where $\mathcal B_0$ is selected sufficiently large such that we can apply Remark~\ref{rem:first term}. (Thanks to Lemma~\ref{lem-H=kappa}, $\mu_1(\kappa,H)<0$ when $\kappa H<\mathcal B_0$).\\
Remark~\ref{rem:first term} tells us that there exist positive
constants $C_{1}$ and $\kappa_{0}$ such that, for $\kappa\geq
\kappa_{0}$,
\begin{align}
\mu_{1}(\kappa,H)&\leq \kappa^{2}\inf_{x\in\Omega} \left(\frac{H}{\kappa}B_{0}(x)-a(x)\right)+C_{1}\,\kappa\nonumber\\
&\leq  \kappa^{2}\left(\frac{H}{\kappa}B_{0}(x_0)-a(x_0)\right)+C_{1}\,\kappa^{\frac{3}{2}}\\
&\leq (C_1-C\,M) \,\kappa^\frac 32 \,.
\end{align}
By choosing $C$ such that $C\,M> C_{1}$, we get,
$$
\mu_{1}(\kappa,H)<0\,.
$$
\textbf{Case 2.} Here, we suppose that
$$\sup_{x\in\partial\Omega}\frac{a(x)}{\Theta_{0}\,B_{0}(x)}\geq \sup_{x\in\Omega}\frac{a(x)}{B_{0}(x)}\,.$$
By compactness, there exists $x'_0\in\partial\Omega$, such that,
$$\sup_{x\in\partial\Omega}\frac{a(x)}{\Theta_{0}\,B_{0}(x)}\ = \frac{a(x'_0)}{\Theta_{0}\,B_{0}(x'_0)}
$$
If \eqref{cond:HC3} is satisfied for some $C>0$, then,
$$
\kappa^{2}\left(\frac{H}{\kappa}\Theta_{0}\,B_{0}(x'_0)-a(x'_0)\right)\leq -C\,M'\,\kappa^{\frac{3}{2}}\,.
$$
Thanks to Remark~\ref{rem:second term}, we get the existence of positive constants $C_{2}$ and $\kappa_{0}$ such that, for $\kappa\geq \kappa_0$,
\begin{align}
\mu_{1}(\kappa,H)&\leq \kappa^{2}\inf_{x\in\partial\Omega} \left(\frac{H}{\kappa}\,\Theta_{0}\,B_{0}(x)-a(x)\right)+C_{2}\,\kappa\nonumber\\
&\leq  \kappa^{2}\left(\frac{H}{\kappa}\,\Theta_{0}\,B_{0}(x'_0)-a(x'_0)\right)+C_{2}\,\kappa^{\frac{3}{2}}\\
&\leq (C_2-C\,M')\,\kappa^\frac 32\,.
\end{align}
By choosing $C$ such that $C\,M'> C_{2}$, we get,
$$
\mu_{1}(\kappa,H)<0\,.
$$
This finishes the proof of the proposition.
\end{proof}
\begin{prop}\label{prop:mu>0}
Suppose that $\{a>0\}\neq\varnothing$, $\lambda_{\max} >0$ and $\Gamma=\varnothing$. There exist constants $C>0$ and $\kappa_{0} > 0$ such that if
\begin{equation}\label{cond:HC3-2}
\kappa\geq \kappa_{0}\,,\qquad \lambda_{\max}\,\kappa\geq H>\kappa\,\max\left(\sup_{x\in\Omega}\frac{a(x)}{B_{0}(x)},\sup_{x\in\partial\Omega}\frac{a(x)}{\Theta_{0}\,B_{0}(x)}\right)+C\,\kappa^{\frac{1}{2}}\,,
\end{equation}
then,
$$
\mu_{1}(\kappa,H)>0\,.
$$
 Moreover,
$$
\overline{H}_{C_{3}}^{loc}\leq\kappa\,\max\left(\sup_{x\in\Omega}\frac{a(x)}{B_{0}(x)},\sup_{x\in\partial\Omega}\frac{a(x)}{\Theta_{0}\,B_{0}(x)}\right)+C\,\kappa^{\frac{1}{2}}\,.
$$
\end{prop}
\begin{proof}
To apply the previous results, we  take
$$\lambda_{min} = \frac 12 \max\left(\sup_{x\in\Omega}\frac{a(x)}{B_{0}(x)},\sup_{x\in\partial\Omega}\frac{a(x)}{\Theta_{0}\,B_{0}(x)}\right) \,.$$
If \eqref{cond:HC3-2} holds for some $C>0$, then, for any  $x\in\Omega$, we have,
\begin{equation}\label{eq:bulk}
\frac{H}{\kappa}B_{0}(x)-a(x)\geq C\,\kappa^{-\frac{1}{2}}\,,
\end{equation}
and, for any $x'\in\partial\Omega$, we have,
\begin{equation}\label{eq:bnd}
\frac{H}{\kappa}\Theta_{0}B_{0}(x')-a(x')\geq C\,\kappa^{-\frac{1}{2}}\,.
\end{equation}
 Having in mind the definition of $\Lambda_1$ in \eqref{Lambda11}, the estimates in \eqref{eq:bulk} and in \eqref{eq:bnd} give us that for $\kappa$ large enough,
$$
\Lambda_{1}\left(B_0,a,\frac{H}{\kappa}\right)\geq C\,\kappa^{-\frac{1}{2}}\,.
$$
Thanks to Theorem~\ref{thm:mu1-cst}, we get the existence of positive constants $C'$ and $\kappa_0$ such that, for $\kappa\geq\kappa_0$,
\begin{align}
\mu_{1}(\kappa,H)&\geq \kappa^{2}\,\Lambda_{1}\left(B_{0},a,\frac H \kappa\right)-C'\,\kappa^{\frac{3}{2}}\nonumber\\
&\geq (C-C')\,\kappa^{\frac{3}{2}}\,.
\end{align}
To finish this proof, we choose $C>C'$.
\end{proof}
As a consequence, we have proved Theorem~\ref{thm:HC3} for $\underline{H}_{C_3}^{loc}$ and $\overline{H}_{C_3}^{loc}$
\subsection{Analysis of $\underline{H}_{C_3}^{cp}$ and $\overline{H}_{C_3}^{cp}$.}~\\
 In this subsection we give a lower bound of the critical field $\underline{H}_{C_3}^{cp}$ (see \eqref{def:HC3-o}) and we give an upper bound of the critical field $\overline{H}_{C_3}^{cp}$ in the case when the magnetic field $B_0$ is constant with a pining term. We start with a proposition which measures the effect of the localization at the boundary when $H$ is sufficiently large.
\begin{prop}\label{lem:psi-2<4}
Suppose that $\Gamma=\varnothing$ and \eqref{cond:HC3-2} holds. There exists a positive constant $C$, such that if $(\psi,\Ab)$ is a solution of \eqref{eq-2D-GLeq}, then the following estimate holds:
\begin{equation}
\|\psi\|^{2}_{L^{2}(\Omega)}\leq C\,\kappa^{-\frac{3}{8}}\|\psi\|^{2}_{L^{4}(\Omega)}\,.
\end{equation}
\end{prop}

\begin{proof}~\\
The techniques that will be used in this proof are similar with the ones in \cite[Lemma~2.6]{FK2}. If $H$ satisfies \eqref{cond:HC3-2} for some $C>0$, then, for any  $x\in\Omega$, we have.
\begin{equation}\label{bulk2}
\kappa\,H\,B_{0}(x)-\kappa^{2}\,a(x)\geq C\,\kappa^{\frac{3}{2}}\,.
\end{equation}
First, we let $\chi\in C^{\infty}(\R)$ be a standard cut-off function such that
\begin{equation}\label{def:chi-cst}
\chi=1\quad{\rm in}~[1,\infty]\qquad{\rm and }\qquad\chi=0\quad{\rm in}~]-\infty,1/2]\,.
\end{equation}
Next, we define $\lambda=\kappa^{-\frac{3}{4}}$, and $\chi_{\kappa}$ as follows:
\begin{equation}
\chi_{\kappa}(x)=\chi\left(\frac{\dist(x,\partial\Omega)}{\lambda}\right)\,,\qquad \forall x\in\Omega\,.
\end{equation}
Referring to \eqref{eq:1}, we have
\begin{equation}\label{eq:2}
\int_{\Omega} \left(|(\nabla-i\kappa H \Ab)\chi_{\kappa}\psi|^{2}-|\nabla\chi_{\kappa}|^{2}|\psi|^{2}\right)\,dx=\kappa^{2}\int_{\Omega}|\chi_{\kappa}|^{2}(a(x)-|\psi|^{2})|\psi|^{2}\,dx\,.
\end{equation}
We estimate $\int_\Omega|(\nabla-i\kappa H\Ab)\chi_{\kappa}\psi|^2\,dx$ from below. As in \cite[Proposition~6.2]{HK}, we can prove that,
$$
\int_{\Omega}|(\nabla-i\kappa H\Ab)\chi_{\kappa}\psi|^2\,dx\geq \kappa\,H\int_\Omega \curl\Fb \,|\chi_{\kappa}\psi|^2\,dx-\kappa\,H\|\curl(\Ab-\Fb)\|_{L^{2}(\Omega)}\|\chi_{\kappa}\psi\|_{L^{4}(\Omega)}^{2}\,.
$$
 Noticing that $\displaystyle \curl\Fb=B_{0}(x)$ and  $\displaystyle\|\curl(\Ab-\Fb)\|_{L^{2}(\Omega)}\leq \frac{c}{H}\|\psi\|_{L^{2}(\Omega)}$,  we have,
$$
\int_{\Omega}|(\nabla-i\kappa H\Ab)\chi_{\kappa}\psi|^2\,dx\geq \kappa\,H\int_\Omega\,B_{0}(x)\,|\chi_{\kappa}\psi|^2\,dx-c\,\kappa\,\|\psi\|_{L^{2}(\Omega)}\|\chi_{\kappa}\psi\|_{L^{4}(\Omega)}^{2}\,.
$$
Implementing a Cauchy-Schwarz inequality, we get
\begin{equation}\label{eq:grd-eng}
\int_{\Omega}|(\nabla-i\kappa H\Ab)\chi_{\kappa}\psi|^2\,dx\geq \kappa\,H\int_\Omega\,B_{0}(x)\,|\chi_{\kappa}\psi|^2\,dx-c^{2}\,\|\psi\|_{L^{2}(\Omega)}^{2}-\kappa^{2}\|\chi_{\kappa}\psi\|_{L^{4}(\Omega)}^{4}\,.
\end{equation}
Inserting \eqref{eq:grd-eng} into \eqref{eq:2}, we obtain,
$$
\int_\Omega\,\left(\kappa\,H\,B_{0}(x)-\kappa^{2}\,a(x)\right)\,|\chi_{\kappa}\psi|^2\,dx\leq c^{2}\int_{\Omega}|\psi|^{2}\,dx+\int_{\Omega}|\nabla\chi_{\kappa}|^{2}|\psi|^{2}\,dx-\kappa^{2}\int_{\Omega}\left(\chi_{\kappa}^{2}-\chi_{\kappa}^{4}\right)|\psi|^{4}\,dx\,.
$$
As a consequence of \eqref{bulk2}, the inequality above becomes,
$$
C\,\kappa^{\frac{3}{2}}\int|\chi_{\kappa}\psi(x)|^{2}\,dx\leq c^{2}\int_{\Omega}|\psi|^{2}\,dx+\int_{\Omega}|\nabla\chi_{\kappa}|^{2}|\psi|^{2}\,dx-\kappa^{2}\int_{\Omega}\left(\chi_{\kappa}^{2}-\chi_{\kappa}^{4}\right)|\psi|^{4}\,dx\,.
$$
Notice that $-\kappa^{2}\int_{\Omega}\left(\chi_{\kappa}^{2}-\chi_{\kappa}^{4}\right)|\psi|^{4}\,dx\leq 0\,$.\\
Decomposing the integral $\displaystyle\int_{\Omega}|\psi|^{2}\,dx=\int_{\Omega}|\chi_{\kappa}\psi|^{2}\,dx+\int_{\Omega}(1-\chi_{\kappa}^{2})|\psi|^{2}\,dx$, using \eqref{bulk2} and choosing $C$ such that $C\geq 2 c^{2}$, we get,
$$
\frac{1}{2}C\,\kappa^{\frac{3}{2}}\int|\chi_{\kappa}\psi(x)|^{2}\,dx\leq \left(c^{2}+\|\chi'\|_{L^{\infty}(\R)}^{2}\,\lambda^{-2}\right)\int_{\left\{x\in\Omega:\,\dist(x,\Gamma)\leq\lambda\right\}}|\psi|^{2}\,dx\,.
$$
Recall that $\lambda=\kappa^{-\frac{3}{4}}$, we observe that,
$$
\int|\chi_{\kappa}\psi(x)|^{2}\,dx\leq 4\|\chi'\|_{L^{\infty}(\R)}^{2} \int_{\left\{x\in\Omega:\,\dist(x,\Gamma)\leq\lambda\right\}}|\psi|^{2}\,dx\,,
$$
and consequently, we get,
$$
\int|\psi(x)|^{2}\,dx\leq\left(4\|\chi'\|_{L^{\infty}(\R)}^{2}+1\right) \int_{\left\{x\in\Omega:\,\dist(x,\Gamma)\leq\lambda\right\}}|\psi|^{2}\,dx\,.
$$
By choosing $C=\max \left(2\,c^{2},4\|\chi'\|_{L^{\infty}(\R)}^{2}+1\right)$, we obtain,
$$
\|\psi\|^{2}_{L^{2}(\Omega)}\leq C\,\kappa^{-\frac{3}{8}}\|\psi\|^{2}_{L^{4}(\Omega)}\,.
$$
\end{proof}

\begin{theorem}\label{thm:lb-H}
Supose that $\Gamma=\varnothing$ and $\{a>0\}\neq\varnothing$. There exists $C>0$ and $\kappa_0$ such that, if $H$ satisfies
\begin{equation}\label{cond:HC3-2w}
 H>\kappa\,\max\left(\sup_{x\in\Omega}\frac{a(x)}{B_{0}(x)},\sup_{x\in\partial\Omega}\frac{a(x)}{\Theta_{0}\,B_{0}(x)}\right)+C\,\kappa^{\frac{1}{2}}\,,
\end{equation} then $(0,\Fb)$ is the unique solution to \eqref{eq-2D-GLeq}.\\
 Moreover,
$$
\overline{H}_{C_3}^{cp}\leq\kappa\,\max\left(\sup_{x\in\Omega}\frac{a(x)}{B_{0}(x)},\sup_{x\in\partial\Omega}\frac{a(x)}{\Theta_{0}\,B_{0}(x)}\right)+C\,\kappa^{\frac{1}{2}}\,.
$$
\end{theorem}

\begin{proof}
We first observe that it results from Giorgi-Phillips like Theorem  \ref{thm:GP2} that it remains only  to prove the theorem under the stronger Assumption \eqref{cond:HC3-2}.
Suppose now that  $(\psi,\Ab)$ is a solution of \eqref{eq-2D-GLeq} with $\psi\neq 0$, we observe that,
\begin{equation}\label{l5est}
0<\kappa^{2}\|\psi\|_{L^{4}(\Omega)}^{4}=-\int_{\Omega}\left(|(\nabla-i\kappa H\Ab)\psi|^2-\kappa^{2}a(x)|\psi|^{2}\right)\,dx:=\top\,.
\end{equation}
We can write,
 \begin{align}\label{main-eq}
-\top&\geq (1-\sqrt{\top}\,\kappa^{-1})\int_{\Omega}|(\nabla-i\kappa H\Fb)\psi|^{2}\,dx-\kappa^{2}\,\int_{\Omega}\,a(x)|\psi|^{2}\,dx-\frac{(\kappa H)^{2}}{\sqrt{\top}\kappa^{-1}}\int_{\Omega}|(\Ab-\Fb)\psi|^{2}\,dx\nonumber\\
&\geq \mu_{1}(\kappa,H)\,\|\psi\|_{L^{2}(\Omega)}^{2}-\sqrt{\top}\,\kappa^{-1}\|(\nabla-i\kappa H\Fb)\psi\|^{2}_{L^{2}(\Omega)}-\frac{(\kappa H)^{2}}{\sqrt{\top}\kappa^{-1}}\|(\Ab-\Fb)\psi\|_{L^{2}(\Omega)}^{2}\,.
\end{align}
We reffer to \eqref{8d-<} and \eqref{5d-<}, we have,
\begin{equation}\label{est:top1}
-\top\geq \mu_{1}(\kappa,H)\,\|\psi\|_{L^{2}(\Omega)}^{2}-C\,\sqrt{\top}\,\kappa\,\|\psi\|_{L^{2}(\Omega)}^{2}\,.
\end{equation}
Thanks to Proposition~\ref{lem:psi-2<4}, using \eqref{l5est}, we get,
\begin{equation}\label{est:psi1}
\|\psi\|^{2}_{L^{2}(\Omega)}\leq C\,\kappa^{-\frac{11}{8}}\,\sqrt{\top}\,.
\end{equation}
As a consequence of \eqref{est:psi1}, \eqref{est:top1} becomes,
\begin{equation}\label{est:top2}
-\top\geq \mu_{1}(\kappa,H)\,\|\psi\|_{L^{2}(\Omega)}^{2}-C'\,\kappa^{-\frac{3}{8}}\,\top\,.
\end{equation}
Having in mind that $\psi\neq 0$ and $\top>0$ (see \eqref{l5est}), we deduce for $\kappa$ sufficiently large $\mu_{1}(\kappa,H)<0$, which is in contradiction with Proposition~\ref{prop:mu>0}. Therefore, we conclude that $\psi=0$, which is what we needed to prove.
\end{proof}

\begin{prop}\label{prop:cp}
Supose that $\Gamma=\varnothing$ and $\{a>0\}\neq\varnothing$. There exists $C>0$ and $\kappa_0$ such that, if $H$ satisfies
\begin{equation}\label{cond:HC3-2w2}
 H\leq \kappa\,\max\left(\sup_{x\in\Omega}\frac{a(x)}{B_{0}(x)},\sup_{x\in\partial\Omega}\frac{a(x)}{\Theta_{0}\,B_{0}(x)}\right)-C\,\kappa^{\frac{1}{2}}\,,
\end{equation}
then there exists a solution $(\psi,\Ab)$ of \eqref{eq-2D-GLeq} with $\|\psi\|_{L^{2}(\Omega)}\neq 0$.\\
 Moreover,
$$
\kappa\,\max\left(\sup_{x\in\Omega}\frac{a(x)}{B_{0}(x)},\sup_{x\in\partial\Omega}\frac{a(x)}{\Theta_{0}\,B_{0}(x)}\right)-C\,\kappa^{\frac{1}{2}}\leq \underline{H}_{C_3}^{cp}\,.
$$
\end{prop}
\begin{proof}
We use $(t\psi_{\ast},\Fb)$, with $t$ sufficiently small and $\psi_{\ast}$ an eigenfunction associated with $\mu_{1}(\kappa,H)$, as a test configuration for the functional \eqref{eq-2D-GLf}, i.e.
 $$
\int_{\Omega}\left(|(\nabla-i\kappa H\Fb)\psi_{\ast}|^{2}-\kappa^{2}\,a(x)|\psi_{\ast}|^{2}\right)\,dx=\mu_{1}(\kappa,H)\|\psi_{\ast}\|^{2}_{L^{2}(\Omega)}\,.
$$
Proposition~\ref{prop:mu1<0} tells us that there exists a constant $C$ such that, under Assumption \eqref{cond:HC3-2w2}, $\mu_{1}(\kappa,H)<0\,$.\\
Therefore,
$$
C_{1}(\kappa,H):=\int_{\Omega}\left(|(\nabla-i\kappa H\Fb)\psi_{\ast}|^{2}-\kappa^{2}\,a(x)|\psi_{\ast}|^{2}\right)\,dx<0\,.
$$
We can write,
 \begin{align*}
\mathcal E_{\kappa,H,a,B_{0}}(t\psi_{\ast},\Fb)&=t^{2}\int_{\Omega}\left(|(\nabla-i\kappa H\Fb)\psi_{\ast}|^{2}-\kappa^{2}\,a(x)|\psi_{\ast}|^{2}\right)\,dx+t^{4}\,\frac{\kappa^2}{2}\int_{\Omega}|\psi_{\ast}|^{4}\,dx+\frac{\kappa^2}{2}\int_{\Omega}a(x)\,dx\\
&=t^2\left(C_{1}(\kappa,H)+t^{2}\,\frac{\kappa^2}{2}\int_{\Omega}|\psi_{\ast}|^{4}\,dx\right)+ \mathcal E_{\kappa,H,a,B_{0}}(0,\Fb)\,.
\end{align*}
We choose $t$ such that,
$$
C_{1}(\kappa,H)+t^{2}\,\frac{\kappa^2}{2}\int_{\Omega}|\psi_{\ast}|^{4}\,dx<0\,.
$$
Thus, we get
$$
\mathcal E_{\kappa,H,a,B_{0}}(t\psi_{\ast},\Fb)<\mathcal E_{\kappa,H,a,B_{0}}(0,\Fb)\,.
$$
Hence a minimizer, which is a solution of \eqref{eq-2D-GLeq}, will be non-trivial.
\end{proof}
\subsection{End of the proof of Theorem~\ref{thm:HC3}}
First, we will prove the following inclusion,
$$
\mathcal{N}^{\rm loc}(\kappa)\subset \mathcal{N}(\kappa)\,.
$$
We see that if $H\notin  \mathcal{N}(\kappa)$, then $(0,\Fb)$ is a local minimizer of $\mathcal E_{\kappa,H,a,B_{0}}$. Thus, the Hessian of the functional $\mathcal E_{\kappa,H,a,B_{0}}$ at the normal state $(0,\Fb)$ should be positive.\\
For every $(\widetilde{\phi},\widetilde{\Ab})\in H^1(\Omega)\times H^1_{\rm div}(\Omega)$   we have,
$$
\mathcal E_{\kappa,H,a,B_{0}}(t\widetilde{\phi},\Fb+t\widetilde{\Ab})=t^{2}\left[\mathcal{Q}_{\kappa H\Fb,-\kappa^{2}a}^{\Omega}(\widetilde{\phi})+(\kappa H)^{2}\int_{\Omega}|\curl \widetilde{\Ab}|^{2}\,dx\right]+\mathcal{O}(t^3)\,.
$$
This implies that the Hessian of the functional $\mathcal E_{\kappa,H,a,B_{0}}$ at the normal state $(0,\Fb)$ can be written as follows:
$$
Hess_{(0,\Fb)}[\widetilde{\phi},\widetilde{\Ab}]=\mathcal{Q}_{\kappa H\Fb,-\kappa^{2}a}^{\Omega}(\widetilde{\phi})+(\kappa H)^{2}\int_{\Omega}|\curl \widetilde{\Ab}|^{2}\,dx\,.
$$
Since $Hess_{(0,\Fb)}[\widetilde{\phi},\widetilde{\Ab}]\geq 0$, we get that $\mu_{1}(\kappa H)\geq 0\,,$ and consequently $H\notin \mathcal{N}^{\rm loc}(\kappa)$.
Hence we obtain the above inclusion.\\
On the other hand, if $(\psi,\Ab)$ is a minimizer of the functional in \eqref{eq-2D-GLf} with $\psi\neq 0$, then $(\psi,\Ab)$ is a solution of \eqref{eq-2D-GLeq}, and we have the following inclusion,
$$
\mathcal{N}(\kappa)\subset \mathcal{N}^{\rm cp}(\kappa)\,,
$$
and consequently,
\begin{equation}\label{inclusion-N}
\mathcal{N}^{\rm loc}(\kappa)\subset \mathcal{N}(\kappa)\subset \mathcal{N}^{\rm cp}(\kappa)\,.
\end{equation}
Having in mind the definition of all the critical fields in \eqref{def:HC3-o}, \eqref{def:HC3} and \eqref{def:HC3-u}, we deduce that,
\begin{equation}\label{eq:ov}
\overline{H}_{C_3}^{loc}(\kappa)\leq\overline{H}_{C_3}(\kappa)\leq\overline{H}_{C_3}^{cp}(\kappa)\,,
\end{equation}
Using \eqref{inclusion-N}, we observe that,
$$
\R^{+}\setminus\mathcal{N}^{\rm cp}(\kappa)\subset\R^{+}\setminus\mathcal{N}(\kappa)\subset\R^{+}\setminus\mathcal{N}^{\rm loc}(\kappa)\,.
$$
From the definition of all the critical fields, we conclude that,
\begin{equation}\label{eq:un}
\underline{H}_{C_3}^{loc}(\kappa)\leq\underline{H}_{C_3}(\kappa)\leq\underline{H}_{C_3}^{cp}(\kappa)\,.
\end{equation}
We note that $\underline{H}_{C_{3}}^{loc}\leq\overline{H}_{C_{3}}^{loc}$ and $\underline{H}_{C_{3}}^{cp}\leq\overline{H}_{C_{3}}^{cp}$. Therefore, all the critical fields are contained in the interval $[\underline{H}_{C_{3}}^{loc},\overline{H}_{C_{3}}^{cp}]$.\\
By Proposition~\ref{prop:mu1<0} and Theorem~\ref{thm:lb-H}, we get the existence of positive constants $C$ and $\kappa_0$, such that for $\kappa\geq\kappa_0$,
\begin{multline}
\kappa\,\max\left(\sup_{x\in\Omega}\frac{a(x)}{B_{0}(x)},\sup_{x\in\partial\Omega}\frac{a(x)}{\Theta_{0}\,B_{0}(x)}\right)-C\,\kappa^{\frac{1}{2}}\leq \underline{H}_{C_{3}}^{loc}\leq\overline{H}_{C_{3}}^{cp}\\
\leq\kappa\,\max\left(\sup_{x\in\Omega}\frac{a(x)}{B_{0}(x)},\sup_{x\in\partial\Omega}\frac{a(x)}{\Theta_{0}\,B_{0}(x)}\right)+C\,\kappa^{\frac{1}{2}}\,.
\end{multline}
As a consequence, we have proved Theorem~\ref{thm:HC3} for the six critical fields.
\begin{rem}
As in \cite{FH1}, it would be interesting to show that all the critical fields coincide when $\kappa$ is large enough.
\end{rem}
\section{Asymptotics of $\mu_{1}(\kappa,H)$: the case with  vanishing magnetic field}\label{Section:Asympt-m1-vanish}
 In this section we give an estimate for the lowest eigenvalue
$\mu_{1}(\kappa,H)$ of the operator $P_{\kappa
H\Fb,-\kappa^{2}a}^{\Omega}$ (see \eqref{def:mu1}) in the case when
$\Gamma=\varnothing$ with a $\kappa$-independent pinning, i.e.
$a(\kappa,x)=a(x)$. The results in this section are valid under the
assumption $\Gamma\not=\emptyset$, where the set $\Gamma$ is introduced in
\eqref{gamma}. Let
\begin{equation}
\mathcal{B}=\kappa H\qquad{\rm and}\qquad \widehat{\sigma}=\frac{H}{\kappa^2}\,.
\end{equation}
We observe that,
$$
P_{\kappa H\Fb,-\kappa^{2}a}^{\Omega}=P_{\mathcal{B}\Fb,-\left(\frac{\mathcal{B}}{\widehat{\sigma}}\right)^{\frac{2}{3}}\,a}^{\Omega}\,.
$$
We will give an estimate for the lowest eigenvalue
$\mu_{\mathcal{B},\widehat{\sigma}}$ of
$P_{\mathcal{B}\Fb,-\left(\frac{\mathcal{B}}{\widehat{\sigma}}\right)^{\frac{2}{3}}\,a}^{\Omega}$.
After a change of notation, we deduce an estimate for
$\mu_{1}(\kappa,H)$.

\subsection{Lower bound}
 In the
absence of a pinning term, that is when $a=1$, Pan and Kwek
\cite{XB-KH} gave the lower bound for the lowest eigenvalue
$\mu(\mathcal{B}\Fb)$ of $P_{\mathcal{B}\Fb,0}^{\Omega}$ when
$\mathcal{B}\to+\infty$.
In this subsection, we determine a lower bound for $\mu_{1}$ when $\kappa\to+\infty$ and the pinning term is present.\\
We first recall the definition of $\lambda_{0}$ in \eqref{lambda0}, the definition of $\Gamma$ in \eqref{gamma} and for any $\theta\in(0,\pi)$ we recall that $\lambda(\R_{+}^{2},\theta)$ is the bottom of the spectrum of the operator
$ P_{\Ab_{\rm app,\theta},0}^{\R^{2}_{+}}$, with
$$\Ab_{\rm app,\theta}=-\left(\frac{x^{2}_{2}}{2}\cos\,\theta,\frac{x^{2}_{1}}{2}\sin\,\theta \right)\,.$$
We then define for any $\widehat{\sigma}>0$,
\begin{multline}\label{Lambda1}
\widehat\Lambda_{1}(B_{0},a,\widehat{\sigma})=\min\left\{\inf_{x\in \Gamma\cap\Omega}\left\{\lambda_{0}\,\Big(\widehat{\sigma}\,|\nabla B_{0}(x)|\Big)^{\frac{2}{3}}-a(x)\right\},\right.\\
\left.\inf_{x\in \Gamma\cap\partial\Omega}\left\{\lambda(\R^{2}_{+},\theta(x))\,\Big(\widehat{\sigma}\,|\nabla B_{0}(x)|\Big)^{\frac{2}{3}}-a(x)\right\}\right\}\,.
\end{multline}
Here, for $x\in\Gamma\cap\partial\Omega$, $\theta(x)$ denotes the angle between $\nabla B_{0}(x)$ and the inward normal vector $-\nu(x)$.

We start with a proposition that states a lower bound of $\mu_{1}(\kappa,H)$
 in the case when $\Gamma\neq\varnothing$.
\begin{prop}\label{prop:mu1-variable}
Let $I$ be a closed interval in $(0,\infty)$. There exist two
positive constants $\mathcal{B}_{0}>0$ and $C>0$ such that if
$\mathcal{B}\geq\mathcal{B}_{0}$, $\widehat\sigma\in I$, $\psi\in
H^{1}(\Omega)\setminus \{0\} $ and $a\in C^{1}(\overline{\Omega})$,
then,
\begin{equation}\label{eq:lower-bound-variable}
\frac{\mathcal{Q}_{\mathcal{B}\Fb,-\left(\frac{\mathcal{B}}{\widehat{\sigma}}\right)^{\frac{2}{3}}\,a}^{\Omega}(\psi)}{\|\psi\|^{2}_{L^{2}(\Omega)}}\geq
\left(\frac{\mathcal{B}}{\widehat{\sigma}}\right)^{\frac{2}{3}}\Big(\widehat\Lambda_{1}(B_{0},a,\widehat{\sigma})
-C\mathcal{B}^{-\frac{1}{18}}\Big)\,.
\end{equation}
\end{prop}
\begin{proof}
Let $\ell=B^{-7/29}$. We define the following sets,
\begin{align*}
&D_{1}=\{x\in\Omega: \dist(x,\Gamma)<2\,\ell\}\,,
&D_{2}=\{x\in\Omega: \dist(x,\Gamma)>\ell\}\,.
\end{align*}
Let $h_{j}$ be a partition of unity satisfying
$$
\sum_{j=1}^{2} h_{j}^{2}=1\,,\qquad \sum_{j=1}^{2}|\nabla h_{j}|^{2}\leq C\,\ell^{-2}=C\mathcal B^{14/29}\qquad{\rm and}\qquad \supp h_{j}\subset D_{j}\quad(j\in\{1,2\})\,.
$$
There holds the following decomposition,
\begin{multline}\label{eq:pu}
\mathcal{Q}_{\mathcal{B}\Fb,-\left(\frac{\mathcal{B}}{\widehat{\sigma}}\right)^{\frac{2}{3}} a}^{\Omega}(\psi)
=\mathcal{Q}_{\mathcal{B}\Fb,-\left(\frac{\mathcal{B}}{\widehat{\sigma}}\right)^{\frac{2}{3}}a}^{D_{1}}(h_{1}\psi)+\mathcal{Q}_{\mathcal{B}\Fb,-\left(\frac{\mathcal{B}}{\widehat{\sigma}}\right)^{\frac{2}{3}}a}^{D_{2}}(h_{2}\psi)-\sum_{j=1}^{2}\int_{\Omega}|\nabla h_{j}|^{2}|\psi|^{2}\,dx\\
\geq \mathcal{Q}_{\mathcal{B}\Fb,-\left(\frac{\mathcal{B}}{\widehat{\sigma}}\right)^{\frac{2}{3}}a}^{D_{1}}(h_{1}\psi)+\mathcal{Q}_{\mathcal{B}\Fb,-\left(\frac{\mathcal{B}}{\widehat{\sigma}}\right)^{\frac{2}{3}}a}^{D_{2}}(h_{2}\psi)
-C\mathcal B^{14/29}\int_{\Omega}|\psi|^{2}\,dx\,.
\end{multline}
We cover the curve $\Gamma$ by a family of disks
$$D(\omega_{j},\ell)\subset\{x\in\R^{2}:\dist(x,\Gamma)\leq 2\ell\}\qquad{\rm and}\qquad D_{1}\subset\bigcup_{j} D(\omega_{j},\ell) \qquad\left(\omega_{j}\in \Gamma\right)\,.$$
Consider a partition of unity satisfying
$$
\sum_{j} \chi_{j}^{2}=1\,,\qquad \sum_{j} |\nabla \chi_{j}|^{2}\leq C\,\ell^{-2}\qquad{\rm and}\qquad \supp \chi_{j}\subset D(\omega_{j},\ell)\,.
$$
Moreover, we can add the property that:
$$
{\rm either~supp}\chi_{j}\cap\Gamma\cap\partial\Omega=\varnothing\,,\quad{\rm either}~\omega_{j}\in\Gamma\cap\partial\Omega\,.
$$
We may write,
\begin{equation}\label{eq:main-bulk+bnd}
\mathcal{Q}_{\mathcal{B}\Fb,-\left(\frac{\mathcal{B}}{\widehat{\sigma}}\right)^{\frac{2}{3}} a}^{D_{1}}(h_{1}\psi)=\sum_{int}\mathcal{Q}_{\mathcal{B}\Fb,-\left(\frac{\mathcal{B}}{\widehat{\sigma}}\right)^{\frac{2}{3}}a}^{D_{1}}(\chi_{j}h_{1}\psi)+\sum_{bnd}\mathcal{Q}_{\mathcal{B}\Fb,-\left(\frac{\mathcal{B}}{\widehat{\sigma}}\right)^{\frac{2}{3}}a}^{D_{1}}(\chi_{j}h_{1}\psi)-\sum_{j}\int_{D_{1}}|\nabla\chi_{j}|^{2}|h_{1}\psi|^{2}\,dx\,,
\end{equation}
where `int' is in reference to the $j$'s such that $\omega_{j}\in\Gamma\cap\Omega$ and `bnd' is in reference to the $j$'s such that $\omega_{j}\in\Gamma\cap\partial\Omega$.\\
For the last term on the right side of \eqref{eq:main-bulk+bnd}, we get using the assumption on $\chi_{j}$:
\begin{equation}\label{eq:est-error}
\int_{D_{1}}|\nabla\chi_{j}|^{2}|h_{1}\psi|^{2}\,dx\leq C\,\ell^{-2}\,\int_{D_{1}}|h_{1}\psi|^{2}\,dx
= C\,\mathcal B^{14/29}\,\int_{D_{1}}|h_{1}\psi|^{2}\,dx\,.
\end{equation}
We have to find a lower bound for
$\mathcal{Q}_{\mathcal{B}\Fb,-\left(\frac{\mathcal{B}}{\widehat{\sigma}}\right)^{\frac{2}{3}}a}^{D_{1}}(h_{1}\psi)$
for each $j$ such that $\omega_{j}\in\Gamma\cap\Omega$ and for each
$j$ such that $\omega_{j}\in\Gamma\cap\partial\Omega$. Thanks to
\cite{JPM},  we have,
$$\int_\Omega|(\nabla-i\mathcal B\Fb)\chi_jh_1\psi|^2\,dx\geq \mathcal
B^{\frac23}\int_\Omega\Big((\lambda_{0}\,|\nabla B_{0}(\omega_{j})|\Big)^{\frac{2}{3}}-CB^{-1/18}\Big)|\chi_jh_1\psi|^2\,dx\,.
$$
Using Taylor's formula, we can write in every disk $D(w_j,\ell)$,
\begin{equation}\label{eq:aT}
|a(x)-a(w_j)|\leq C\ell=C\mathcal B^{-7/29}\leq C\mathcal B^{-1/18}\,.
\end{equation}
In that way, we get,
\begin{align}\label{eq:lb-bulk-vr}
&\sum_{int}\mathcal{Q}_{\mathcal{B}\Fb,-\left(\frac{\mathcal{B}}{\widehat{\sigma}}\right)^{\frac{2}{3}}a}^{D_{1}}(\chi_{j}h_{1}\psi)\nonumber\\
&\quad\geq\sum_{int}\left(\frac{\mathcal{B}}{\widehat{\sigma}}\right)^{\frac{2}{3}} \left(\lambda_{0}\,\Big(\widehat{\sigma}\,|\nabla B_{0}(\omega_{j})|\Big)^{\frac{2}{3}}-a(\omega_{j})-C\mathcal B^{-1/18}
\right) \int|\chi_{j}h_{1}\psi|^{2}\,dx\nonumber\\
&\quad\geq\left(\frac{\mathcal{B}}{\widehat{\sigma}}\right)^{\frac{2}{3}}\left(
\inf_{x\in \Gamma\cap\Omega}\left\{\lambda_{0}\,\Big(\widehat{\sigma}\,|\nabla B_{0}(x)|\Big)^{\frac{2}{3}}-a(x)\right\}-C\mathcal B^{-1/18}\right)
\sum_{int} \int|\chi_{j}h_{1}\psi|^{2}\,dx\,.
\end{align}
In a similar fashion, the analysis in \cite{JPM} and \eqref{eq:aT}
yields,
\begin{align}\label{eq:lb-bnd-vr}
&\sum_{bnd}
\mathcal{Q}_{\mathcal{B}\Fb,-\left(\frac{\mathcal{B}}{\widehat{\sigma}}\right)^{\frac{2}{3}}a}^{D_{1}}(\chi_{j}h_{1}\psi)\nonumber\\
&\quad\geq\sum_{bnd}
\left(\frac{\mathcal{B}}{\widehat{\sigma}}\right)^{\frac{2}{3}}
\left(\lambda(\R^{2}_{+},\theta(\omega_{j}))\,\Big(\widehat{\sigma}\,|\nabla B_{0}(\omega_{j})|\Big)^{\frac{2}{3}}-a(\omega_{j})-C\mathcal B^{-1/18}\right)
\int|\chi_{j}h_{1}\psi|^{2}\,dx\nonumber\\
&\quad\geq\left(\frac{\mathcal{B}}{\widehat{\sigma}}\right)^{\frac{2}{3}}
\left(\inf_{x\in \Gamma\cap\partial\Omega}\left\{\lambda(\R^{2}_{+},\theta(x))\,\Big(\widehat{\sigma}\,|\nabla B_{0}(x)|\Big)^{\frac{2}{3}}-a(x)\right\}-C\mathcal B^{-1/18}\right)
\sum_{bnd} \int|\chi_{j}h_{1}\psi|^{2}\,dx\,.
\end{align}
We insert \eqref{eq:lb-bulk-vr}, \eqref{eq:lb-bnd-vr}
 and \eqref{eq:est-error}
into \eqref{eq:main-bulk+bnd} to obtain,
\begin{equation}\label{eq:main-bulk+bnd-D2}
\mathcal{Q}_{\mathcal{B}\Fb,-\left(\frac{\mathcal{B}}{\widehat{\sigma}}\right)^{\frac{2}{3}} a}^{D_{1}}(h_{1}\psi)\geq \left(\frac{\mathcal{B}}{\widehat{\sigma}}\right)^{\frac{2}{3}}\left(\widehat\Lambda_{1}(B_{0},a,\widehat{\sigma})
-C\mathcal B^{-1/18}\right)\,\int|h_{1}\psi|^{2}\,dx\,.
\end{equation}

Now, we will bound $\int_\Omega |(\nabla-i\mathcal B\Fb)h_{2}\psi|^2\,dx$ from below. Let $\ell_{1}<\ell$, we cover $D_{2}$ by a family of disks
$$D(\omega_{j}',\ell_{1})\subset\{x\in\R^{2}:\dist(x,\Gamma)\geq \ell_{1}\}\qquad\left(\omega_{j}'\in \overline \Omega\right)\,.$$
Consider a partition of unity satisfying
$$
\sum_{j} \chi_{j}^{2}=1\,,\qquad \sum_{j} |\nabla \chi_{j}|^{2}\leq C\,\ell_{1}^{-2}\qquad{\rm and}\qquad \supp \chi_{j}\subset D(\omega_{j}',\ell_{1})\,.$$
There holds the decomposition formula,
\begin{align}\label{eq:int-bnd}
\int_\Omega |(\nabla-i\mathcal B\Fb)h_{2}\psi|^2\,dx&=\sum_{j}\int_\Omega |(\nabla-i\mathcal B\Fb)\chi_{j}\,h_{2}\psi|^2\,dx-\sum_{j}\int_{\Omega}|\nabla \chi_{j}|^{2} |h_{2}\psi|^{2}\,dx\nonumber\\
&\geq \sum_{j}\int_\Omega |(\nabla-i\mathcal B\Fb)\chi_{j}\,h_{2}\psi|^2\,dx-C \ell_{1}^{-2}\int_{\Omega} |h_{2}\psi|^{2}\,dx\,,
\end{align} 
We observe that there exists a gauge function $\varphi_{j}$ satisfying (see \cite[Equation~(A.3)]{KA}),
$$
\left|\Fb(x)-(B_{0}(\omega_{j}')\Ab_{0}(x-\omega_{j}')+\nabla\varphi_{j})\right|\leq C\,\ell_{1}^{2} \quad{\rm in}~D(\omega_{j}',\ell_{1}')\,.
$$
Using Cauchy-Schwarz inequality, we may write,
\begin{multline*}
\int_\Omega |(\nabla-i\mathcal B\Fb)\chi_{j}\,h_{2}\psi|^2\,dx\geq \frac{1}{2}\int_\Omega |(\nabla-i\mathcal{B}\,B_{0}(\omega_{j}')\Ab_{0}(x-\omega_{j}'))e^{-i\mathcal{B}\varphi_{j}}\chi_{j}\,h_{2}\psi|^2\,dx\\
-C\,\mathcal{B}^{2}\,\ell_{1}^{4}\int_{\Omega}|\chi_{j}\,h_{2}\psi|^2\,dx\,.
\end{multline*}
We are reduced to the analysis of the Neumann realization of the Schr\"odinger operator with a constant magnetic field equal to $\mathcal{B}\,B_{0}(\omega_{j}')$ in our case.\\
Notice that by the assumption on  $h_{2}$, there exist constants $M>0$ and $\mathcal B_0>0$ such that, for all $j$,  $|B_{0}(\omega_{j}')|\geq M\,\ell$ in the support of $h_{2}$. Thus,
$$
\forall j,\quad\mathcal{B}|B_{0}(\omega_{j}')|\geq M\,\mathcal{B}\,\ell=M\mathcal B^{22/29}\gg 1\,.
$$
Moreover, the magnetic potentials $\Ab_{0}(x)$ and $\Ab_{0}(x-\omega_{j}')$ are gauge equivalent since
$$
\Ab_{0}(x-\omega_{j}')=\Ab_{0}(x)-\Ab_{0}(\omega_{j}')=\Ab_{0}(x)-\nabla(\Ab_{0}(\omega_{j}')\cdot x)\,.
$$

Thanks to Theorem~\ref{thm:FH}, there exists a constant $\mathcal{B}_{0}$ such that, for any $\mathcal{B}\geq \mathcal{B}_{0}$, we write by the min-max principle,
\begin{align}\label{eq:first-term}
\sum_{j}\int_\Omega |(\nabla-i\mathcal B\Fb)\chi_{j}\,h_{2}\psi|^2\,dx&\geq \frac{\Theta_0\mathcal{B}\,|B_{0}(\omega_{j}')|}{2}\sum_{int}\int_{\Omega}|\chi_{j}\, h_{2} \psi|^{2}\,dx-C\,\mathcal{B}^{2}\,\ell_{1}^{4}\sum_{int}\int_{\Omega}|\chi_{j}\,h_{2}\psi|^2\,dx\nonumber\\
&\geq \left(\frac{M\Theta_0}2\mathcal{B}\,\ell-C\mathcal{B}^{2}\,\ell_{1}^{4}\right)\sum_{j}\int_{\Omega}|\chi_{j}\,h_{2}\psi|^2\,dx\nonumber\\
&=\left(\frac{M\Theta_0}2\mathcal{B}\,\ell-C\mathcal{B}^{2}\,\ell_{1}^{4}\right)\int_{\Omega}|h_{2}\psi|^2\,dx\,.
\end{align}
Putting \eqref{eq:first-term}  into \eqref{eq:int-bnd}, we obtain
\begin{align}\label{eq:D2}
\mathcal{Q}_{\mathcal{B}\Fb,-\left(\frac{\mathcal{B}}{\widehat{\sigma}}\right)^{\frac{2}{3}}a}^{D_{2}}(h_{2}\psi)&=\int_\Omega |(\nabla-i\mathcal B\Fb)h_{2}\psi|^2\,dx-\left(\frac{\mathcal B}{\widehat\sigma}\right)^{2/3}\int_{\Omega} a(x)|h_{2}\psi|^2\,dx\nonumber\\
&\geq \left(\frac{M\Theta_0}2\mathcal{B}\,\ell-C\mathcal{B}^{2}\,\ell_{1}^{4}-C\ell_{1}^{-2}\right)\int_{\Omega}|h_{2}\psi|^2\,dx-\left(\frac{\mathcal B}{\widehat\sigma}\right)^{2/3}\int_{\Omega} a(x)|h_{2}\psi|^2\,dx\,.
\end{align}
We choose $\ell_1=B^{-\rho}$ and $\frac{9}{22}<\rho<\frac{11}{29}$. We observe that,
$$\mathcal{B}^{2}\,\ell_{1}^{4}=\mathcal{B}^{2-4\rho}\ll \mathcal B^{22/29}= \mathcal{B}\,\ell\,,\quad
\ell_{1}^{-2}=B^{2\rho}\ll\mathcal B\,\ell\,,\quad \mathcal{B}^{2/3}\ll \mathcal B^{22/29}=\mathcal B\,\ell\,.$$
In this way, we infer from \eqref{eq:D2}, that there exists a constant $c>0$ such that, for $\mathcal{B}$ sufficiently large,
\begin{equation}\label{eq:D2'}
\mathcal{Q}_{\mathcal{B}\Fb,-\left(\frac{\mathcal{B}}{\widehat{\sigma}}\right)^{\frac{2}{3}}a}^{D_{2}}(h_{2}\psi)\geq c\mathcal B^{22/9}\int_{\Omega} |h_{2}\psi|^2\,dx
\geq \left(\frac{\mathcal{B}}{\widehat{\sigma}}\right)^{\frac{2}{3}}\widehat\Lambda_{1}(B_{0},a,\widehat{\sigma})\int_{\Omega} |h_{2}\psi|^2\,dx\,.
\end{equation}
Collecting \eqref{eq:pu}, \eqref{eq:main-bulk+bnd-D2} and \eqref{eq:D2'}, we finish the proof of Proposition~\ref{prop:mu1-variable}.

\end{proof}

Theorem~\ref{thm:mu1-up-vr} is valid under the assumption that,
\begin{equation}\label{cond:sigma-hat}
\widehat{\lambda}_{\min}\leq \frac{H}{\kappa^2}\leq \widehat{\lambda}_{\max}\,,
\end{equation}
where $0<\widehat{\lambda}_{\min}<\widehat{\lambda}_{\max}<\infty$ are constants independent of $\kappa$ and $H$.

\begin{theorem}\label{thm:mu1-up-vr}
Let $\Omega\subset\R^2$ is an open bounded set with a smooth boundary and $\Gamma\neq\varnothing$. Suppose that \eqref{cond:sigma-hat} hold and $a\in C^1(\overline{\Omega})$, we have
$$
\mu_{1}(\kappa,H)\geq \kappa^{2}\,\widehat\Lambda_{1}\left(B_{0},a,\frac {H}{\kappa^{2}}\right)+  \mathcal{O}(\kappa^{\frac{11}{6}})\,,\qquad{\rm as}\,\kappa\to+\infty\,.
$$
Here,  $\widehat\Lambda_{1}$ is introduced in \eqref{Lambda1}.
\end{theorem}
\begin{proof}
We apply Proposition~\ref{prop:mu1-variable} with
$$
\mathcal{B}=\kappa H\,,\quad \widehat{\sigma}=\frac{H}{\kappa^2}\quad{\rm and}\quad I=[\widehat{\lambda}_{\min},\widehat{\lambda}_{\max}]\,.
$$
Let us verify that the conditions of the proposition are satisfied for this choice.
Thanks to \eqref{cond:sigma-hat}, $\widehat{\sigma}\in I$. Now, as $\kappa\to+\infty$, we have,
$$
\mathcal{B}=\widehat{\sigma}\,\kappa^{3}\to+\infty\,.
$$
This implies that, as $\kappa\to+\infty$,
$$
\mu_{1}(\kappa,H)\geq \kappa^{2}\,\widehat\Lambda_{1}\left(B_{0},a,\frac{H}{\kappa^2}\right)+\mathcal{O}(\kappa^{\frac{11}{6}})\,.
$$
This finish the proof of the theorem.
\end{proof}
\subsection{Upper bound}~\\

The next theorem is a generalization of the results in \cite{XB-KH}
and \cite{JPM} valid when the pinning term $a(\kappa,x)=a(x)$ is
independent of $\kappa$ and non-constant.

We denote by $\mu_{\mathcal B,\widehat{\sigma}}$ the lowest eigenvalue of the operator $P_{\mathcal{B}\Fb,-\left(\frac{\mathcal{B}}{\widehat{\sigma}}\right)^{\frac{2}{3}}\,a}^{\Omega}$ i.e.
$$\mu_{\mathcal B,\widehat{\sigma}}=\inf_{\psi\in H^{1}(\Omega)\setminus\{0\}}\frac{\mathcal{Q}_{\mathcal{B}\Fb,-\left(\frac{\mathcal{B}}{\widehat{\sigma}}\right)^{\frac{2}{3}}\,a}^{\Omega}(\psi)}{\|\psi\|^{2}_{L^{2}(\Omega)}}\,.$$

\begin{prop}\label{prop:up-var}
Suppose that $\Gamma\neq\varnothing$ and $\widehat\lambda_{\max}>0$. There exist positive constants $C$ and $B_{0}$ such that, for $\widehat\sigma\in (0,\widehat\lambda_{\max}]$, $a\in C^{1}(\overline{\Omega})$ and $\mathcal{B}\geq\mathcal{B}_0$, we have,
\begin{equation}\label{eq:upper-bound-variable}
\mu_{\mathcal{B},\widehat{\sigma}}\leq
\left(\frac{\mathcal{B}}{\widehat{\sigma}}\right)^{\frac{2}{3}}\Big(\widehat\Lambda_{1}(B_{0},a,\widehat{\sigma})
-C\mathcal{B}^{-\frac{1}{18}}\Big)\,.
\end{equation}
\end{prop}
\begin{proof}
Let $x_0\in\Gamma$. In \cite{XB-KH, JPM}, a quasi-mode $u(\mathcal
B,x_0;x)$ is constructed such that, ${\rm supp}\,u(\mathcal
B,x_0;\cdot)\subset \overline{\Omega}\cap B(0,\mathcal B^{-1/18})$
and,
$$\forall~\mathcal B\geq \mathcal B_0\,,\quad\frac{\displaystyle\int_\Omega|(\nabla-i\mathcal B\Fb)u(\mathcal
B,x_0;x)|^2\,dx}{\displaystyle\int_\Omega|u(\mathcal
B,x_0;x)|^2\,dx}\leq \mathcal B^{\frac{2}{3}}\Big(\Lambda(x_0)+C\mathcal
B^{-1/18}\Big)\,,$$ where $\mathcal B_0$ and $C$ are constants independent of the point $x_0$ and the parameter $\mathcal B$, and
$$\Lambda(x_0)=
\left\{
\begin{array}{ll}
\lambda_{0}\,|\nabla B_{0}(x_0)|^{\frac{2}{3}}&{\rm if~}x_0\in\Gamma\cap\Omega\,,\\
\lambda(\R^{2}_{+},\theta(x_0))\,|\nabla
B_{0}(x_0)|^{\frac{2}{3}}&{\rm if~}x_0\in\Gamma\cap\partial\Omega\,.
\end{array}
\right.
$$
Using the smoothness of the function $a(\cdot)$, we get in the
support of $u(\mathcal B,x_0;\cdot)$,
$$|a(x)-a(x_0)|\leq C\mathcal B^{-1/18}\,.$$
Thus, we deduce that,
$$\frac{\mathcal Q^\Omega_{ \mathcal B\Fb,-\left(\frac{\mathcal B}{\widehat\sigma}\right)^{\frac{2}{3}}a }(u(\mathcal B,x_0;\cdot)}{\|u(\mathcal B,x_0;\cdot)\|^2_{L^2(\Omega)}}
\leq \left(\frac{\mathcal
B}{\widehat\sigma}\right)^{\frac{2}{3}}\Big(\widehat\sigma^{\frac{2}{3}}\Lambda(x_0)-a(x_0)+C\mathcal B^{-1/18}\Big)\,.$$
Thanks to the min-max principle, we deduce that,
$$\mu_{\mathcal B,\widehat\sigma}\leq
\left(\frac{\mathcal
B}{\widehat\sigma}\right)^{\frac{2}{3}}\Big(\widehat\sigma^{\frac{2}{3}}\Lambda(x_0)-a(x_0)+C\mathcal B^{-1/18}\Big)\,.$$
 Since this is true
for all $x_0\in\Gamma$, we deduce that,
$$\mu_{\mathcal B,\widehat\sigma}\leq
\left(\frac{\mathcal
B}{\widehat\sigma}\right)^{\frac{2}{3}}\Big(\widehat\Lambda_1(B_0,a,\widehat\sigma)+C\mathcal
B^{-1/18}\Big)\,,$$ where $\widehat\Lambda_1(B_0,a,\widehat\sigma)$ is introduced in \eqref{Lambda1}.
\end{proof}
Proposition~\ref{prop:up-var} permits to obtain:
\begin{theorem}\label{thm:mu1-upp-var}
Let $\widehat{\lambda}_{\max}>0$. Suppose that $\Gamma\neq\varnothing$ and $a\in C^1(\overline{\Omega})$. There exist two constants $C_1>0$ and $\kappa_0>0$ such that,
if,
\begin{equation}\label{cond:H-m1<}
\kappa\geq\kappa_0\,,\quad  {\rm and}\quad \kappa_0\kappa^{-1}<H<\widehat{\lambda}_{\max}\kappa^2\,
\end{equation}
then
$$
\mu_{1}(\kappa,H)\leq \kappa^{2}\,\widehat\Lambda_{1}\left(B_{0},a,\frac{H}{\kappa^2}\right)+C_1\kappa^{\frac{11}{6}}\,,\qquad{\rm as}\,\kappa\to+\infty\,.
$$
\end{theorem}
\begin{proof}
To apply the results of Proposition~\ref{prop:up-var}, we take $\mathcal{B}=\kappa H$ and $\widehat{\sigma}=\frac{H}{\kappa^2}$. We see for $\kappa$ sufficiently large that $\widehat{\sigma}\in (0,\widehat{\lambda}_{\max})$ and $\mathcal{B}$ large.
\end{proof}
Theorem~\ref{thm:mu1-upp-var} is valid when $\kappa H\geq \kappa_0$ and $\kappa_0$ is sufficiently large.

\section{Proof of Theorem~\ref{thm:HC3-vr}}\label{12}
\subsection{Analysis of $\underline{H}_{C_3}^{loc}$ and $\overline{H}_{C_3}^{loc}$.}~\\
In this subsection we will prove Theorem~\ref{thm:HC3-vr} for $\underline{H}_{C_3}^{loc}$ and $\overline{H}_{C_3}^{loc}$.
We first recall some useful results from \cite{XB-KH} about the relation between the eigenvalues $\lambda_0$ and $ \lambda(\R_{+}^{2},\theta)$, introduced in \eqref{lambda0} and in \eqref{def:lambda-theta}.
\begin{thm}\label{thm:PK-R2+}~
\begin{enumerate}
\item[(i)] $\lambda(\R^{2}_{+},0)=\lambda_{0}$\,.
\item[(ii)] If $0<\theta<\pi$, then  $\lambda(\R^{2}_{+},\theta)<\lambda_{0}$.
\end{enumerate}
\end{thm}
The next proposition gives the region where $\mu_{1}(\kappa,H)<0$ that allows us to obtain an information about $\underline{H}_{C_3}^{loc}$ (see \eqref{def:HC3-u}) in the case when the magnetic field $B_0$ is constant with a pining term.
\begin{prop}\label{prop:mu1<0-var}
Suppose that $\{a>0\}\neq\varnothing$ and $\Gamma\neq\varnothing$.  There exist constants $C>0$ and $\kappa_{0}\geq 0$ such that if
\begin{equation}\label{cond:HC3-var}
\kappa\geq \kappa_{0}\,,\qquad H\leq \max\left(\sup_{x\in\Gamma\cap\Omega}\frac{a(x)^{\frac{3}{2}}}{\lambda_{0}^{\frac{3}{2}}|\nabla B_{0}(x)|},\sup_{x\in\Gamma\cap\partial\Omega}\frac{a(x)^{\frac{3}{2}}}{\lambda(\R^{2}_{+},\theta(x))^{\frac{3}{2}}|\nabla B_{0}(x)|}\right)\,\kappa^{2}-C\,\kappa^{\frac{11}{6}}\,,
\end{equation}
then,
$$
\mu_{1}(\kappa,H)<0\,.
$$
Moreover,
$$
\max\left(\sup_{x\in\Gamma\cap\Omega}\frac{a(x)^{\frac{3}{2}}}{\lambda_{0}^{\frac{3}{2}}|\nabla B_{0}(x)|},\sup_{x\in\Gamma\cap\partial\Omega}\frac{a(x)^{\frac{3}{2}}}{\lambda(\R^{2}_{+},\theta(x))^{\frac{3}{2}}|\nabla B_{0}(x)|}\right)\,\kappa^{2}- C\,\kappa^{\frac{11}{6}}\leq \underline{H}_{C_{3}}^{loc}\,.
$$
\end{prop}
\begin{proof}
We have two cases:\\
\textbf{Case 1.} Here, we suppose that,
$$\sup_{x\in\Gamma\cap\overline{\Omega}}\frac{a(x)^{\frac{3}{2}}}{\lambda_{0}^{\frac{3}{2}}|\nabla B_{0}(x)|}>\sup_{x\in\Gamma\cap\partial\Omega}\frac{a(x)^{\frac{3}{2}}}{\lambda(\R^{2}_{+},\theta(x))^{\frac{3}{2}}|\nabla B_{0}(x)|}\,.$$
Thanks to the assumption in \eqref{B(x)}, we have, for all $x\in\Gamma\cap\partial\Omega$, $0<\theta(x)<\pi$.  Theorem~\ref{thm:PK-R2+} then tells us that,
$$
\forall~x\in\Gamma\cap\partial\Omega\,,\quad \frac{a(x)^{\frac{3}{2}}}{\lambda(\R^{2}_{+},\theta(x))^{\frac{3}{2}}|\nabla B_{0}(x)|}>\frac{a(x)^{\frac{3}{2}}}{\lambda_{0}^{\frac{3}{2}}|\nabla B_{0}(x)|}\,.
$$
Thus, there exists $x_0\in\Omega\cap\Gamma$ such that (the supremum of $\frac{a(x)^{\frac{3}{2}}}{\lambda_{0}^{\frac{3}{2}}|\nabla B_{0}(x)|}$ in $\Gamma\cap\overline{\Omega}$ can not be attained on the boundary),
$$
\sup_{x\in\Gamma\cap\overline\Omega}\frac{a(x)^{\frac{3}{2}}}{\lambda_{0}^{\frac{3}{2}}|\nabla B_{0}(x)|}= \frac{a(x_0)^{\frac{3}{2}}}{\lambda_{0}^{\frac{3}{2}}|\nabla B_{0}(x_0)|} \,.
$$
If \eqref{cond:HC3-var} is satisfied for some $C>0$, then,
$$
\frac{H}{\kappa^{2}}\leq\frac{a(x_0)^{\frac{3}{2}}}{\lambda_{0}^{\frac{3}{2}}|\nabla B_{0}(x_0)|} -C\,\kappa^{-\frac{1}{6}}\,,
$$
that we can write in the form,
\begin{equation}\label{eq:appendix}
\kappa^{2}\left(\lambda_0\left(\frac{H}{\kappa^2}|\nabla B_{0}(x_0)|\right)^{\frac{2}{3}}-a(x_0)\right)\leq -C \,M\,\kappa^{\frac{11}{6}}\,,
\end{equation}
where $M>0$ is a constant independent of $C$.

Suppose that $\kappa H\geq \mathcal B_0$ where $\mathcal B_0$ is selected sufficiently large such that we can apply Theorem~\ref{thm:mu1-upp-var}. (Thanks to Lemma~\ref{lem-H=kappa}, $\mu_1(\kappa,H)<0$ when $\kappa H<\mathcal B_0$).

By Theorem~\ref{thm:mu1-upp-var}, there exist positive constants $C_{1}$ and $\kappa_{0}$ such that, for $\kappa\geq \kappa_{0}$,
\begin{align}
\mu_{1}(\kappa,H)&\leq \kappa^{2}\inf_{x\in\Gamma\cap\overline\Omega} \left(\lambda_0\left(\frac{H}{\kappa^2}|\nabla B_{0}(x)|\right)^{\frac{2}{3}}-a(x)\right)+C_{1}\,\kappa^{\frac{11}{6}}\nonumber\\
&\leq  \kappa^{2}\left(\lambda_0\left(\frac{H}{\kappa^2}|\nabla B_{0}(x_0)|\right)^{\frac{2}{3}}-a(x_0)\right)+C_{1}\,\kappa^{\frac{11}{6}}\nonumber\\
&\leq (C_1-C\,M) \,\kappa^\frac {11}{6} \,.
\end{align}
By choosing $C$ such that $C\,M> C_{1}$, we get,
$$
\mu_{1}(\kappa,H)<0\,.
$$
\textbf{Case 2.} Here, we suppose that
$$\sup_{x\in\Gamma\cap\partial\Omega}\frac{a(x)^{\frac{3}{2}}}{\lambda(\R^{2}_{+},\theta(x))^{\frac{3}{2}}|\nabla B_{0}(x)|}\geq\sup_{x\in\Gamma\cap\overline{\Omega}}\frac{a(x)^{\frac{3}{2}}}{\lambda_{0}^{\frac{3}{2}}|\nabla B_{0}(x)|}\,.$$
The assumption in  \eqref{cond:HC3-var} and the upper bound in Theorem~\ref{thm:mu1-upp-var} give us, for all $\kappa\geq \kappa_0$, $\kappa H\geq \mathcal B_0$ and $\mathcal B_0$ a sufficiently large constant,
$$
\mu_{1}(\kappa,H)\leq (C_1-C\,\widetilde M)\,\kappa^\frac{11}{6}\,.
$$
where $\widetilde M>0$ is a constant independent of $C$.
By choosing $C$ such that $C\,\widetilde M> C_{1}$, we get,
$$
\mu_{1}(\kappa,H)<0\,.
$$
This finishes the proof of the proposition.
\end{proof}

The next proposition gives us a lower bound of $\overline{H}_{C_3}^{loc}$ (see \eqref{def:HC3-u}). This is obtained by localizing the region where $\mu_{1}(\kappa,H)>0$ holds.

\begin{prop}\label{prop:mu>0-var}
Suppose that $\{a>0\}\neq\varnothing$, $\widehat\lambda_{\max} >0$ and $\Gamma=\varnothing$. There exist constants $C>0$ and $\kappa_{0} > 0$ such that if
\begin{equation}\label{cond:HC3-2-var}
\begin{aligned}
\kappa\geq \kappa_{0}\,,\qquad \widehat\lambda_{\max}\,\kappa&\geq H\\
&>\max\left(\sup_{x\in\Gamma\cap\overline{\Omega}}\frac{a(x)^{\frac{3}{2}}}{\lambda_{0}^{\frac{3}{2}}|\nabla B_{0}(x)|},\sup_{x\in\Gamma\cap\partial\Omega}\frac{a(x)^{\frac{3}{2}}}{\lambda(\R^{2}_{+},\theta(x))^{\frac{3}{2}}|\nabla B_{0}(x)|}\right)\,\kappa^{2}+C\,\kappa^{\frac{11}{6}}\,,
\end{aligned}
\end{equation}
then,
$$
\mu_{1}(\kappa,H)>0\,.
$$
 Moreover,
$$
\overline{H}_{C_{3}}^{loc}\leq\max\left(\sup_{x\in\Gamma\cap\overline{\Omega}}\frac{a(x)^{\frac{3}{2}}}{\lambda_{0}^{\frac{3}{2}}|\nabla B_{0}(x)|},\sup_{x\in\Gamma\cap\partial\Omega}\frac{a(x)^{\frac{3}{2}}}{\lambda(\R^{2}_{+},\theta(x))^{\frac{3}{2}}|\nabla B_{0}(x)|}\right)\,\kappa^{2}+C\,\kappa^{\frac{11}{6}}\,.
$$
\end{prop}
\begin{proof}
 Having in mind the definition of $\widehat\Lambda_1$ in \eqref{Lambda1}, under the assumption in \eqref{cond:HC3-2-var}, we get for $\kappa$ large enough,
 \begin{equation}\label{eq:appendix2}
\widehat\Lambda_{1}\left(B_0,a,\frac{H}{\kappa^2}\right)\geq C\,M\,\kappa^{-\frac{1}{6}}\,,
 \end{equation}
where $M>0$ is a constant independent of the constant $C$.

Thanks to Theorem~\ref{thm:mu1-up-vr}, we get the existence of positive constants $C'$ and $\kappa_0$ such that, for $\kappa\geq\kappa_0$,
$$
\mu_{1}(\kappa,H)\geq (C\,M-C')\,\kappa^{\frac{11}{6}}
$$
To finish the proof, we choose $C$ sufficiently large such that $C\,M>C'$.
\end{proof}

\subsection{Analysis of $\underline{H}_{C_3}^{cp}$ and $\overline{H}_{C_3}^{cp}$.}~\\

 Proposition~\ref{prop:est-psi-var} below is an adaptation of an analogous result obtained in \cite{HK} for the functional in \eqref{eq-2D-GLf} with a constant pinning term.  Proposition~\ref{prop:est-psi-var} is valid when $\Gamma\not=\emptyset$.
 Proposition~\ref{prop:est-psi-var} says that, if $(\psi,\Ab)$ is a critical point of the functional  in \eqref{eq-2D-GLf} and  $H$ is of order $\kappa^2$, then $|\psi|$ is concentrated near the set $\Gamma$.

 \begin{prop}\label{prop:est-psi-var}
Let $\varepsilon>0$. There exist two positive constants $C$ and $\kappa_0$ such that, if
\begin{equation}\label{cond:H>kappa}
\kappa\geq\kappa_0\,,\quad H\geq\varepsilon\,\kappa^{2}\,,
\end{equation}
and $(\psi,\Ab)$ is a solution of \eqref{eq-2D-GLeq}, then
 \begin{equation}\label{l3est}
\|\psi\|^{2}_{L^2(\Omega)}\leq C\,\kappa^{-\frac{1}{4}}\|\psi\|^{2}_{L^4(\Omega)}\,.
\end{equation}
\end{prop}
\begin{proof}
Let $\lambda=\kappa^{-\frac 12}$ and $\Omega_{\lambda}=\{x\in\Omega:\dist(x,\partial\Omega)>\lambda~\&~\dist(x,\Gamma)>\lambda\}$. We introduce a function $h\in C^{\infty}_{c}(\Omega)$ satisfying
$$
0\leq h \leq 1~{\rm in}~\Omega\,,\quad h=1~{\rm in}~\Omega_{\lambda}\,,\quad {\rm supp}\,h\subset\Omega_{\lambda/2}\,,
$$
and
$$
|\nabla h |\leq \frac{C}{\lambda}\quad{\rm in}~\Omega\,,
$$
where $C$ is a positive constant.\\
Using \eqref{3d-<}, we can prove that (see the detailed proof in
\cite[Eq.~(6.6)]{HK} when $a$ is constant),
$$
 \kappa\,H\int_\Omega|B_{0}(x)|\,|h\psi|^2\,dx-c\,\kappa\,\|\psi\|_{L^{2}(\Omega)}\|h\psi\|_{L^{4}(\Omega)}^{2}\leq \int_{\Omega}|(\nabla-i\kappa H\Ab)h\psi|^2\,dx\,.
$$
 Now, the Cauchy-Schwarz inequality yields,
$$
c\,\kappa\,\|\psi\|_{L^{2}(\Omega)}\|h\psi\|_{L^{4}(\Omega)}^{2}\leq c^{2}\|\psi\|_{L^{2}(\Omega)}^{2}+\kappa^2\|h\psi\|_{L^{4}(\Omega)}^{4}\,,
$$
which implies that
\begin{multline*}
\int_\Omega\,\left(\kappa\,H\,|B_{0}(x)|-\kappa^{2}\,a(x)\right)\,|h\psi|^2\,dx\leq\int_{\Omega}|(\nabla-i\kappa H\Ab)h\psi|^2\,dx-\kappa^{2}\int_{\Omega}\,a(x)\,|h\psi|^2\,dx\\
+c^{2}\|\psi\|_{L^{2}(\Omega)}^{2}+\kappa^{2}\|h\psi\|_{L^{4}(\Omega)}^{4}\,.
\end{multline*}
We may use a  localization formula as the one in  \eqref{eq:2} (but with $\chi_\kappa=h$) to write,
\begin{align*}
\int_\Omega\,\left(\kappa\,H\,|B_{0}(x)|-\kappa^{2}\,a(x)\right)\,|h\psi|^2\,dx&\leq c^{2}\int_{\Omega}|\psi|^{2}\,dx+\int_{\Omega}|\nabla h|^{2}|\psi|^{2}\,dx+\kappa^{2}\int_{\Omega}(h^{4}-h^{2})|\psi|^{4}\,dx\\
&\leq c^{2}\int_{\Omega}|\psi|^{2}\,dx+\int_{\Omega}|\nabla h|^{2}|\psi|^{2}\,dx\,.
\end{align*}
Here, we have used the fact that $h^{4}\leq h^{2}$ since $0\leq h\leq 1$.

By assumption \eqref{B(x)}, $|\nabla B_{0}|$ does not vanish on $\Gamma$, hence
\begin{equation}
|B_0(x)|\geq \frac{1}{M}\,\kappa^{-\frac 12}\qquad{\rm in}\quad\{\dist(x,\Gamma)\geq \lambda\}\,,
\end{equation}
for some constant $M>0$.\\
Thus, by using \eqref{def:sup-a} and \eqref{cond:H>kappa}, we get,
$$
\left(\frac{\varepsilon}{M}\,\kappa^{\frac{5}{2}}-\kappa^{2}\,\overline{a}\right)\int_\Omega|h\psi|^2\,dx\leq c^{2}\int_{\Omega}|\psi|^{2}\,dx+\int_{\Omega}|\nabla h|^{2}|\psi|^{2}\,dx\,.
$$
Writing $\displaystyle\int_{\Omega}|\psi|^{2}\,dx=\int_{\Omega}|h\psi|^{2}\,dx+\int_{\Omega}(1-h^{2})|\psi|^{2}\,dx$ and using the assumption on $h$, we have,
$$
\left(\frac{\varepsilon}{M}\,\kappa^{\frac{5}{2}}-\kappa^{2}\,\overline{a}-c^{2}\right)\int_{\Omega}|h\psi(x)|^{2}\,dx\leq (c^{2}+C\,\kappa)\int_{\Omega\setminus\Omega_{\lambda}}|\psi|^{2}\,dx\,.
$$
For $\kappa$ large enough, $\frac{\varepsilon}{M}\,\kappa^{\frac{5}{2}}-\kappa^{2}\,\overline{a}-c^{2}\geq \frac{\varepsilon}{2M}\,\kappa^{\frac{5}{2}}$ and
$$
\int_\Omega| h\psi(x)|^{2}\,dx\leq 2\frac{M}{\varepsilon}C\,\kappa^{-\frac{3}{2}} \int_{\Omega\setminus\Omega_{\lambda}}|\psi|^{2}\,dx\,.
$$
Thanks to the assumption on the support of $h$, we get further,
$$
\int_\Omega|\psi(x)|^{2}\,dx\leq \left(2\frac{M}{\varepsilon}C\,\kappa^{-\frac{3}{2}}+1\right)\int_{\Omega\setminus\Omega_{\lambda}}|\psi|^{2}\,dx\,.
$$
Recall that $\lambda=\kappa^{-\frac{1}{2}}$. The Cauchy Schwarz inequality yields,
$$
\int_{\Omega\setminus\Omega_\lambda}|\psi(x)|^{2}\,dx\leq |\Omega\setminus\Omega_\lambda|^{1/2} \left(\int_{\Omega\setminus\Omega_\lambda} |\psi|^{4}\,dx\right)^{\frac{1}{2}}
\leq C\,\kappa^{-\frac{1}{4}}\left(\int_\Omega |\psi|^{4}\,dx\right)^{\frac{1}{2}}\,.
$$
This finishes the proof of the proposition.
\end{proof}

Now, we can  give an upper bound of the critical field $\overline{H}_{C_3}^{cp}$ in the case when $\Gamma\neq\varnothing$ and with a pining term.
 \begin{theorem}\label{thm:lb-H-var}
Supose that $\Gamma\neq\varnothing$ and $\{a>0\}\neq\varnothing$. There exists $C>0$ and $\kappa_0$ such that, if $H$ satisfies
\begin{equation}\label{cond:HC3-2w-var}
 H>\max\left(\sup_{x\in\Gamma\cap\overline{\Omega}}\frac{a(x)^{\frac{3}{2}}}{\lambda_{0}^{\frac{3}{2}}|\nabla B_{0}(x)|},\sup_{x\in\Gamma\cap\partial\Omega}\frac{a(x)^{\frac{3}{2}}}{\lambda(\R^{2}_{+},\theta(x))^{\frac{3}{2}}|\nabla B_{0}(x)|}\right)\,\kappa^{2}+C\,\kappa^{\frac{11}{6}}\,,
\end{equation} then $(0,\Fb)$ is the unique solution to \eqref{eq-2D-GLeq}.\\
 Moreover,
$$
\overline{H}_{C_3}^{cp}\leq\max\left(\sup_{x\in\Gamma\cap\overline{\Omega}}\frac{a(x)^{\frac{3}{2}}}{\lambda_{0}^{\frac{3}{2}}|\nabla B_{0}(x)|},\sup_{x\in\Gamma\cap\partial\Omega}\frac{a(x)^{\frac{3}{2}}}{\lambda(\R^{2}_{+},\theta(x))^{\frac{3}{2}}|\nabla B_{0}(x)|}\right)\,\kappa^{2}+C\,\kappa^{\frac{11}{6}}\,.
$$
\end{theorem}

\begin{proof}
In light of the result in   Theorem~\ref{thm:GP}, we may assume the extra condition that $H\leq \lambda_{\max}\kappa^2$ for a sufficiently large constant $\lambda_{\max}$.

We take the constant $C$ in \eqref{cond:HC3-2w-var} as in Proposition~\ref{prop:mu>0-var}. In that way, under the assumption in \eqref{cond:HC3-2w-var}, we have
\begin{equation}\label{eq:prop:mu>0-var''}
\mu_1(\kappa,H)<0\,.
\end{equation}
Suppose now that  $(\psi,\Ab)$ is a solution of \eqref{eq-2D-GLeq} with $\psi\neq 0$. Similarly, as in the proof of Theorem~\ref{thm:lb-H}, we have,
\begin{equation}\label{est:top1-var}
-\top\geq \mu_{1}(\kappa,H)\,\|\psi\|_{L^{2}(\Omega)}^{2}-C\,\sqrt{\top}\,\kappa\,\|\psi\|_{L^{2}(\Omega)}^{2}\,,
\end{equation}
where $\top=\kappa^2\|\psi\|^4_{L^4(\Omega)}$ is introduced in \eqref{l5est}.

To apply the result of Proposition~\ref{prop:est-psi-var}, we take
$$\varepsilon=\frac{1}{2}\max\left(\sup_{x\in\Gamma\cap\overline{\Omega}}\frac{a(x)^{\frac{3}{2}}}{\lambda_{0}^{\frac{3}{2}}|\nabla B_{0}(x)|},\sup_{x\in\Gamma\cap\partial\Omega}\frac{a(x)^{\frac{3}{2}}}{\lambda(\R^{2}_{+},\theta(x))^{\frac{3}{2}}|\nabla B_{0}(x)|}\right)\,,
$$
and get,
\begin{equation}\label{est:psi1-var}
\|\psi\|^{2}_{L^{2}(\Omega)}\leq C\,\kappa^{-\frac14}\|\psi\|^2_{L^4(\Omega)}= C\kappa^{-\frac{5}{4}}\,\sqrt{\top}\,.
\end{equation}
Putting \eqref{est:psi1-var} into \eqref{est:top1-var}, we obtain,
$$
-\top\geq \mu_{1}(\kappa,H)\,\|\psi\|_{L^{2}(\Omega)}^{2}-C'\,\kappa^{-\frac{1}{4}}\,\top\,.
$$
We conclude that, for $\kappa\geq\kappa_0$ and $\kappa_0$ a sufficiently large constant, $\mu_{1}(\kappa,H)<0$, which is in contradiction with \eqref{eq:prop:mu>0-var''}. Therefore, we conclude that $\psi=0$.
\end{proof}

Following the argument given in Proposition~\ref{prop:cp}, we get:

\begin{prop}\label{prop:cp-var}
Supose that $\Gamma\neq\varnothing$ and $\{a>0\}\neq\varnothing$. There exists $C>0$ and $\kappa_0$ such that, if $\kappa\geq\kappa_0$ and $H$ satisfies
\begin{equation}\label{cond:HC3-2w2-var}
 H\leq \max\left(\sup_{x\in\Gamma\cap\overline{\Omega}}\frac{a(x)^{\frac{3}{2}}}{\lambda_{0}^{\frac{3}{2}}|\nabla B_{0}(x)|},\sup_{x\in\Gamma\cap\partial\Omega}\frac{a(x)^{\frac{3}{2}}}{\lambda(\R^{2}_{+},\theta(x))^{\frac{3}{2}}|\nabla B_{0}(x)|}\right)\,\kappa^{2}-C\,\kappa^{\frac{11}{6}}\,,
\end{equation}
then there exists a solution $(\psi,\Ab)$ of \eqref{eq-2D-GLeq} with $\|\psi\|_{L^{2}(\Omega)}\neq 0$.\\
 Moreover,
$$
\max\left(\sup_{x\in\Gamma\cap\overline{\Omega}}\frac{a(x)^{\frac{3}{2}}}{\lambda_{0}^{\frac{3}{2}}|\nabla B_{0}(x)|},\sup_{x\in\Gamma\cap\partial\Omega}\frac{a(x)^{\frac{3}{2}}}{\lambda(\R^{2}_{+},\theta(x))^{\frac{3}{2}}|\nabla B_{0}(x)|}\right)\,\kappa^{2}-C\,\kappa^{\frac{11}{6}}\leq \underline{H}_{C_3}^{cp}\,.
$$
\end{prop}

 \subsection*{End of the proof of Theorem~\ref{thm:HC3-vr}}
 All the critical fields are contained in the interval $[\underline{H}_{C_{3}}^{loc},\overline{H}_{C_{3}}^{cp}]$ (the proof of this statement  is exactly as the one given for \eqref{eq:ov} and \eqref{eq:un}).\\
By Proposition~\ref{prop:mu1<0-var} and Theorem~\ref{thm:lb-H-var}, we get the existence of positive constants $C$ and $\kappa_0$, such that for $\kappa\geq\kappa_0$,
\begin{multline}
\max\left(\sup_{x\in\Gamma\cap\overline{\Omega}}\frac{a(x)^{\frac{3}{2}}}{\lambda_{0}^{\frac{3}{2}}|\nabla B_{0}(x)|},\sup_{x\in\Gamma\cap\partial\Omega}\frac{a(x)^{\frac{3}{2}}}{\lambda(\R^{2}_{+},\theta(x))^{\frac{3}{2}}|\nabla B_{0}(x)|}\right)\,\kappa^{2}-C\,\kappa^{\frac{11}{6}}\leq \underline{H}_{C_{3}}^{loc}\leq\overline{H}_{C_{3}}^{cp}\\
\leq\max\left(\sup_{x\in\Gamma\cap\overline{\Omega}}\frac{a(x)^{\frac{3}{2}}}{\lambda_{0}^{\frac{3}{2}}|\nabla B_{0}(x)|},\sup_{x\in\Gamma\cap\partial\Omega}\frac{a(x)^{\frac{3}{2}}}{\lambda(\R^{2}_{+},\theta(x))^{\frac{3}{2}}|\nabla B_{0}(x)|}\right)\,\kappa^{2}+C\,\kappa^{\frac{11}{6}}\,.
\end{multline}
As a consequence, we have proved that the asymptotics in Theorem~\ref{thm:HC3-vr} is valid for  for the six critical fields in \eqref{def:HC3-o}, \eqref{def:HC3} and \eqref{def:HC3-u}.

%
%
%

\section*{Acknowledgements}
This work is partially supported by a grant from Lebanese University
and a grant of Universit\'e Paris-Sud. I would like to thank my
supervisors \textit{B.Helffer} and \textit{A.Kachmar} for their
support, and \textit{J.P. Miqueu} for the communication of the
preprint \cite{JPM}.

\end{document}